\documentclass[11pt]{article}
\usepackage{fullpage}
\usepackage{amsfonts}
\usepackage{bbding}
\usepackage{graphicx,float}
\usepackage{amsfonts,amssymb,amsmath,amsthm,amscd}
\usepackage{mathrsfs}
\usepackage{xcolor}
\usepackage{empheq}
\usepackage{adforn}
\usepackage{cancel}
\usepackage{xr-hyper}
\usepackage{xr}
\usepackage{mdframed}
\usepackage{mathdots}
\usepackage{enumerate}
\usepackage{lmodern}
\usepackage{mdframed}
\usepackage{stmaryrd}
\usepackage{blindtext}
\usepackage[all, knot]{xy}
\usepackage{leftidx}
\usepackage{accents}
\usepackage{appendix}
\usepackage{cases}
\usepackage{bbm}
\usepackage{subfigure}

\definecolor{slightblue}{rgb}{.8, .8, 1}
\definecolor{hair}{RGB}{100,225,190}
\definecolor{ruby}{RGB}{220,50,120}
\definecolor{grass}{RGB}{150,220,110}
\definecolor{ceruleanblue}{rgb}{0.16, 0.32, 0.75}
\definecolor{deepcarmine}{rgb}{0.66, 0.13, 0.24}
\definecolor{otterbrown}{rgb}{0.4, 0.26, 0.13}
\definecolor{sapphire}{rgb}{0.03, 0.15, 0.4}

\usepackage[colorlinks=true,  linkcolor=otterbrown, citecolor=sapphire,  urlcolor=hair!80!black]{hyperref}

\usepackage[T1]{fontenc}

\newtheorem{theorem}{Theorem}[section] \newtheorem{lemma}[theorem]{Lemma}
\newtheorem{proposition}[theorem]{Proposition}

\theoremstyle{definition} 
\newtheorem{definition}[theorem]{Definition}

\newtheorem{remark}[theorem]{Remark} \numberwithin{equation}{section}
\numberwithin{figure}{section}

\newcommand{\Cb}{\mathbb{C}}

\newcommand{\Eb}{\mathbb{E}}

\newcommand{\Hb}{\mathbb{H}}

\newcommand{\Nb}{\mathbb{N}}

\newcommand{\Pb}{\mathbb{P}}

\newcommand{\Rb}{\mathbb{R}}

\newcommand{\Zb}{\mathbb{Z}}

\newcommand{\Ac}{\mathcal{A}}
\newcommand{\Bc}{\mathcal{B}}
\newcommand{\Cc}{\mathcal{C}}
\newcommand{\Dc}{\mathcal{D}}
\newcommand{\Ec}{\mathcal{E}}
\newcommand{\Fc}{\mathcal{F}}

\newcommand{\Kc}{\mathcal{K}}
\newcommand{\Lc}{\mathcal{L}}

\newcommand{\Uc}{\mathcal{U}}

\usepackage[section]{placeins}
\usepackage{wrapfig}
\usepackage{lpic}

\newcommand{\wt}{\widetilde}

\DeclareFontFamily{OMX}{yhex}{}
\DeclareFontShape{OMX}{yhex}{m}{n}{<->yhcmex10}{}
\DeclareSymbolFont{yhlargesymbols}{OMX}{yhex}{m}{n}
\DeclareMathAccent{\wideparen}{\mathord}{yhlargesymbols}{"F3}

\makeatletter
\newcommand*\rel@kern[1]{\kern#1\dimexpr\macc@kerna}
\newcommand*\wb[1]{%
	\begingroup
	\def\mathaccent##1##2{%
		\rel@kern{0.8}%
		\overline{\rel@kern{-0.8}\macc@nucleus\rel@kern{0.2}}%
		\rel@kern{-0.2}%
	}%
	\macc@depth\@ne
	\let\math@bgroup\@empty \let\math@egroup\macc@set@skewchar
	\mathsurround\z@ \frozen@everymath{\mathgroup\macc@group\relax}%
	\macc@set@skewchar\relax
	\let\mathaccentV\macc@nested@a
	\macc@nested@a\relax111{#1}%
	\endgroup
}
\makeatother

\def \eps {\varepsilon}

\newcommand{\dist}{\mathrm{dist}}
\newcommand{\diam}{\mathrm{diam}}
\newcommand{\ind}{\mathbf{1}}
\newcommand{\loc}{\mathrm{loc}}
\newcommand{\exc}{\mathrm{exc}}

\newcommand{\CLE}{\mathrm{CLE}}

\newcommand{\Fill}{\mathrm{Fill}}
\newcommand{\cl}{\mathrm{CL}}

\newcommand{\half}{1/2}
\usepackage{soul}

\newcommand{\thmGFF}{1.1} 
\newcommand{\thmGFFcarpet}{1.10} 

\newcommand{\newxi}{\xi^+}
\newcommand{\neweta}{\eta^+}

\title{Percolation of discrete GFF in dimension two\\ I. Arm events in the random walk loop soup}
\author{Yifan Gao \and Pierre Nolin \and Wei Qian}
\date{City University of Hong Kong}
\begin{document}

\maketitle
	
\begin{abstract}
In this work, which is the first part of a series of two papers, we study the random walk loop soup in dimension two. More specifically, we estimate the probability that two large connected components of loops come close to each other, in the subcritical and critical regimes. The associated four-arm event can be estimated in terms of exponents computed in the Brownian loop soup, relying on the connection between this continuous process and conformal loop ensembles (with parameter $\kappa \in (8/3,4]$).

Along the way, we need to develop several useful tools for the loop soup, based on separation for random walks and surgery for loops, such as a ``locality'' property and quasi-multiplicativity. The results established here then play a key role in a second paper, in particular to study the connectivity properties of level sets in the random walk loop soup and in the discrete Gaussian free field.

\bigskip

\end{abstract}

\tableofcontents

\section{Introduction} \label{sec:intro}

The central object of study in this work is the random walk loop soup (RWLS), in dimension two. More precisely, we consider the RWLS on the square lattice $\Zb^2$, as well as on subsets of it, for any intensity $\alpha$ at most $\half$. This range of values for $\alpha$ encompasses the so-called subcritical ($\alpha <\half$) and critical ($\alpha = \half$) regime for the corresponding continuum object, known as the Brownian loop soup.

The RWLS is defined as a random collection of finite loops on $\Zb^2$, and such loops can be grouped into connected components, or clusters, in a natural way. We aim to describe the structure of macroscopic clusters in a finite, large, domain $D \subseteq \Zb^2$, and in particular the contact points between them. In other words, how they are coming close to each other, so that they are almost connected. These geometric considerations will then be exploited in the companion paper \cite{GNQ2024b} to establish connectivity properties for the occupation time field of the RWLS, and, through Le Jan's isomorphism  \cite{LeJa2010}, two-sided level sets of the discrete Gaussian free field (where the field remains, in absolute value, below a given level).

\subsection{Background: random walk loop soup}

Consider the square lattice $\Zb^2$, whose vertices are the points in the plane with integer coordinates, and with edges connecting any two vertices which lie at a Euclidean distance $1$. The random walk loop soup $\Lc$ on it was constructed in \cite{LTF2007} by Lawler and Trujillo Ferreras, as a Poisson process of random walk loops. It is produced by the measure $\alpha \nu$, where the parameter $\alpha > 0$ is called the intensity, for a well-chosen measure $\nu$ on random walk loops (see \eqref{eq:nu} below for a precise definition).

The RWLS has a counterpart in the continuum, called the Brownian loop soup (BLS), which was introduced by Lawler and Werner in \cite{LW2004}. This process is a natural random collection of loops in $\Rb^2$ displaying conformal invariance.
The BLS plays a key role in the construction of the conformal loop ensemble (CLE) process with parameter $\kappa \in (8/3,4]$ by Sheffield and Werner \cite{SW2012}. Introduced in \cite{MR2494457}, CLE itself arises (or is conjectured to arise) as the scaling limit of many statistical mechanics models on lattices. The BLS is a random collection $\wt \Lc$ of Brownian (i.e. continuous) loops, obtained as a Poisson process corresponding to $\alpha \wt\nu$, for some canonical measure $\wt\nu$ on Brownian loops in $\Rb^2$ (we refer to \eqref{eq:nu_D} for its precise definition). We also consider the RWLS and the BLS in subsets $D$ of, respectively, $\Zb^2$ and $\Rb^2$, by only keeping the loops which remain entirely in $D$. We denote the corresponding processes in $D$ by $\Lc_D$ and $\wt \Lc_D$, respectively.

For both the RWLS and the BLS (in $\Zb^2$ and $\Rb^2$, or in subsets of them), we can introduce a notion of connectivity for loops, as follows. Two loops $\gamma$ and $\gamma'$ are considered as neighbors if they overlap (their images have a non-empty intersection), and more generally, they are connected if we can go from one to the other by ``jumping'' a finite number of times between neighboring loops. This produces connected components of loops, also called clusters, which form a partition of all loops in the collection.

The existence of a phase transition at $\alpha = \half$ for the BLS, and some of its properties, were established by Sheffield and Werner in their seminal work \cite{SW2012}, using the link between the BLS and CLE. 
More precisely, in any given bounded domain $D$, for all values $\alpha < \half$, which corresponds to the subcritical regime, there are infinitely many connected components of loops. This remains true at $\alpha = \half$ (critical case), although at this value, the clusters ``almost touch''. But for $\alpha > \half$ (supercritical regime), the picture becomes totally different, with only one cluster of loops, which is dense in $D$.

In a separate paper \cite{GNQ2024c}, we extract arm exponents for the BLS from explicit computations with CLE. In the present paper, we use these exponents as a key input, and we explain how to transfer them to the discrete setting.
In the aforementioned paper \cite{LTF2007} where the RWLS was introduced, a version of convergence of the RWLS to the BLS was shown. Later, works by van de Brug, Camia and Lis \cite{BCL2016} and Lupu  \cite{Lu2019} proved that the clusters in the RWLS converge in a certain sense to CLE. We emphasize however that the convergence established in \cite{BCL2016, Lu2019} is qualitative, and to the best of our knowledge, no precise rate is known for this convergence. On the other hand, since the arm exponents describe the rate of decay of rare arm events, this qualitative statement of convergence would not allow one to transfer the continuum exponents to the discrete ones.
The goal of the current paper is to develop a toolbox to study arm events of the RWLS which, among other things, will allow us to
prove that the arm exponents for the RWLS indeed coincide with the continuum ones.

\subsection{Motivation: relation to the companion paper \cite{GNQ2024b}}
We can associate with the RWLS an occupation field in a natural way, which measures the total amount of time that each vertex is visited by all loops (i.e. the combined contribution of all loops in $\Lc_D$). This is a natural and important object, especially because at the critical intensity $\alpha = \half$, it has the same distribution as ($1/2$ of) the square of the (discrete) Gaussian free field (GFF) in $D$, from Le Jan's isomorphism \cite{LeJa2010} (see also the lecture notes \cite{MR2815763}, and the references therein). Consequently, valuable geometric information derived for the RWLS at criticality, about loops and clusters of loops, can then be translated into properties of the GFF.

In fact, this work on the RWLS, as well as \cite{GNQ2024b}, originated from questions about the connectivity of the two-sided level sets of the GFF in, e.g., $B_N := [-N,N]^2 \cap \Zb^2$ (denote it by $(\varphi_N(v))_{v \in B_N}$), which are sets of the form
$$\{v \in B_N : -\lambda \leq \varphi_N(v) \leq \lambda\},$$
where we may allow the threshold $\lambda > 0$ to depend on $N$. We then extended our analysis to the RWLS in the whole regime $\alpha \in (0,\half]$. However, we are not discussing these questions explicitly in the present paper, since they are the focus of the second part \cite{GNQ2024b}. Let us just mention briefly one of the results that we establish there, by playing with loops. Consider the GFF $(\varphi_N(v))_{v \in B_N}$ in $B_N$ with Dirichlet boundary conditions, i.e. value $0$ outside of $B_N$: for all $\eps > 0$,
\begin{equation} \label{eq:result_GFF}
\Pb \big( \Cc_N( \{v \in B_N : - \eps \sqrt{\log N} \leq \varphi_N(v) \leq \eps \sqrt{\log N} \} ) \big) \longrightarrow 1 \quad \text{as $N \to \infty$},
\end{equation}
where $\Cc_N$ refers to the existence of a horizontal crossing (left-to-right path) in the box $B_N$ (we actually prove more precise properties, see Theorems~\thmGFF{} and~\thmGFFcarpet{} in \cite{GNQ2024b}). Note that the maximum of the GFF in $B_N$ is of order $\log N$ \cite{BDG2001}, and its average value is of order $\sqrt{\log N}$. Moreover, it is not difficult to prove that the two-sided level set appearing in the above probability has a vanishing density as $\eps \searrow 0$.

For our proofs, a specific power-law upper bound plays a central role, on the event that two large clusters of loops (in the RWLS) approach close to each other. In the present paper, we indeed establish this bound, with an exponent $\xi(\alpha)$ that we identify (as explained earlier) with an exponent for the BLS computed in \cite{GNQ2024c}. We then use this estimate to analyze the RWLS, and prove results such as \eqref{eq:result_GFF} above. This is achieved in \cite{GNQ2024b} through a summation technique from \cite{BN2022}, itself inspired by \cite{KMS2015}. These arguments are based in principle on the fact that the above exponent for two clusters is strictly larger than $2$ (or, more generally, the dimension of the ambient space), which is true for all $\alpha < \half$. However, we were able to adapt them in order to handle the critical regime $\alpha = \half$, where the exponent is $2$ exactly.

\subsection{Main results: toolbox for arm events}

We now discuss some of the main properties of arm events that we establish in this paper. We are not at all aiming to be exhaustive, and at this stage, the statements that we provide are necessarily rather vague, and also weaker than the actual results. We refer the reader to the corresponding results in Sections~\ref{sec:sep}, \ref{sec:arm-lwrs} and \ref{sec:quasi-mult} for precise statements.

\begin{itemize}
\item We first derive a \emph{separation property} for a pair of random walks inside a RWLS. This technical result can then easily be generalized to packets of random walks, and it happens to be instrumental to perform surgery on random walk loops. For example, if we have two loops in two distinct clusters, cutting them along the boundary of a box produces two packets of two random walks (each) in the loop soup. This separation property is often instrumental for statistical mechanics models on lattices, and it has already been established for various processes, e.g.\ famously for near-critical Bernoulli percolation in two dimensions \cite{Ke1987a}. It turns out to be a basic building block in all our later reasonings.
\end{itemize}

We then combine repeatedly our separation result with a-priori estimates for random walk measures, in order to establish several useful properties of arm events, through suitable constructions with loops. For $1 \leq l < d$, for any subset of vertices $D\subseteq \Zb^2$ containing $B_d$, we denote by $\Ac_D(l,d)$ the \emph{four-arm event} that in the RWLS in $D$ with intensity $\alpha$, there exist two outermost clusters crossing the annulus $A_{l,d} := B_d \setminus B_{l-1}$, i.e. each of these clusters connects the boundaries of $B_d$ and $B_l$. For concreteness, we focus on this arm event, which is the one appearing in chains of clusters in the companion paper \cite{GNQ2024b}, even though the same properties would hold for a larger number of crossing clusters.

\begin{itemize}
\item \emph{Locality}. If $D \supseteq B_{2d}$, we can replace $\Ac_D(l,d)$ by $\Ac_{\loc}(l,d) := \Ac_{B_{2d}}(l,d)$. We call it the ``local'' arm event, involving only the loops in $B_{2d}$ (instead of the whole of $D$). We have the following (see Proposition~\ref{prop:locality} for a slightly stronger statement).
\begin{proposition}
There exists a universal constant $C>0$ such that for any intensity $\alpha\in (0,\half]$, all $1 \leq l\le d/2$, and all subsets $D \supseteq B_{2d}$ of $\Zb^2$,
$$\Pb(\Ac_D(l,d)) \le C \, \Pb(\Ac_{\loc}(l,d)).$$
\end{proposition}
This estimate allows us to bound four-arm events in terms of ``localized'' versions of them, which makes appear some very useful independence. In particular, it can be used to establish a quasi-multiplicativity property, analogous to the one used in 2D critical percolation to derive arm exponents.

\item \emph{Quasi-multiplicativity}. We prove the following upper bound (see also Proposition~\ref{lem:quasi} below).
\begin{proposition}
	For any $\alpha\in (0,\half]$, there exists $c_1=c_1(\alpha)>0$ such that for all $1 \leq d_1 \le d_2/2\le d_3/16$, and all subsets $D \supseteq B_{2d_3}$ of $\Zb^2$,
$$\Pb( \Ac_{D}(d_1,d_3) ) \le c_1\, \Pb( \Ac_{\loc}(d_1,d_2) )\, \Pb( \Ac_{D}(4d_2,d_3) ).$$
\end{proposition}
This is a key property to connect discrete arm events to the corresponding arm events in the Brownian loop soup, and thus to derive an upper bound (Theorem~\ref{prop:armp}) for the discrete event, with the same exponent $\xi(\alpha)$ (computed directly on CLE in \cite{GNQ2024c}):
\begin{theorem}
	For all $\alpha\in(0,1/2]$ and all $\eps>0$, there exists $c_2=c_2(\alpha,\eps)>0$ such that for all $1 \leq d_1<d_2$, and all subsets $D \supseteq B_{2d}$ of $\Zb^2$,
	\begin{equation}
	\Pb( \Ac_D(d_1,d_2) ) \le c_2\, (d_2/d_1)^{-\xi(\alpha)+\eps}.
	\end{equation}
\end{theorem}
\end{itemize}
We also obtain similar locality and quasi-multicativity results for three other types of arm events, namely, interior two-arm events, boundary two-arm events, and boundary four-arm events. The corresponding results are stated in Proposition~\ref{prop:2n-locality}, Proposition~\ref{prop:2n-quasi}, and Theorem~\ref{prop:2n-armp}, respectively.

\subsection{Structure of the paper}

Our paper is organized as follows. In Section~\ref{sec:prelim}, we first set up our notation, and we then provide some background on the random walk loop soup. In addition, we collect elementary results which are needed in the remainder of the paper. In Section~\ref{sec:BLS}, we define the Brownian loop soup and examine some of its properties, in relation with the random walk loop soup. We also consider the arm events in the Brownian loop soup, and we give the associated exponents. In Section~\ref{sec:sep}, we prove a key separation lemma, for random walks inside an independent loop soup. This result is then used in Section~\ref{sec:arm-lwrs} to derive properties of arm events in the random walk loop soup, in particular a locality result for the (interior) four-arm event in Section~\ref{sec:locality}, and then a reversed version of it in Section~\ref{sec:locality2}. In Section~\ref{subsec:other_arm}, we finally consider the corresponding results for the (interior and boundary) two-arm events, as well as for the boundary four-arm events. Next, we prove a quasi-multiplicativity upper bound for the four-arm events in Section~\ref{sec:quasi-mult}. This allows us to obtain the desired power-law upper bound on such events, by relating the discrete events to the corresponding events for the Brownian loop soup. We then explain how to derive the analogous results for the two-arm events, and for the boundary four-arm events.

\section{Preliminaries} \label{sec:prelim}
	
In this preliminary section, we first set notation in Section~\ref{subsec:notation}. We then provide the precise definition of the object which is studied throughout the paper, the random walk loop soup, in Section~\ref{subsec:RWLS}. Finally, we collect several elementary results in Section~\ref{subsec:elem}, that will be particularly useful for our proofs in the subsequent sections.
	
\subsection{Notation and conventions} \label{subsec:notation}

    In the following, $\Nb_0:=\{ 0,1,\ldots \}$ denotes the set of all non-negative integers. 
    Let $\Hb:=\{ (x,y)\in \Rb^2: y>0 \}$ be the upper half plane. If $A$ is a set in $\Rb^2$, we will use $A^+$ to denote its part in the upper half-plane, i.e., $A^+ := A\cap \Hb$. 
    Let $|\cdot|$ be the usual Euclidean norm on the plane $\Rb^2$. For all $z\in\Rb^2$ and $A\subseteq \Rb^2$, let $\dist(z,A) := \inf_{w\in A}|z-w|$ be the associated distance between $z$ and $A$, and write $\diam(A) := \sup_{w,w'\in A}|w-w'|$ for the diameter of $A$ with respect to $|\cdot|$.
    
    For $\delta>0$, we denote the \emph{$\delta$-sausage} of $A\subseteq \Rb^2$ by 
    \begin{equation}\label{eq:sausage}
    	A^{\delta}:=\{z \in \Rb^2 : \dist(z,A)\le\delta \}.
    \end{equation}
    This can be used to define the Hausdorff distance between any two subsets $A$ and $B$ of $\Rb^2$, as
	\[
	d_H\left(A, B\right) := \inf \left\{\delta>0: A \subseteq B^\delta \text { and } B \subseteq A^\delta\right\}.
	\]
	More generally, for any two finite collections $\Ac$ and $\Bc$ of subsets of the plane, their induced Hausdorff distance is given by
\begin{equation}\label{eq:dH}
d_H^*\left(\mathcal{A}, \mathcal{B}\right) :=
\begin{cases}
+\infty \quad \text{ if } |\Ac| \not= |\Bc|,\\
\min_{\sigma \in \mathrm{Bij} (\Ac, \Bc)} \, \max_{A \in \Ac} d_H(A, \sigma(A)) \quad \text{ otherwise},
\end{cases}
\end{equation}
where $\mathrm{Bij} (\Ac, \Bc)$ is the set of all bijections from $\Ac$ to $\Bc$.

    For $z=(x,y),$ $z'=(x',y')\in \Rb^2$, let
    $$d_{\infty}(z,z')=|z-z'|_{\infty} := |x-x'| \vee |y-y'|$$
    be the $l_{\infty}$ distance between $z$ and $z'$. If $A$ and $B$ are two subsets of $\Zb^2$, we write $$d_{\infty}(A,B):=\inf_{z \in A,z'\in B} |z-z'|_{\infty}.$$
    For $A\subseteq\Zb^2$, we denote $A^c := \Zb^2 \setminus A$, and we define its (inner) vertex boundary by
$$\partial A= \partial^{\mathrm{in}} A:= \{ z\in A : |z-w|_{\infty}=1 \text{ for some } w \in A^c\}.$$
    Moreover, we write $\mathring{A} := A\setminus \partial A$.

	For any vertex $z \in \Zb^2$, let
	$$B_R(z) := \{w \in \Zb^2 : |z-w|_{\infty}\le R\} = z + [-R,R]^2 \cap \Zb^2$$
	be the box of side length $2R$ centered on $z$, and write $B_R=B_R(0)$. For $0<r<R$, let
	$$A_{r,R}(z):=\{ w \in \Zb^2 : r\le|z-w|_{\infty}\le R \}$$
	be the annulus of radii $r$ and $R$ around $z$, and write $A_{r,R}$ for $A_{r,R}(0)$.
	We say that a set \emph{crosses} the annulus $A_{r,R}(z)$ if this set intersects both $\partial B_r(z)$ and $\partial B_R(z)$.
In the continuous setting, we use the notation
$$\Bc_{r}(z):=\{ w\in \Rb^2: |z-w|_{\infty} < r \}$$
for the open ball with center $z$ and radius $r$ in the plane, and we denote $\Bc_r = \Bc_r(0)$.
	
	A {\it path} $\eta$ in $\Zb^2$ is a finite sequence of vertices $(z_1,\cdots,z_j)$ in $\Zb^2$ such that $|z_i-z_{i+1}|=1$ for all $1\le i\le j-1$. For such a path $\eta$, the vertices $z_1$ and $z_j$ are called its starting and ending points, respectively, and we let $|\eta|=j-1$ be its length.
	Furthermore, the {\it reversal} of the path $\eta$ is the path $\eta^R := (z_j,\cdots,z_1)$.
	A \emph{rooted loop} is a path for which the starting and ending points coincide. 
	An \emph{unrooted loop} is an equivalence class of rooted loops under rerooting (relabeling).
	For example, the two rooted loops $(z_1,\cdots,z_{j-1},z_j)$ (rooted at $z_1=z_j$) and $(z_2,\cdots,z_j=z_1,z_2)$ (rooted at $z_2$) are in the same equivalence class under rerooting. 
	We denote the corresponding \emph{unrooting map} by $\Uc$, which maps each rooted loop $l$ to its equivalence class $\gamma=\Uc(l)$ as an unrooted loop. Later, we will simply use the term \emph{loop} to refer to unrooted loops. 
For $A\subseteq\Zb^2$, $\eta$ is called a path in $A$ if it is composed only of vertices in $A$, and it is an {\it excursion} in $A$ if in addition to staying in $A$, both of its endpoints belong to $\partial A$.

	We introduce some measures on discrete paths that will be used later. First, 
	let $\mu$ be the measure on paths which assigns, to any path $\eta$ in $\Zb^2$ with length $|\eta|\ge 1$, the weight $4^{-|\eta|}$ (and the weight $0$ otherwise). For $z,w\in \Zb^2$, let $\mu^{z,w}$ be the measure $\mu$ restricted to paths from $z$ to $w$ in $\Zb^2$.
	For $A \subseteq \Zb^2$, let $\mu^{\exc}_A$ be the {\it excursion measure} in $A$, which assigns to each excursion $\eta$ in $A$ the weight $4^{-|\eta|}$. 
	Let $\mu_0$ be the measure on (unrooted) loops, assigning to each loop $\gamma$ the weight $4^{-|\gamma|}$.
	Obviously, these measures are naturally related to the simple random walk in $\Zb^2$. For example, if $z\in D\setminus \partial D, w\in \partial D$, and $\Gamma^{z,w}_D$ is the set of paths $\eta$ from $z$ to $w$ in $D$ such that $\eta\cap\partial D=\{w\}$, then $\mu^{z,w}$ restricted to $\Gamma^{z,w}_D$ is the measure of a simple random walk started from $z$, conditioned on first hitting $\partial D$ at $w$, and stopped when this happens. Moreover, the probability measure of a simple random walk started from $z$ and stopped when it first hits $\partial D$ can be represented as 
	\begin{equation}\label{eq:path-measure}
	\sum_{w\in \partial D}\mu^{z,w}\big|_{\Gamma^{z,w}_D}.
	\end{equation}
	
 Let $A$ be a bounded (not necessarily connected) subset of the plane $\Rb^2$.
 The {\it filling} of $A$ is the complement of the unique unbounded connected component of $\Cb\setminus \bar A$, and we denote it by $\mathrm{Fill}(A)$.
 We write $A^b$ for the (external) {\it frontier} of $A$, which is defined as the boundary of $\mathrm{Fill}(A)$.
 We also view a discrete path (or a discrete loop) as a continuous curve, by linearly interpolating between neighboring vertices, so that it can be considered as a connected set in $\Rb^2$.
 Hence, the filling or the frontier of a path (or a loop) on the lattice is well-defined in this sense, and the notion extends naturally to sets of loops.  
	
For two positive functions $f(x)$ and $g(x)$, we write $f(x)\lesssim g(x)$ if there exists a constant $C>0$ such that $f(x)\le C \, g(x)$ for all $x$ (or for all $x$ large enough, depending on the context); and if the constant $C$ depends on $\alpha$, we write $f(x)\lesssim_{\alpha} g(x)$ instead. We use $c,c',C,C',\ldots$ to denote arbitrary constants which may vary from line to line, while the constants $c_1$ and $c_2$ will be fixed throughout the paper. If a constant depends on some other quantity, this will be made explicit. For example, if $c$ depends on $\alpha$, we write $c(\alpha)$ to indicate it.

\subsection{Random walk loop soup} \label{subsec:RWLS}

We can now define the random walk loop soup, but before that, we need to set briefly some additional notations for loops.

\subsubsection*{Notations for loops}

Denoting by $\Gamma$ the set of loops in $\Zb^2$, a loop configuration $L$ is a multiset of $\Gamma$, i.e., $L\in (\Nb_0)^{\Gamma}$. For any $\gamma \in \Gamma$, we denote by $n_L(\gamma)$ the number of occurrences of $\gamma$ in the loop configuration $L$. We define the union of two loop configurations $L_1$ and $L_2$ as $L:=L_1 \uplus L_2$, such that $n_L(\gamma)=n_{L_1}(\gamma)+n_{L_2}(\gamma)$ for all $\gamma\in \Gamma$. Similarly, let $L:=L_1\setminus L_2$ be the loop configuration that satisfies $n_L(\gamma)=(n_{L_1}(\gamma)-n_{L_2}(\gamma))\vee 0$ for all $\gamma\in \Gamma$. We use $L^b$ to denote the collection of all frontiers of loops in $L$.

A chain of loops (with length $j \geq 1$) is a sequence of loops $\gamma_1,\cdots,\gamma_j$ such that for all $1\le i\le j-1$, $\gamma_i$ intersects $\gamma_{i+1}$ (i.e. these two paths have at least one vertex in common). Two loops $\gamma$ and $\gamma'$ are said to be connected if there exists a finite chain of loops from $\gamma$ to $\gamma'$, that is, if for some $j \geq 1$, there exists a chain of $j$ loops with $\gamma_1 = \gamma$ and $\gamma_j = \gamma'$. This defines an equivalence relation on loops, whose classes are called {\it clusters}. In other words, a cluster is a maximal connected component in $L$, and the set of clusters induces a partition of $\Zb^2$.

An outermost cluster $\Cc$ of $L$ is such that there exists no other cluster $\Cc'$ of $L$ for which $\mathrm{Fill}(\Cc)\subseteq \mathrm{Fill}(\Cc')$.  For any subset $D$ of $\Zb^2$, we use $\Lambda(D,L)$ to denote the union of $D$ and all the clusters in $L$ that intersect $D$. We write $L_D$ for the set of loops in $L$ that are contained in $D$, and we write $L^D$ for the set of loops in $L$ that intersect $D$. We use the shorthand notation $L_d(z):=L_{B_d(z)}$. Moreover, for $B_d(z)\subseteq D$, we denote $L_{d,D}(z):=L_D\setminus L_d(z)$. As usual, we omit the parameter $z$ if $z=0$, so that for example, we write $L_d$ for $L_d(0)$.

\subsubsection*{Definition of RWLS}

For $\gamma\in \Gamma$, let $J(\gamma)$ be the \emph{multiplicity} of $\gamma$, i.e., the largest integer $j$ such that $\gamma$ is exactly the concatenation of $j$ copies of the same loop. Define the \emph{unrooted loop measure} $\nu$ by
\begin{equation}\label{eq:nu}
	\nu(\gamma):=\frac{\mu_0(\gamma)}{J(\gamma)}, \quad \gamma \in \Gamma,
\end{equation}
where $\mu_0(\gamma) = 4^{-|\gamma|}$ is the measure (on unrooted loops) defined in Section~\ref{subsec:notation}.
The {\it random walk loop soup} (RWLS) $\Lc$ (in $\Zb^2$) of intensity $\alpha>0$ is made of the Poisson point process with intensity $\alpha\nu$.

All the notations introduced for deterministic loop configurations can be extended to the RWLS in a natural way, simply replacing $L$ by $\Lc$. For example, $\Lc_D$ is the \emph{random walk loop soup in $D$}, consisting of all the loops in $\Lc$ which are fully contained in $D$. The definitions of $\Lc^D$, $\Lc_d(z)$, and $\Lc_{d,D}(z)$ follow similarly. Finally, $\Lc_D^b$, obtained as the collection of frontiers of the loops in $\Lc_D$, is called the {\it frontier loop soup in $D$}.

\subsection{Elementary results} \label{subsec:elem}

In this final section of preliminaries, we collect basic properties of random walks and loop soups, which play an important role in our later proofs.

\subsubsection*{Estimates on Green's function}

If $D \subseteq \Zb^2$ is finite, recall that the \emph{Green's function} $G_D(.,.)$ in $D$ is defined as follows: for all $u, v \in D$,
$$G_D(u,v) := \Eb^u\bigg( \sum_{n=0}^{\zeta_D} \ind_{\{S_n=v \}}  \bigg),$$
where $S$ is a simple random walk, we write
$$\zeta_{D}:=\min\{ n \ge 0: S_n \notin D \}$$
for its exit time from $D$, and $\Eb^u$ denotes the expectation with respect to $S$ starting from $u$. In the case when $D$ is an arbitrary subset of $\Rb^2$ (i.e., not necessarily consisting of vertices), with a slight abuse of notation, we also write $G_D$ for $G_{D \cap \Zb^2}$.

In several instances below, we use the following standard estimate about the Green's function, see e.g.\ Theorem 1.6.6 and Proposition 1.6.7 of \cite{La1991}.

\begin{lemma}\label{lem:green's}
	There exists a universal constant $c>0$ such that for all $n \geq 1$,
	\[
	G_{B_n}(0,0)=\frac2\pi\log n+c+O(n^{-1}),
	\]
	and furthermore, for all $x\in B_n\setminus\{0\}$, 
	\[
	G_{B_n}(x,0)=\frac2\pi(\log n-\log |x|)+O \big( |x|^{-1}+n^{-1} \big).
	\]
\end{lemma}

This result can be used to derive a handy lower bound for the Green's function inside the $(\delta n)$-sausage (see \eqref{eq:sausage} for its definition) of a subset with a diameter comparable to $n$, for each given $\delta > 0$ (uniformly in $n \geq 1$). Later, we make use of the specific version below.

\begin{lemma} \label{lem:green_sausage}
	For $n \geq 1$, and two vertices $x, y \in \partial B_n$, let $J_{x,y}$ be any of the two arcs (viewed as a continuous curve, interpolating between the vertices) along $\partial B_n$ with endpoints $x$ and $y$ (i.e., clockwise or counterclockwise). For each $\delta \in (0,1/100)$, there exists $C(\delta)>0$ such that the following holds. For all $n$, $x$ and $y$ as before, and such that $\diam(\partial B_n \setminus J_{x,y}) \geq 10 \delta n$, we have
	\[
	G_{J_{x,y}^{\delta n}}(x,y)\ge C.
	\]
\end{lemma}

\begin{proof}
	Recall that the $(\delta n)$-sausage of $J_{x,y}$ is defined as $J_{x,y}^{\delta n}=\{ z\in \Rb^2: \dist(z,J_{x,y})\le \delta n \}$, and also that $G_{J_{x,y}^{\delta n}}$ is a shorthand notation for $G_{J_{x,y}^{\delta n} \cap \Zb^2}$. Clearly, we can always assume $n$ to be large enough in the proofs below, if needed.

We introduce the ball $B_{\delta n/2}(y)$, and decompose a random walk path from $x$ to $y$ according to its first hitting of $\partial B_{\delta n/2}(y)$. Using the strong Markov property, we have
	\[
	G_{J_{x,y}^{\delta n}}(x,y)=\sum_{z\in \partial B_{\delta n/2}(y)} \Pb^x( S_{\zeta_{J_{x,y}^{\delta n} \setminus B_{\delta n/2}(y)}}=z ) \,  G_{J_{x,y}^{\delta n}}(z,y).
	\]
	By Lemma~\ref{lem:green's}, there exists a universal constant $c>0$ such that 
	\[
	G_{J_{x,y}^{\delta n}}(z,y) \ge G_{B_{\delta n}(y)}(z,y) \ge c.
	\]
	Thus, denoting by $\tau_A$ the hitting time (for $S$) of a subset $A \subseteq \Zb^2$,
	\[
	G_{J_{x,y}^{\delta n}}(x,y)\ge c\, \Pb^x( S_{\zeta_{J_{x,y}^{\delta n} \setminus B_{\delta n/2}(y)}}\in \partial B_{\delta n/2}(y) ) \ge c\, \Pb^x( \zeta_{J_{x,y}^{\delta n/2}} > \tau_{\partial B_{\delta n/2}(y)} )\ge C,
	\]
	where the last inequality follows from the fact that for a simple random walk $S$ starting from $x$, the probability that it makes a loop inside $(\partial B_n)^{\delta n/2}$ before exiting this subset is uniformly bounded from below (from the invariance principle). This finishes the proof.
\end{proof}

\subsubsection*{Non-disconnection by random walks}

Next, we review non-disconnection probabilities for simple random walks.

\begin{lemma}[\cite{La1996,LSW2002a,LP2000}]
	\label{lem:disconnection-prob}
	There exists a universal constant $c > 0$ such that the following holds. Let $1 \leq r < R$,  $z \in \partial B_r$, and $w \in \partial B_R$. Let $S$ denote a simple random walk in $\Zb^2$, and for any $s>0$, let $\tau_s = \tau_{\partial B_s}$ be the first hitting time of $\partial B_s$ by $S$. Then, the non-disconnection probabilities satisfy the following upper bounds:
	\begin{equation}\label{eq:inout}
		\Pb^z\left(S\left[0, \tau_R\right] \text { does not disconnect }  B_r \text { from } \partial B_R  \right) \leq c\, (r/R)^{1/4},
	\end{equation}
	and similarly,
	\begin{equation}\label{eq:outin}
		\Pb^w\left(S\left[0, \tau_r\right] \text { does not disconnect }  B_r \text { from } \partial B_R  \right) \leq c\, (r/R)^{1/4}.
	\end{equation}
\end{lemma}
The analogous results for Brownian motion were first proved in \cite{La1996} and \cite{LSW2002a} (the first paper obtained a power-law upper bound, and the second one obtained the exponent $1/4$). On the other hand, it was shown in \cite{LP2000}, by using Skorokhod's embedding, how to convert the result for Brownian motion into a result for the simple random walk, with the same upper bound.

We can use the above lemma to show that even though the measure $\mu^{z,w}$ has infinite total mass for all $z,w\in \Zb^2$, its mass becomes finite after restricting to some non-disconnection event. For any $1 \leq r < R$, and $z,w\in \partial B_r$, let $M^{z,w}_{r,R}$ be the set of paths from $z$ to $w$ which intersect $\partial B_R$, and do not disconnect $B_r$ from $\partial B_R$. We have the following upper bound.

\begin{lemma}
	\label{lem:non-disconnecting paths}
	Let $c$ be the constant from Lemma~\ref{lem:disconnection-prob}. There exists a universal constant $c' > 0$ such that for all $1 \leq r < R$ with $(2\vee 8c^4)r < R$, for all $z,w\in \partial B_r$,
	\[
	\mu^{z,w}( M^{z,w}_{r,R} ) \le c'\, (R/r)^{-1/2} \log(R/r).
	\]
\end{lemma}

\begin{proof}
	Each path $\gamma$ in $M^{z,w}_{r,R}$ can be decomposed into crossings between $\partial B_{2r}$ and $\partial B_R$, that is, we let $\sigma_0=0$, and for $k\ge 0$, 
	\[
	\tau_k=\inf\{ j\ge \sigma_{k}: \gamma_j \in \partial B_{2r} \},\quad \text{and } \sigma_{k+1}=\inf\{ j\ge \tau_{k}: \gamma_j \in \partial B_{R} \}
	\]
(with $\inf \emptyset = \infty$).

For $k \geq 1$, let $U_k$ be the set of paths in $M^{z,w}_{r,R}$ with exactly $2k$ such crossings, i.e., satisfying $\tau_k<\infty$ and $\sigma_{k+1}=\infty$. First, we note that $\gamma[\sigma_0,\tau_0]$ under $\mu^{z,w}$ is a simple random walk started from $z$ and stopped when it first hits $\partial B_{2r}$ (see \eqref{eq:path-measure}).
Let $1\le j\le k$. By the strong Markov property, conditioned on $\tau_{j-1}<\infty, \gamma_{\tau_{j-1}}=x\in \partial B_{2r}$ and $\sigma_{j}<\infty$, $\gamma[\tau_{j-1},\sigma_{j}]$ under $\mu^{z,w}$ has the same law as a simple random walk started from $x$ and stopped when it first hits $\partial B_R$, which does not disconnect $B_{2r}$ from $\partial B_R$ with probability at most $c\, (2r/R)^{1/4}$ by \eqref{eq:inout}; and conditioned on $\sigma_{j}<\infty, \gamma_{\sigma_{j}}=y\in \partial B_R$ and $\tau_{j}<\infty$, $\gamma[\sigma_{j},\tau_{j}]$ under $\mu^{z,w}$ is the same as a simple random walk started from $y$ and stopped when it first hits $\partial B_{2r}$, which does not disconnect $B_{2r}$ from $\partial B_R$ with probability bounded by $c\, (2r/R)^{1/4}$, from \eqref{eq:outin}. Moreover, the last part of $\gamma\in U_k$, from $\gamma_{\tau_k}\in \partial B_{2r}$ to $w\in \partial B_r$ in $B_R$, has total mass given by the Green's function $G_{B_R}(\gamma_{\tau_k},w)$, which is uniformly bounded by $C\log(R/r)$ for some universal constant $C>0$ by Lemma~\ref{lem:green's}. Indeed,
$$G_{B_R}(\gamma_{\tau_k},w) \leq G_{B_{R+r}}(0,\gamma_{\tau_k}-w) \leq \frac{2}{\pi} \log \Big( \frac{R+r}{r} \Big) + O(r^{-1}),$$
where we used Lemma~\ref{lem:green's} in the second inequality, as well as $|\gamma_{\tau_k}-w| \geq r$. It then suffices to use that by assumption, $r \leq R/2$.

Therefore, for each $k\ge 1$,
	\begin{align*}
		\mu^{z,w}( U_k )& \leq \mu^{z,w} ( \gamma[\tau_j,\tau_{j+1}] \text{ does not disconnect } B_{2r} \text { from } \partial B_R \text{ for all } 0\le j\le k-1,\, \sigma_{k+1}=\infty)\\
		&\le \big( c\, (2r/R)^{1/4} \big)^{2k} \, C\log(R/r).
	\end{align*}
	Note that every path in $M^{z,w}_{r,R}$ makes at least two crossings, and that $\big( c\, (2r/R)^{1/4} \big)^{2}<1/2$ since $R>8c^4 r$. It follows that 
	\[
	\mu^{z,w}( M^{z,w}_{r,R} ) = \sum_{k\ge 1} \mu^{z,w}( U_k )\le \sum_{k\ge 1} \big( c\, (2r/R)^{1/4} \big)^{2k} \, C\log(R/r)
	\le 4c^2 C (R/r)^{-1/2} \log (R/r),
	\]
	concluding the proof.
\end{proof}

We now improve Lemma~\ref{lem:non-disconnecting paths} by allowing the ratio $R/r$ to be arbitrarily close to $1$, which happens to be useful later. At this stage, we want to make formally a remark about notations. In the result below, when we write $(1+\delta)r$, what we really mean is its integer part $[(1+\delta)r]$. In the remainder of the paper, we adopt this lighter notation for convenience, and we always assume implicitly that the relevant parameter ($r$ here) is large enough for the statements to make sense.

\begin{lemma}
	\label{lem:non-disconnecting paths-1}
	There exist universal constants $C',C''>0$ such that for all $\delta>0$, $r \geq 1$, and $z,w\in \partial B_r$,
	\[
	\mu^{z,w}\big( M^{z,w}_{r,(1+\delta)r} \big) \le C' \log(1/\delta)+C''.
	\]
\end{lemma}
\begin{proof}
	Let $\rho:=2\vee 8c^4$.
	By Lemma~\ref{lem:non-disconnecting paths}, we only need to prove the lemma when $1+\delta \le \rho$, which we assume from now on. Note that $M^{z,w}_{r,2\rho r} \subseteq M^{z,w}_{r,(1+\delta)r}$. By decomposing each path in $M^{z,w}_{r,(1+\delta)r}\setminus M^{z,w}_{r,2\rho r}$ according to its first visit of $\partial B_{(1+\delta)r}$, it is easy to see that
$$\mu^{z,w} \big( M^{z,w}_{r,(1+\delta)r}\setminus M^{z,w}_{r,2\rho r} \big) \le \max_{z'\in \partial B_{(1+\delta)r} } G_{B_{2\rho r}}(z',w).$$
Using again the estimate for the Green's function from Lemma~\ref{lem:green's}, we get that for some universal constant $C'>0$, for all $z'\in \partial B_{(1+\delta)r}$,
$$G_{B_{2\rho r}}(z',w) \le G_{B_{(2\rho+1) r}}(0,z'-w) \leq C' \log \left( \frac{(2\rho+1) r}{\delta r} \right)=C' \log(1/\delta)+C'\log(2\rho+1).$$
Therefore, 
	\begin{align*}
		\mu^{z,w} \big( M^{z,w}_{r,(1+\delta)r} \big) &= \mu^{z,w} \big( M^{z,w}_{r,(1+\delta)r}\setminus M^{z,w}_{r,2\rho r} \big)
		+ \mu^{z,w} \big( M^{z,w}_{r,2\rho r} \big)\\
		&\le C' \log(1/\delta)+C'\log(2\rho+1)+ c' (2\rho)^{-1/2} \log(2\rho),
	\end{align*}
	where we applied Lemma~\ref{lem:non-disconnecting paths} with $R=2\rho r$ to bound the term $\mu^{z,w}( M^{z,w}_{r,2\rho r} )$. This completes the proof, by choosing $C'' = C'\log(2\rho+1) + c' (2\rho)^{-1/2} \log(2\rho)$.
\end{proof}

\subsection*{Consequences on crossing loops}

The next lemma deals with loops in the loop soup that do not visit the origin and do not disconnect some given box. It can be derived from Lemma~\ref{lem:disconnection-prob}.

\begin{lemma}[{\cite[Lemma 2.7]{La2020}}]
	\label{lem:disconnect-loops-1}
	Let $\Lc_{\Zb^2\setminus\{0\}}$ be the random walk loop soup in $\Zb^2\setminus\{0\}$ with intensity $1$. There exists a universal constant $c > 0$ such that for all $1 \leq r < R$,
	\begin{equation}
		\Pb\left( \begin{array}{c} 
			\text{there is a loop $\gamma$ in } \Lc_{\Zb^2\setminus\{0\}} \text{ such that $\gamma$ crosses } \\ 
			\text{$A_{r,R}$ and does not disconnect $B_r$ from $\partial B_{R}$} \end{array}\right)
		\le c\, (r/R)^{1/2}.
	\end{equation}
\end{lemma}

Next, we take into account the loops in a bounded domain that visit the origin.

\begin{lemma}
	\label{lem:disconnect-loops-2}
	Let $1 \leq r < R<R'$. Let $\Lc_{R'}$ be the random walk loop soup in $B_{R'}$ with intensity $1$. There exists a universal constant $c > 0$ such that 
	\begin{equation}\label{eq:gamma}
		\Pb\left( \begin{array}{c} 
			\text{there is a loop $\gamma$ in } \Lc_{R'} \text{ such that $\gamma$ crosses } \\ 
			\text{ $A_{r,R}$ and does not disconnect $B_r$ from $\partial B_{R}$} \end{array}\right)
		\le c\, (r/R)^{1/2}.
	\end{equation}
\end{lemma}
\begin{proof}
Let $E$ be the event that there is a loop $\gamma$ in  $\Lc_{R'}$ which visits the origin such that $\gamma$ crosses $A_{r,R}$ and does not disconnect $B_r$ from $\partial B_{R}$.
	By Lemma~\ref{lem:disconnect-loops-1}, we only need to show $\Pb(E)\le c(r/R)^{1/2}$. Let $S$ be a simple random walk starting from the origin, and define the stopping times
	\[
	\tau_{m} = \tau_{\partial B_m} :=\inf\{k\ge 0: S_k\in \partial B_m \} \:\: (m \geq 0),\quad \text{and }
	\sigma:=\inf\{ k\ge\tau_R: S_k\in B^c_{R'}\cup\{0\} \}.
	\]
	By Lemma~2.5 of \cite{La2020}, $\Pb(E)$ is bounded by 
	\begin{equation}\label{eq:S_sigma}
		\Pb\big( S_\sigma=0, \text{ and both $S[\tau_r,\tau_R]$ and $S[\tau_R,\sigma]$ do not disconnect $B_r$ from $\partial B_{R}$} \big).
	\end{equation}
	By the strong Markov property, and using Lemma~\ref{lem:disconnection-prob} twice, we get that \eqref{eq:S_sigma} is bounded by a constant multiple of $(r/R)^{1/2}$, which concludes the proof. 
\end{proof}

Note that if $\eta$ is the frontier of a loop $\gamma$ and $\eta$ crosses the annulus $A_{r,R}$, then $\gamma$ crosses $A_{r,R}$ and it does not disconnect $B_r$ from $\partial B_{R}$. Therefore, we get the following result as a corollary of the previous lemma.
\begin{lemma}\label{lem:frontier-loops}
	Let $1 \leq r < R<R'$. Let $\Lc^b_{R'}$ be the set of frontiers of loops in the random walk loop soup in $B_{R'}$ with intensity $1$. There exists a universal constant $c > 0$ such that 
	\begin{equation}\label{eq:gamma'}
		\Pb \big(
		\text{there is a loop in $\Lc^b_{R'}$ that crosses 
			$A_{r,R}$} \big)
		\le c\, (r/R)^{1/2}.
	\end{equation}
\end{lemma}

\subsubsection*{Two classical tools for loop soups}

Finally, let us record some important properties for Poisson ensembles of loops. The first one below is known as Palm's formula (see e.g.\ \cite[Proposition 15]{MR2815763}) or Mecke's equation (see e.g.\ \cite[Theorem 4.1]{LP2017}).
\begin{lemma}[Palm's formula]\label{lem:Palm formula}
	Let $\alpha > 0$, and $\Lc_{\alpha}$ be the random walk loop soup with intensity $\alpha$ in the whole plane.
	Given any bounded functional $\Phi$ on loop configurations, and any integrable loop functional $F$, we have:
	\[
	\mathbb{E}\left(\sum_{\gamma \in \mathcal{L}_\alpha} F(\gamma) \Phi\left(\mathcal{L}_\alpha\right)\right)=\int \mathbb{E}\left(\Phi\left(\mathcal{L}_\alpha \uplus\{\gamma\}\right)\right) \alpha F(\gamma) \nu(d \gamma).
	\]
	Moreover, this relation also holds if we restrict all the objects considered to a finite domain $D$.
\end{lemma}

The next property is the celebrated FKG inequality, which in fact holds for general Poisson point processes. Here, we state it specifically for the RWLS. Recall that a function $f$ is said to be increasing (on the space of loop configurations) if for any two $L \subseteq L'$, we have $f(L)\le f(L')$. The following result can be found in \cite[Theorem 20.4]{LP2017}.

\begin{lemma}[FKG inequality]\label{lem:FKG-RWLS}
	Let $f,g\in L^2(\Pb)$ be increasing functions on loop configurations, and $\alpha > 0$. Then,
	\[
	\Eb(fg)\ge \Eb(f)\cdot \Eb(g),
	\]
where the expectation is for the RWLS with intensity $\alpha$.
\end{lemma}

\section{Brownian loop soup} \label{sec:BLS}

We now exploit the link between the RWLS and its natural counterpart in the continuum, the Brownian loop soup (BLS). We first define this process in Section~\ref{sec:def_BLS}, and we explain how it is related to the conformal loop ensemble (CLE) processes. This connection is then used in Section~\ref{sec:csq_CLE_RWLS} to derive properties of the RWLS. Finally, we state arm estimates for the BLS in Section~\ref{sec:exp_cle}.

\subsection{Definition of BLS} \label{sec:def_BLS}

The Brownian loop soup is the continuous analog of the random walk loop soup. For a domain $D$ in $\Rb^2$, the \emph{Brownian loop soup} $\wt\Lc_D$ in $D$ with intensity $\alpha\in (0,\half]$ is defined as a Poisson point process with intensity $\alpha\wt\nu_D$. Here, $\wt\nu_D$ is a measure on Brownian loops given by 
\begin{equation}\label{eq:nu_D}
\wt\nu_D:=\int_D\int_{0}^{\infty} \frac1t\, p^D_t(z,z)\, \Pb^D_{t,z}\, dt\, dz,
\end{equation}
where $p^D_t(z,z)$ is the heat kernel in $D$ and $\Pb^D_{t,z}$ is the law of a Brownian bridge in $D$ from $z$ to $z$ of duration $t$. Then, $\wt\nu_D$ is a sigma-finite measure, which blows up as $t \searrow 0$. Therefore, $\wt\Lc_D$ forms an infinite collection of Brownian loops in $D$. 

It has been shown in \cite{SW2012} that the outer boundaries of the outermost clusters in the Brownian loop soup $\wt\Lc_D$ of intensity $\alpha(\kappa)$ is distributed as the conformal loop ensemble (CLE) in $D$ with parameter $\kappa \in (8/3,4]$, via the relation
\begin{align}\label{eq:alp-kap}
	\alpha(\kappa) =\frac{(3\kappa-8)(6-\kappa)}{4\kappa}.
\end{align}
Recall that $(\CLE_{\kappa})_{\kappa\in (8/3,4]}$ is defined in \cite{SW2012} as the only random collection of disjoint simple loops which satisfies conformal invariance and a certain restriction axiom. 
Note that $\alpha(\kappa)$ is increasing in $\kappa$ for $\kappa\in (8/3,4]$. When $\kappa$ varies from $8/3$ to $4$, $\alpha(\kappa)$ varies from $0$ to $1/2$. In this paper, we will always use $\alpha$ to denote the intensity of the loop soup, and $\kappa$ to denote the parameter of the corresponding CLE.

\subsection{Consequences for the RWLS} \label{sec:csq_CLE_RWLS}

Next, we consider clusters in a subcritical or critical loop soup, i.e. with intensity at most $1/2$, and control their sizes simultaneously using the connection with CLE.

We first recall the following result on the convergence of RWLS clusters towards CLE proved in \cite{BCL2016,Lu2019}.
\begin{theorem}[\cite{BCL2016,Lu2019}]
\label{thm:convergence}
For all $\alpha\in (0,\half]$, and $R \geq 1$, let $\Lc_R$ be the random walk loop soup  in $B_R$ with intensity $\alpha$. Let $\Fc(\Lc_R)$ be the collection of all frontiers (seen as continuous loops in $\Rb^2$, by linear interpolation, as explained in Section~\ref{subsec:notation}) of outermost clusters in $\Lc_R$. In the collection of rescaled loops $R^{-1}\Fc(\Lc_R)$ (each loop is rescaled by $R^{-1}$), if we consider the loops that have a diameter greater than $\delta$ (there are a.s.\ finitely many), then they converge in law  to the collection of loops in a $\mathrm{CLE}_{\kappa}$ in $\Bc_1$ with a diameter greater than $\delta$, with respect to $d^*_H$ (recall that $d^*_H$ was defined in \eqref{eq:dH}), where the relation between $\alpha$ and $\kappa$ is given by \eqref{eq:alp-kap}.
\end{theorem}

\begin{lemma}\label{lem:cluster-size}
	For any $\delta\in (0,1)$, there exists $c(\delta)>0$ such that the following holds. For all $\alpha\in (0,\half]$, and $R \geq 1$, we have 
	\begin{equation}\label{eq:cluster-converge}
		\Pb \big( \text{all clusters in $\Lc_R$ have a diameter $< \delta R$} \big) \ge c(\delta).
	\end{equation}
\end{lemma}
\begin{proof}
First, we can assume that $\alpha=1/2$, since by monotonicity it is enough to prove \eqref{eq:cluster-converge} in that case. 
Our proof uses the continuous process $\mathrm{CLE}_{\kappa}$, for the associated value $\kappa = 4$.
	
By Theorem~\ref{thm:convergence}, as $R\rightarrow\infty$, the probability in \eqref{eq:cluster-converge} converges to 
	\[
	c'(\delta):=\Pb \big( \text{all loops in $\mathrm{CLE}_{4}$ in $\Bc_1$ have a diameter $< \delta$} \big)>0.
	\]
	We deduce that there exists some constant $R_0>0$ such that for all $R\ge R_0$, the probability in \eqref{eq:cluster-converge} is greater than $c'(\delta)/2$. When $R<R_0$, with a probability $c''(R_0)>0$ which only depends on $R_0$, the collection $\Lc_R$ is simply empty. Thus, choosing $c(\delta) := c'(\delta)/2 \wedge c''(R_0)$ completes the proof.
\end{proof}

Now, we deal with the clusters made by frontiers of loops, and we control the sizes of those intersecting some given box.

\begin{lemma}\label{lem:cluster-size-2}
	For all $\delta>0$, there exists $C(\delta)>0$ such that for all $\alpha\in (0,\half]$, and all $1 \leq r < R$, we have $\Pb( E(\delta,\alpha,r,R) )\ge C(\delta)$, where $E(\delta,\alpha,r,R)$ is the event that in the frontier loop soup $\Lc^b_R$ of intensity $\alpha$, all clusters intersecting $B_r$ have a diameter $< \delta r$.
\end{lemma}
\begin{proof}
	Let $c$ be the universal constant in Lemma~\ref{lem:frontier-loops}. Choose a large constant $\rho=\rho(\delta)>1+\delta$ such that 
	\[
	c':=c\, \bigg(\frac{1+\delta}{\rho}\bigg)^{1/2}<1.
	\]
	Let $E'(\delta,\alpha,r,R, \rho)$ be the event that there is no loop in $\Lc^b_{R}$ that crosses $A_{(1+\delta) r,\rho r}$. Note that if $R\le \rho r$, then $E'(\delta,\alpha,r,R, \rho)$ occurs with probability one. Since $E'(\delta,\alpha,r,R)$ is decreasing in $\alpha$, by Lemma~\ref{lem:frontier-loops}, 
	\[
	\Pb( E'(\delta,\alpha,r,R, \rho) )\ge \Pb( E'(\delta, 1,r,R,\rho ))\ge 1-c'.
	\]
	Let $E''(\delta,\alpha,r, \rho)$ be the event that in the loop soup $\Lc_{\rho r}$ with intensity $\alpha$, all clusters have a diameter $< \delta r$. By Lemma~\ref{lem:cluster-size}, we have 
	\[
	\Pb( E''(\delta,\alpha,r, \rho) )\ge c(\delta/\rho)>0.
	\]
	We observe that for all $\rho>1+\delta$,
	\[
	E'(\delta,\alpha,r,R, \rho)\cap E''(\delta,\alpha,r, \rho) \subseteq E(\delta,\alpha,r,R),
	\]
	and that the two events $E'(\delta,\alpha,r,R)$ and $E''(\delta,\alpha,r)$ are independent. Therefore,
	\[
	\Pb( E(\delta,\alpha,r,R) )\ge \Pb( E'(\delta,\alpha,r,R, \rho) ) \, \Pb( E''(\delta,\alpha,r, \rho) ) \ge (1-c') \, c(\delta/\rho).
	\]
	This completes the proof, by choosing $C(\delta):=(1-c') \, c(\delta/\rho)$.
\end{proof}

\begin{lemma}\label{lem:crossing loops}
	For all $\delta>0$, there exists $C(\delta)>0$ such that for all $\alpha\in (0,\half]$, and $1 \leq r < R$, we have
	\begin{equation}\label{eq:crossing loop in Lb}
		\Pb(\text{there is no loop in $\Lc^b_R$ crossing $A_{r,(1+\delta)r}$})\ge C(\delta),
	\end{equation}
and moreover,
	\begin{equation}\label{eq:crossing loop}
	\Pb(\text{there is no loop in $\Lc_R$ crossing $A_{r,(1+\delta)r}$ without disconnecting $B_r$ from $\partial B_{(1+\delta)r}$})\ge C(\delta).
\end{equation}
\end{lemma}
\begin{proof}
First, \eqref{eq:crossing loop in Lb} is an immediate consequence of Lemma~\ref{lem:cluster-size-2}, since every loop in $\Lc^b_R$ is contained in some cluster of $\Lc^b_R$. On the other hand, \eqref{eq:crossing loop} follows from the fact that if $\gamma$ is a loop that crosses $A_{r,(1+\delta)r}$, but that does not disconnect $B_r$ from $\partial B_{(1+\delta)r}$, then $\gamma^b$ (the frontier of $\gamma$) crosses $A_{r,(1+\delta)r}$. 
\end{proof}

\begin{remark}\label{rmk:frontier}
Note that the results in Lemma~\ref{lem:cluster-size-2} and \eqref{eq:crossing loop in Lb} do not hold if we replace $\Lc_R^b$ by $\Lc_R$ in the statement.
	Indeed, with high probability (when $R$ is large), there are big loops in $\Lc_R$ that intersect $B_r$. This is the reason why we sometimes need to work with the frontier loop soup.
\end{remark}

\subsection{Arm exponents in BLS and CLE} \label{sec:exp_cle}

In this section, we introduce certain two-arm and four-arm events in the Brownian loop soup (BLS), which can be equivalently seen as events about the CLE, due to the correspondence invoked in Section~\ref{sec:def_BLS}. 
We state two results which follow from \cite{GNQ2024c} and which will be used as an input in this paper. We consider two different situations, which correspond to the interior and the boundary cases respectively. 

\subsubsection*{Interior arm exponents}
We first consider the interior case. In order to stress the dependence on $\alpha$, we denote by $\wt\Lc_{\alpha}$ the BLS in the unit box $\Bc_1$, with intensity $\alpha$. We define the following two-arm and four-arm events.

\begin{definition}[Two-arm events for BLS]\label{def:n_arm_bls}
	For $\alpha\in (0,\half]$ and $0<\eps<r<1$, let $\wt\Ac^{2}_{\alpha}(\eps, r)$ be the event that there is at least one outermost cluster in $\wt\Lc_{\alpha}$ whose frontier crosses $\Bc_r\setminus \Bc_\eps$.
\end{definition}

\begin{definition}[Four-arm event for BLS]\label{def:arm_bls}
For $\alpha\in (0,\half]$, and $0<\eps<r<1$, let $\wt\Ac_{\alpha}(\eps, r)$ be the  event  that there are at least two outermost clusters in $\wt\Lc_{\alpha}$ crossing $\Bc_r\setminus \Bc_\eps$.
\end{definition}

Since we introduce and analyze extensively the four-arm event for the RWLS later (see Definition~\ref{def:arm event for rwls}), we reserve the notation $\Ac(\cdot,\cdot)$, without a ``$\sim$'', for that discrete arm event.

\begin{remark}
Definitions~\ref{def:n_arm_bls} and~\ref{def:arm_bls} can be equivalently formulated for a CLE$_\kappa$ in $\Bc_1$, where $\alpha$ and $\kappa$ are related by \eqref{eq:alp-kap}, namely $\wt\Ac^2_{\alpha}(\eps, r)$ (resp.\ $\wt\Ac_{\alpha}(\eps, r)$) is equal to the event that there are at least one (resp.\ two) loop(s) in the CLE$_\kappa$ crossing $\Bc_r\setminus \Bc_\eps$. In the CLE$_\kappa$, on the event  $\wt\Ac^2_{\alpha}(\eps, r)$ (resp.\ $\wt\Ac_{\alpha}(\eps, r)$), there are two (resp.\ four) curves (arms) crossing $\Bc_r\setminus \Bc_\eps$.
However, without invoking CLE, there is another reason why we call it the ``four-arm event''. Indeed, we use this terminology also in analogy with the four-arm events in Bernoulli percolation, corresponding to the existence of four crossing paths with alternating types (occupied and vacant). This event encodes the existence of two disjoint connected components of occupied sites crossing a given annulus.
\end{remark}

For $\kappa\in (8/3,4]$, let 
\begin{align}\label{eq:arm-exponent}
 \eta^{2}(\kappa) := 1-\frac{\kappa}{8} \quad \text{and} \quad \eta(\kappa) :=\frac{(12-\kappa)(\kappa+4)}{8\kappa}.
\end{align}
These values $\eta^2(\kappa)$ and $\eta(\kappa)$ are also equal to the well-known interior two-arm and four-arm exponents of SLE$_\kappa$, which appeared in various contexts.
A mathematical proof for SLE$_\kappa$ arm exponents was first given by Smirnov and Werner \cite{SW2001} for $\kappa=6$ (due to the interest in percolation), using suitable SLE martingales. 
A derivation of $\eta^2(\kappa)$ for all $\kappa\in(0,8)$
was contained in \cite{MR2435854}, which focused on establishing the Hausdorff dimension of SLE$_\kappa$. These works relied on the earlier works \cite{MR2153402, LSW2001a,LSW2001b,MR1899232}.
In \cite{MR2112128,MR2581884}, physicists also used KPZ relations to obtain the formulas of the SLE$_\kappa$ arm exponents.
Finally, a proof of SLE$_\kappa$ arm exponents was written down in \cite{MR3846840}, again based on  SLE martingales.

For any $\alpha\in (0,1/2]$, we define $\xi^2(\alpha):=\eta^2(\kappa)$ and $\xi(\alpha):=\eta(\kappa)$, where $\alpha$ and $\kappa$ are related by \eqref{eq:alp-kap}. We will use the following upper bounds in the BLS as an input, which are consequences of the results established in \cite{GNQ2024c} (but weaker), relying on the connection between the BLS and $\mathrm{CLE}_{\kappa}$.

\begin{proposition}[\cite{GNQ2024c}]
\label{prop:arm-exponent}
	For all $\alpha\in (0,\half]$ and $0<r<1$, we have: as $\eps \searrow 0$,
	\begin{align*}
	\Pb(\wt\Ac^{2}_{\alpha}(\eps, r)) \le \eps^{ \xi^{2}(\alpha)+o(1)} \quad \text{and} \quad
		 \Pb(\wt\Ac_{\alpha}(\eps, r)) \le  \eps^{\xi(\alpha) +o(1)}.
	\end{align*}
\end{proposition}

\subsubsection*{Boundary arm exponents}

We now introduce the boundary two-arm and four-arm events.
We denote by $\wt\Lc_{\alpha}^+$ the BLS in $\Bc_1^+$ with intensity $\alpha$.

\begin{definition}[Boundary two-arm events for BLS]\label{def:2_arm_bls2}
	For $\alpha\in (0,\half]$ and $0<\eps<r<1$, let $\wt\Ac^{2,+}_{\alpha}(\eps, r)$ be the event that there is at least one outermost cluster in $\wt\Lc_{\alpha}^+$ crossing $(\Bc_r\setminus \Bc_\eps)^+$.
\end{definition}

\begin{definition}[Boundary four-arm event for BLS]\label{def:arm_bls2}
	For $\alpha\in (0,\half]$, and $0<\eps<r<1$, let $\wt\Ac^+_{\alpha}(\eps, r)$ be the event that there are at least two outermost clusters in $\wt\Lc_{\alpha}^+$ crossing  $(\Bc_r\setminus \Bc_\eps)^+$.
\end{definition}
Similarly, Definitions~\ref{def:2_arm_bls2} and~\ref{def:arm_bls2} can also be formulated for a CLE$_\kappa$ in $\Bc_1^+$, where $\alpha$ and $\kappa$ are related by \eqref{eq:alp-kap}, namely  $\wt\Ac^{2,+}_{\alpha}(\eps, r)$ (resp.\ $\wt\Ac_{\alpha}^+(\eps, r)$) is equal to the event that there are at least one (resp.\ two) loop(s) in the CLE$_\kappa$ crossing $(\Bc_r\setminus \Bc_\eps)^+$. 
For $\kappa\in (8/3,4]$, let 
\begin{align}\label{eq:b-arm-exponent}
\eta^{2,+}(\kappa) := \frac{8}{\kappa}-1 \quad \text{and} \quad \neweta(\kappa) := \frac{2(12-\kappa)}{\kappa}.
\end{align}
The values $\eta^{2,+}(\kappa)$ and $\neweta(\kappa)$ are also equal to the boundary two-arm and four-arm exponents of SLE$_\kappa$.
As before, for any $\alpha\in(0,1/2]$, we define $\xi^{2,+}(\alpha):=\eta^{2,+}(\kappa)$ and $\xi^+(\alpha):=\eta^+(\kappa)$, where $\alpha$ and $\kappa$ are related by \eqref{eq:alp-kap}. The following estimates on the boundary two-arm and four-arm events in the BLS also follow from \cite{GNQ2024c}.

\begin{proposition}[\cite{GNQ2024c}]
\label{prop:b-arm-exponent}
	For all $\alpha\in (0,\half]$ and $0<r<1$, we have: as $\eps \searrow 0$,
	\begin{align*}
			\Pb(\wt\Ac^{2,+}_{\alpha}(\eps, r)) \le \eps^{ \xi^{2,+}(\alpha)+o(1)} \quad \text{and} \quad
		\Pb(\wt\Ac^+_{\alpha}(\eps, r)) \le \eps^{ \newxi(\alpha)+o(1)}.
	\end{align*}
\end{proposition}

\section{Separation lemmas} \label{sec:sep}

The main goal of this section is to prove a separation lemma for two random walks inside a random walk loop soup, which we do in Section~\ref{subsec:a pair}, before considering, in Section~\ref{subsec:generalization}, extensions which are needed later in the paper. In general, such separation results state, roughly, the following: if two random sets are conditioned not to intersect each other, then with uniformly positive probability, their extremities are ``far apart'' (\emph{well-separated}, in a sense that needs to be made precise in each situation).
Note that a continuous version of our result, namely a separation lemma for Brownian motions inside a Brownian loop soup, was proved in a recent paper \cite{GLQ2022} by the first and third named authors and Li.

Separation properties are very useful and arise in many applications, and we only mention some of the most important instances where they appear. Such ideas were used to prove quasi-multiplicativity for arm probabilities in Bernoulli percolation, at or near criticality (see in particular the seminar work by Kesten \cite{Ke1987a} and subsequent works e.g.\ \cite{SW2001, No2008, SS2010, GPS2013, du2022sharp}). There is also a large amount of literature relying on well-separatedness for non-intersecting Brownian motions, see for instance  \cite{La1998,MR1901950,LV2012}. 

A general framework to establish separation was provided in the appendix of \cite{GPS2013}, that we are going to follow in our case. The proof is based on a suitable quantity called \emph{quality}, defined precisely in \eqref{eq:quality} below, which measures the separatedness of two random walks inside a random walk loop soup. It uses the following two crucial observations.
\begin{itemize}
	\item[(i)] At any scale, the quality at the next scale can be made macroscopic in a not-too-costly way, i.e. with a probability depending only on the quality at the current scale (see Lemma~\ref{lem:sep-1} below).
	\item[(ii)] If the quality is small enough, then intersection occurs at the next scale with high probability (Lemma~\ref{lem:sep-2}).
\end{itemize}
Then, separation can be obtained rather directly as a consequence of these two properties.

\subsection{A pair of non-intersecting random walks inside a RWLS} \label{subsec:a pair}

In this section, we first consider the case of $n=2$ random walks. Later, in Section~\ref{subsec:generalization}, we explain how it can be generalized easily to various, more complicated, situations.

We give the setup first. Let $1 \leq r \le R/2$ and $D \supseteq B_R$. Let $L_r$ be a loop configuration in $B_r$. Let $V_1$ and $V_2$ be two disjoint subsets of $B_r$, which both intersect $\partial B_r$, and let $z_i\in V_i\cap \partial B_r$, $i=1,2$. We view the quintuple $(L_r,V_1,V_2,z_1,z_2)$ as an ``initial configuration'', and we restrict to the case when $\Lambda(V_1,L_r)\cap V_2=\emptyset$ (recall the notation $\Lambda(\cdot,\cdot)$ from Section~\ref{subsec:RWLS}). In other words, there is no cluster of $L_r$ that intersects both $V_1$ and $V_2$.

Let $S^1$ and $S^2$ be two independent simple random walks, started from $z_1$ and $z_2$ respectively. Let $\Lc_{D}$ be a RWLS in $D$ with intensity $\alpha\in(0,1/2]$, which is independent of the random walks considered. Recall that we use $\Lc_{r,D} := \Lc_D \setminus \Lc_r$ to denote the loop soup made of the loops in $\Lc_{D}$ which are not entirely contained in $B_{r}$. We consider the frontier loop soup $\Lc^b_{r,D}$ induced by $\Lc_{r,D}$. For any $r\le s\le R$, let $\tau^i_{s}$ be the first time that $S^i$ hits $\partial B_s$, and let the {\it quality} at $s$ be defined by 
\begin{equation}\label{eq:quality}
	Q(s):=\sup \big\{ \delta \ge 0 \: : \: \Lambda \big( V_1\cup S^1[0,\tau^1_s]\cup B_{\delta s}(S^1_{\tau^1_s}),L_r\uplus\Lc^b_{r,D} \big) \cap \big( V_2\cup S^2[0,\tau^2_s]\cup B_{\delta s}(S^2_{\tau^2_s}) \big) = \emptyset \big\}. 
\end{equation}
Note that $Q(s) \in [0,1]$. Here, we use the same notion of quality as in \cite{SS2010} or \cite{GPS2013} (see the appendices of these two papers). Heuristically, $Q(s)$ measures how well each random walk is separated, at scale $s$, from the other random walk inside the loop soup. We can now state the separation lemma as follows (see Figure~\ref{fig:A} for an illustration).

\begin{proposition}[Separation lemma for a pair]\label{prop:sep}
	There exists a universal constant $c>0$ such that for all $1 \leq r \le R/2$ and $D \supseteq B_R$, for each initial configuration $(L_r,V_1,V_2,z_1,z_2)$, and for any intensity $\alpha\in (0,\half]$ of the loop soup under consideration,
	\begin{equation}\label{eq:Q>0}
		\Pb \big( Q(R)>1/10 \mid Q(R)>0 \big) \ge c.
	\end{equation}
	Moreover, let $\Dc$ be any event for the loop soup. If $\Dc$ is decreasing, then \eqref{eq:Q>0} still holds if one replaces $Q(R)$ by $\bar Q(R):=Q(R)\ind_\Dc$.
\end{proposition}

The quantity $\bar Q(s):=Q(s)\ind_\Dc$ is called the \emph{modified quality} at $s$ for $\Dc$. Note that $\{Q(R)>0\}$ is also a decreasing event.

\begin{figure}[t]
	\centering
	\includegraphics[width=.5\textwidth]{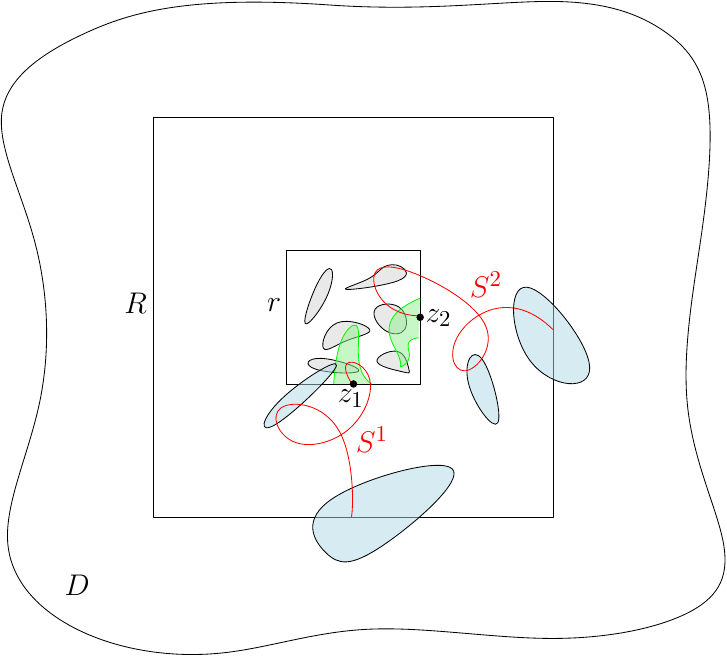}
	\caption{The non-intersection event $\{ Q(R)>0 \}$. The subsets $V_1$ and $V_2$ are shown in green, and clusters of $L_r$ in gray. The simple random walks $S^1$ and $S^2$ are drawn in red, the blue regions being clusters of $\Lc^b_{r,D}$ encountered by these two walks.}
	\label{fig:A}
\end{figure}

\begin{remark}\label{rmk:general-sep}
	It can be checked easily from the proof that the condition $R \geq 2r$ can be relaxed to $R \geq (1+\delta)r$, for any given $\delta>0$. In this case, the lower bound $1/10$ needs to be replaced by some suitable $c'(\delta)>0$, and we have
	$$\Pb \big( Q(R)>c' \mid Q(R)>0 \big) \ge c,$$
	where the constant $c > 0$ depends only on $\delta$. This also holds for later generalizations, in particular Propositions~\ref{prop:sep-jk}, \ref{prop:inv-sep} and \ref{prop:sep-disc}. 
\end{remark}

\begin{remark}\label{rmk:frontier-D}
	For technical reasons (cf Remark~\ref{rmk:frontier}, and the proof of Lemma~\ref{lem:sep-1} below), we often have to use the frontier loop soup $\Lc^b_{r,D}$ instead of the original loop soup $\Lc_{r,D}$. Let us discuss more precisely the impact of this. For a set $A$ and a configuration of loops $L$, a difference between $\Lambda(A,L)$ and $\Lambda(A,L^b)$ can only result from the existence of a loop $\gamma$ in $L$ such that $\gamma \cap A \neq \emptyset$ but $\gamma^b \cap A=\emptyset$. Therefore, if we include, in the event $\Dc$, the condition that all loops in $\Lc_{r,D}$ have a diameter at most $R/2$, then the modified quality $\bar Q(R):=Q(R)\ind_\Dc$ remains unchanged when $\Lc^b_{r,D}$ is replaced by $\Lc_{r,D}$. Indeed, a difference can only be caused by some big loop whose outer boundary encircles one of the simple random walks from $\partial B_r$ to $\partial B_R$: this is only possible if that loop has a diameter strictly larger than $R-r$, so in particular $> R/2$.
\end{remark}

Now, we turn to the proof of Proposition~\ref{prop:sep}, which relies on the next two lemmas. In order to state them precisely, we first need to introduce the notion of {\it relative quality} (see Eq.(A.4) of \cite{GPS2013}), which is used to handle the microscopic scales when going from $\partial B_r$ to $\partial B_{2r}$. Let $r':=rQ(r)$, $M:=\lfloor\log_2\left( r/r' \right)\rfloor = \lfloor\log_2\left( 1/Q(r) \right)\rfloor$ ($\geq 0$, since $Q(r) \leq 1$ by definition), and $L:=\lfloor\log_2\left(R/r\right)\rfloor$. First, set
\begin{equation*}
	r_i=\begin{cases}
		r, &i=0;\\
		r+2^{i-1}rQ(r), & 1\le i\le M;\\
		2^{i-M}r, & M+1\le i\le M+L.
	\end{cases}
\end{equation*}
For future use, note that by definition,
\begin{equation} \label{eq:bounds_rM}
	r_M \in \bigg(\frac{5}{4} r, \frac{3}{2} r \bigg] \quad \text{and} \quad r_{M+L} \in \bigg(\frac{1}{2} R, R \bigg].
\end{equation}
We then define the \emph{relative quality} as
\begin{equation*}
	Q^*(i):=\begin{cases}
		r_i Q(r_i)/(2^{i}rQ(r)), & 0\le i\le M;\\
		Q(r_i), & M+1\le i\le M+L.
	\end{cases}
\end{equation*}
Observe in particular that $Q^*(0)=1$. Additionally, we remind the reader that throughout the paper, we adopt the convention that integer parts are omitted from notations. In particular, we can assume all the scales $r'$ and $r_i$ ($0 \leq i \leq M+L$) to be integers.

We are now in a position to state the two intermediate lemmas mentioned above.
\begin{lemma}\label{lem:sep-1}
	For any $\rho>0$, there exists $c(\rho)>0$ such that for all $\alpha\in (0,\half]$, and $i \in \{0,\ldots,M+L-1\}$,
	\[
	\Pb( Q^*(i+1)>1/10 \mid Q^*(i)>\rho ) \geq c(\rho).
	\]
\end{lemma}

\begin{lemma}\label{lem:sep-2}
	For any $\delta>0$, there exists $\rho_0(\delta)>0$ such that for all $\rho\le \rho_0(\delta)$, $\alpha\in (0,\half]$, and $i \in \{0,\ldots,M+L-1\}$,
	\[
	\Pb( Q^*(i+1)>0 \mid 0<Q^*(i)<\rho ) \le \delta.
	\]
\end{lemma}

We first explain how Proposition~\ref{prop:sep} follows from these two lemmas, and we subsequently prove them.

\begin{proof}[Proof of Proposition~\ref{prop:sep}]
	We first provide a detailed proof for \eqref{eq:Q>0}. 
We can assume that $R=2^L r$ since it only costs a constant probability to extend the separation from scale $2^{L} r$ to $R$ (similar to the proof of Lemma~\ref{lem:sep-1}).
	Then by definition, $Q^*(M+L) = Q(r_{M+L}) = Q(R)$. Thus, it is equivalent to prove that 
	\[
	\Pb( Q^*(M+L)>0 ) \le c^{-1} \,\Pb( Q^*(M+L)>1/10 ).
	\]
	Note that $Q^*(i)>0$ implies that $Q^*(j)>0$ for all $j \in \{0, \ldots, i\}$. Consider an arbitrary $\delta > 0$, that we explain how to choose later. Therefore, by Lemma~\ref{lem:sep-2}, for all $i \in \{1,\ldots,M+L\}$, all $\rho\le \rho_0(\delta)$,
	\begin{align*}
		\Pb( Q^*(i)>0 ) &\le \Pb( Q^*(i-1)>\rho )+ \Pb(0<Q^*(i-1)<\rho, \, Q^*(i)>0) \\
		& \le \Pb( Q^*(i-1)>\rho )+ \delta\, \Pb(Q^*(i-1)>0).
	\end{align*}
	Iterating the above inequality, we obtain 
	\begin{equation}\label{eq:sq1}
		\Pb( Q^*(M+L)>0 ) \le \sum_{i=0}^{M+L-1} \delta^{i}\, \Pb( Q^*(M+L-1-i)>\rho ) +\delta^{M+L}.
	\end{equation}
	By using Lemma~\ref{lem:sep-1} repeatedly, we have
	\begin{equation}\label{eq:sq2}
		\Pb( Q^*(M+L-1-i)>\rho ) \le c(\rho)^{-1} \Pb( Q^*(M+L-i)>1/10 ) \le c(\rho)^{-1} \underline{c}^{-i} \,
		\Pb( Q^*(M+L)>1/10 ),
	\end{equation}
	where $\underline{c}:=c(1/10)$ is provided by Lemma~\ref{lem:sep-1} with $\rho=1/10$.
	By definition $Q^*(0)=1$, so using Lemma~\ref{lem:sep-1} again, we obtain
	\begin{equation}\label{eq:sq3}
		\Pb( Q^*(M+L)>1/10 ) \ge \underline{c}^{M+L}.
	\end{equation}
	Plugging \eqref{eq:sq2} and \eqref{eq:sq3} into \eqref{eq:sq1}, we obtain that for all $\delta>0$ and $\rho\le \rho_0(\delta)$,
	\[
	\Pb( Q^*(M+L)>0 ) \le \Pb( Q^*(M+L)>1/10 )  \left( \frac{1}{c(\rho)}\sum_{i=0}^{M+L-1} \left(\frac{\delta}{\underline{c}}\right)^{i} \, + \left(\frac{\delta}{\underline{c}}\right)^{M+L} \right).
	\]
	Thus, we can conclude the proof of \eqref{eq:Q>0} by taking $\delta=\underline{c}/2$, and then $\rho=\rho_0(\delta)$.
	
	Finally, the same proof as above can be followed to obtain the result for $\bar Q(R) = Q(R)\ind_\Dc$ instead of $Q(R)$. For this purpose, we observe that the strategy of proof for Lemma~\ref{lem:sep-1} remains valid by using the FKG inequality, while the proof of Lemma~\ref{lem:sep-2} is exactly the same.
\end{proof}

\begin{figure}[t]
	\centering
	\includegraphics[width=1\textwidth]{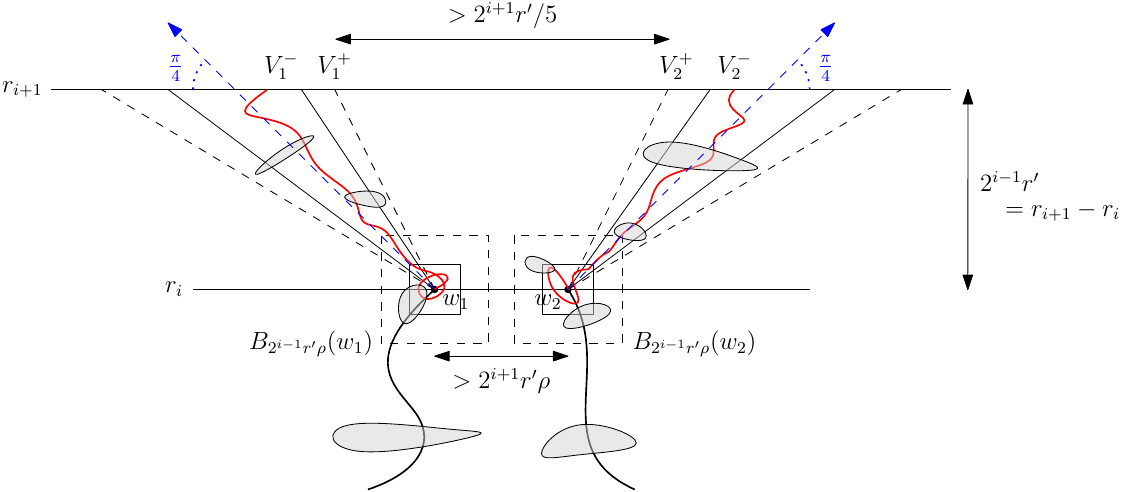}
	\caption{This figure depicts the construction used in the proof of Lemma~\ref{lem:sep-1}, to go from a relative quality which is possibly microscopic (at least $\rho$) on scale $r_i$, to a macroscopic one (at least $1/10$) on scale $r_{i+1}$. For each $j=1,2$, the two boxes around $w_j$ have radii $2^{i-2}r'\rho$ (in solid line) and $2^{i-1}r'\rho$ (in dashed line), and the two cones originating from $w_j$ are $V_j^-$ (solid) and $V_j^+$ (dashed), respectively.}
	\label{fig:cones}
\end{figure}

\begin{proof}[Proof of Lemma~\ref{lem:sep-1}]
	We first consider the case $0 \leq i \leq M-1$. Let $t^j_i := \tau^j_{r_i}$ and $w_j=S^j_{t^j_i}$, for $j = 1,2$. As illustrated on Figure~\ref{fig:cones}, it is easy to see that for each $j =1,2$, one can find two infinite cones $V_j^-$ and $V_j^+$, both with apex $w_j$, having small enough opening angles such that the following properties are verified:
	\begin{enumerate}
		\item $V_j^-$ is contained in $V_j^+$, and its opening angle is half that of $V_j^+$,
		
		\item $V_j^-$ and $V_j^+$ have the same axis of symmetry, with angle $\pi/4$ to the side of $\partial B_{r_i}$ containing $w_j$,
		
		\item $\partial V_j^+\cap \partial B_{r_i}=\{ w_j \}$,
	\end{enumerate}
	and so that furthermore,
	\begin{enumerate} \setcounter{enumi}{3}
		\item $\dist(V_1^+ \cap \partial B_{r_{i+1}}, V_2^+ \cap \partial B_{r_{i+1}}) > 2^{i+1} r'/5$. \label{it:cond_V1V2}
	\end{enumerate}
	Note that the choice of these four cones does not depend on $\rho$. In addition, by the definition \eqref{eq:quality} of $Q$, we have
	$$|w_1-w_2| \geq 2 r_i Q(r_i).$$
	If we now assume that $Q^*(i)>\rho$ holds, by definition $r_i Q(r_i) = (2^i r') Q^*(i) > (2^i r') \rho$, so $|w_1-w_2| > 2^{i+1} r' \rho$. Let 
	\[
	U_j:=\Lambda \big( B_{2^{i-2}r'\rho}(w_j) \cup \big( V_j^-\cap B_{r_{i+1}} \big), \Lc^b_{r,D} \big), \quad j=1,2,
	\]
	and consider the events
	\[
	E_1 := \bigcap_{j=1}^2 \big\{ U_j\subseteq B_{2^{i-1} r'\rho}(w_j) \cup V_j^+ \big\}, \quad \text{and } E_2 := \bigcap_{j=1}^2 \big\{ S^j[\tau^j_{r_i},\tau^j_{r_{i+1}}] \subseteq B_{2^{i-2}r'\rho}(w_j)\cup V_j^- \big\}.
	\]
	Let us give lower bounds on their respective probabilities. For this purpose, we introduce, for $j=1, 2$, the concentric boxes
	$$B_{2^{i-1+k}r'\rho}(w_j), \quad k = 0, \ldots, K(\rho):= \lceil \log_2(\rho^{-1}) \rceil+1$$
(note that in particular, the last box contains the whole cone $V^-_j$).
	For each such $k$, let $F_k$ be the event that all clusters in $\Lc_{D}^b$ intersecting
	$$\big( B_{2^{i-1+k}r'\rho}(w_1)\cup B_{2^{i-1+k}r'\rho}(w_2) \big) \cap B_{r_{i+1}}$$
	have a diameter smaller than $2^{i-1+k}r'\rho/100$. Note that $F_k$ is a decreasing event for each $k$, so the FKG inequality yields
	$$\Pb( E_1 )\ge \Pb\Bigg( \bigcap_{k=0}^{K(\rho)} F_k \Bigg)\ge \prod_{k=0}^{K(\rho)} \Pb( F_k ).$$
	Using Lemma~\ref{lem:cluster-size-2}, and again the FKG inequality, there exists a universal constant $C>0$ such that for all $k$ as before, $\Pb( F_k )\ge C$. Therefore,
	\begin{equation} \label{eq:lowerbd_E1}
		\Pb( E_1 ) \ge \delta_1(\rho) := C^{K(\rho)+1}.
	\end{equation}
	Next, we deal with $E_2$. For $j=1,2$, let $V_j^{--}$ be the cone with apex $w_j$, contained in $V_j^-$, with the same axis of symmetry as that cone, but half of its opening angle. Note that one can find some universal positive constant $\bar C>0$ such that for each $j=1,2$, each of the following events occurs with probability at least $\bar C$: 
	\begin{itemize}
		\item a simple random walk started from $w_j$ stays in $B_{2^{i-2}r'\rho}(w_j)\cup V_j^{--}$ until hitting $\partial B_{2^{i-1}r'\rho}(w_j)$;
		\item a simple random walk started from some point on $V_j^{--}\cap\partial B_{2^{i-1+k}r'\rho}(w_j)$ remains in $V^-_j$ before reaching $\partial B_{2^{i-1+k+1}r'\rho}(w_j)$ ($k \in \{0, \ldots, K(\rho)\}$).
	\end{itemize}
	By the Strong Markov property for a simple random walk, we conclude that (we require, for $j=1$ and $j=2$ simultaneously, the first bullet above to occur, and the second bullet to occur for all $k = 0, \ldots, K(\rho)$)
	\begin{equation} \label{eq:lowerbd_E2}
		\Pb( E_2 ) \ge \delta_2(\rho):=(\bar C^{K(\rho)+2})^2.
	\end{equation}
	Moreover, since the random walk loop soup and the random walks are independent, we know that $E_1$ and $E_2$ are independent. 
	
	It then follows from our construction with cones that 
	\begin{align*}
		\Pb( Q^*(i+1)>1/10 ) &\ge \Pb(\{ Q^*(i)>\rho \} \cap E_1\cap E_2 )\\
		&\ge \Pb(\{ Q^*(i)>\rho \} \cap E_1 ) \, \Pb( E_2 )\\
		&\ge \Pb( Q^*(i)>\rho )\, \Pb(E_1) \,\Pb(E_2)\\
		&\ge \Pb( Q^*(i)>\rho )\, \delta_1(\rho) \,\delta_2(\rho),
	\end{align*}
	where in the second inequality we separate $E_2$ from the other two events by using the strong Markov property of a simple random walk, we then separate $\{ Q^*(i)>\rho \}$ from $E_1$ on the next line by using the FKG inequality (since both events are decreasing with respect to the loop soup), and we finally use the lower bounds provided by \eqref{eq:lowerbd_E1} and \eqref{eq:lowerbd_E2}. Note that here, we also used that the condition \ref{it:cond_V1V2} above on $V_1^+$ and $V_2^+$ ensures that $r_{i+1} Q(r_{i+1}) > 2^i r' / 5$, so
	$$Q^*(i+1) = \frac{r_{i+1} Q(r_{i+1})}{2^{i+1} r'} > \frac{1}{10}.$$
	This finally gives the desired lower bound, with $c(\rho) := \delta_1(\rho) \,\delta_2(\rho)$.
	
	We can proceed in a similar way in the remaining cases $i=M$ and $M+1 \leq i \leq M+L-1$, and we skip the details since only a small tweak of the construction is needed. In the former case, \eqref{eq:bounds_rM} ensures that $r_{M+1} - r_M \geq r/2 = r_{M+1}/4$, and in the latter case, $r_{i+1} - r_i = r_{i+1}/2$. Hence, in both cases, $r_{i+1} - r_i \geq r_{i+1}/4$, and we can thus ensure that the intersections of the two cones $V_1^+$ and $V_2^+$ with $\partial B_{r_{i+1}}$ are sufficiently far apart. More precisely, we can require that $\dist(V_1^+ \cap \partial B_{r_{i+1}}, V_2^+ \cap \partial B_{r_{i+1}}) > r_{i+1}/5$, which then yields $Q^*(i+1) = Q(r_{i+1}) > 1/10$ at the end of the construction. This completes the proof of the lemma.
\end{proof}

\begin{proof}[Proof of Lemma~\ref{lem:sep-2}]
	Write $d_i := r_{i+1}-r_i$, and for $j = 1,2$, $t^j_i := \tau^j_{r_i}$. By definition of $Q$ and $Q^*$, on the event $\{0<Q^*(i)<\rho\}$, one has necessarily 
	$$B_{2 \rho d_i} \big( S^1_{t^1_i} \big) \cap \Lambda^2_{r_i} \neq \emptyset \quad \text{ or } \quad B_{2 \rho d_i} \big( S^2_{t^2_i} \big) \cap \Lambda^1_{r_i} \neq \emptyset,$$
	with
	$$\Lambda^j_{r_i} = \Lambda \big( V_j \cup S^j[0, t^j_i]\cup B_{2 \rho d_i} \big( S^j_{t^j_i} \big), L_r\uplus\Lc^b_{r,D} \big), \quad j = 1,2.$$
	Indeed, it is easy to check that in each of the following three cases, the inequality $Q^*(i)<\rho$ implies that $r_i Q(r_i) < 2 \rho d_i$.
	\begin{itemize}
		\item $0 \leq i \leq M-1$: $d_i = 2^{i-1} r'$ and $Q^*(i) = r_i Q(r_i)/(2^i r')$, so $r_i Q(r_i) < 2^i r' \rho = 2 \rho d_i$;
		
		\item $i=M$: $d_M = 2r - r_M \geq r/2$ (using \eqref{eq:bounds_rM}) and $Q^*(M) = r_M Q(r_M)/(2^M r')$, so $r_M Q(r_M) < 2^M r' \rho \leq r \rho \leq 2 \rho d_M$ (since $2^M \leq r/r'$);
		
		\item $M+1 \leq i \leq M+L$: $d_i = 2^{i-M} r = r_i$ and $Q^*(i) = Q(r_i)$, so $r_i Q(r_i) < \rho r_i = \rho d_i$.
	\end{itemize}
	Therefore, using the strong Markov property (at time $t^1_i$ for $S^1$ on the one hand, and at time $t^2_i$ for $S^2$ on the other hand), we obtain that $\Pb( Q^*(i+1)>0 \mid 0<Q^*(i)<\rho )$ is at most twice the probability that over the time interval $[\tau^1_{r_i},\tau^1_{r_{i+1}}]$, $S^1$ does not disconnect $B_{4 \rho d_i}(S^1_{t^1_i})$ from infinity. This latter probability is $O((4\rho d_i/r_{i+1})^{1/4}) = O(\rho^{1/4})$ by Lemma~\ref{lem:disconnection-prob} (using also $d_i \leq r_{i+1}$), which completes the proof.
\end{proof}

\begin{remark}\label{rmk:conditionl RW}
	Later we often use a more general separation result, where the simple random walks are replaced by conditional walks, which are conditioned on hitting $\partial B_R$ before some given subset of $B_{r/2}$. One can easily check that Lemmas~\ref{lem:sep-1} and~\ref{lem:sep-2} remain true for such conditional random walks, and thus Proposition~\ref{prop:sep} holds in this case as well. Furthermore, it remains true if one replaces only one of the simple random walks by a conditional one. Such generalizations hold similarly for Propositions~\ref{prop:sep-jk}, \ref{prop:inv-sep} and \ref{prop:sep-disc} below.  
\end{remark}

\subsection{Generalizations: packets of random walks and reversed versions} \label{subsec:generalization}

We now generalize the results from the previous section in several directions. First, we state a separation lemma for two packets of random walks inside a RWLS, and we then include a reversed version of this result. Finally, we analyze non-disconnection in the particular case of only one packet of random walks. The framework followed in the proof of Proposition~\ref{prop:sep} is very robust, and it can thus be employed here to show that all the generalizations considered hold true (in the same way as for the classical separation lemmas for random walks). Therefore, we only state the results that are needed later, and we omit their proofs.

\subsubsection*{Two packets of random walks}

First, we need to adapt some of the notations introduced at the beginning of Section~\ref{subsec:a pair}. For $j,k\ge 1$, we consider the initial configuration $(L_r,V_1,V_2,\bar x,\bar y)$ as before, with the requirement that $\Lambda(V_1,L_r)\cap V_2=\emptyset$, where $\bar x=(x_1,\ldots,x_j)$ is a vector of $j$ vertices in $V_1\cap \partial B_r$ and $\bar y=(y_1,\ldots,y_k)$ is a vector of $k$ vertices in $V_2\cap \partial B_r$ (some of the points in $\bar x$ may coincide, and similarly with $\bar y$). For $r\le s\le R$, let $\Pi^1_s$ (resp.\ $\Pi^2_s$) be the union of $j$ (resp.\ $k$) independent simple random walks started, respectively, from each of the $j$ points in $\bar x$ (resp.\ each of the $k$ points in $\bar y$), and stopped upon reaching $\partial B_s$. We require that all of these simple random walks are independent, and also that they are independent of the loop soup $\Lc_{D}$.

Analogously to Section~\ref{subsec:a pair}, the quality at $s$ is then defined as
\begin{align} \notag
	Q^{j,k}(s) := \sup & \Big\{ \delta \ge 0 \: : \: \Lambda \Big( V_1 \cup \Pi^1_s \cup \Big( \cup_{z \in \Pi^1_s \cap \partial B_s} B_{\delta s}(z) \Big), L_r \uplus \Lc^b_{r,D} \Big)\\
	& \hspace{4cm} \cap \Big( V_2 \cup \Pi^2_s \cup \Big( \cup_{z \in \Pi^2_s \cap \partial B_s} B_{\delta s}(z) \Big) \Big) = \emptyset \Big\}. \label{eq:Qjk}
\end{align}
In this definition, we consider the unions of, respectively, $j$ and $k$ balls, all with radius $\delta s$, centered on the hitting points along $\partial B_s$ of each of the $j$ random walks in $\Pi^1_s$, and of each of the $k$ random walks in $\Pi^2_s$.

The following separation result holds for the two packets of random walks.

\begin{proposition}[Separation lemma for two packets]\label{prop:sep-jk}
	For all $j,k\ge 1$, there exists a constant $c(j,k)>0$ such that the following holds. For all $1 \leq r \le R/2$ and $D \supseteq B_R$, for each initial configuration $(L_r,V_1,V_2,\bar x,\bar y)$ with $\bar x=(x_1,\ldots,x_j)$ in $V_1\cap \partial B_r$ and $\bar y=(y_1,\ldots,y_k)$ in $V_2\cap \partial B_r$, and for any intensity $\alpha\in (0,\half]$ of the loop soup under consideration,
	\begin{equation} \label{eq:sep_twopackets}
		\Pb \big( Q^{j,k}(R)>1/(10(j+k)) \mid Q^{j,k}(R)>0 \big) \ge c.
	\end{equation}
	Moreover, \eqref{eq:sep_twopackets} also holds with $Q^{j,k}(R)$ replaced by $\bar Q^{j,k}(R):=Q^{j,k}(R)\ind_\Dc$, for any event $\Dc$ which is decreasing (for the loop soup).
\end{proposition}

\subsubsection*{Reversed separation lemma}

We now derive a reversed separation lemma, which is an inward analog of Proposition~\ref{prop:sep-jk}. For simplicity, we often use the same notations as earlier, even though we are in a different setting, but this should not create any confusion. Let $L$ be a loop configuration in $D\setminus B_R$. Let $V_1$ and $V_2$ be two disjoint subsets of $D\setminus \mathring{B}_R$, which both intersect $\partial B_R$, and assume that $\Lambda(V_1,L) \cap V_2 = \emptyset$. Let $\bar x=(x_1,\ldots,x_j)$ and $\bar y=(y_1,\ldots,y_k)$ be vectors of $j$ and $k$ points, respectively in $V_1\cap \partial B_R$ and $V_2 \cap \partial B_R$. For $r\le s\le R$, let $\Pi^1_{s}$ be the union of $j$ independent simple random walks started from the $j$ points in $\bar x$, respectively, stopped upon reaching $\partial B_s$, and conditioned to stay in $D$; and similarly, let $\Pi^2_s$ be the union of $k$ walks started from $\bar y$, also stopped when they hit $\partial B_s$ and conditioned to stay in $D$.

Let $\Lc_{D}$ be the RWLS in $D$ with intensity $\alpha\in(0,1/2]$, independent of all the random walks considered. We write $\Lc^{B_R}_{D} := \Lc_{D} \setminus \Lc_{B_R^c}$ for the set of all loops in $\Lc_{D}$ intersecting $B_R$. The \emph{reversed quality} is then defined as 
\begin{align} \notag
	Q^{j,k}(s) := \sup & \Big\{ \delta \ge 0 \: : \: \Lambda \Big( V_1 \cup \Pi^1_s \cup \Big( \cup_{z \in \Pi^1_s \cap \partial B_s} B_{\delta s}(z) \Big), L \uplus (\Lc^{B_R}_{D})^b \Big)\\
	& \hspace{4cm} \cap \Big( V_2 \cup \Pi^2_s \cup \Big( \cup_{z \in \Pi^2_s \cap \partial B_s} B_{\delta s}(z) \Big) \Big) = \emptyset \Big\}. \label{eq:Qrjk}
\end{align}

In this setting, we have the result below.

\begin{proposition}[Reversed separation lemma for two packets] \label{prop:inv-sep}
	For all $j,k\ge 1$, there exists a constant $c(j,k)>0$ such that the following holds. For all $1 \leq r \le R/2$ and $D \supseteq B_R$, each initial configuration $(L,V_1,V_2,\bar x,\bar y)$ with $\bar x=(x_1,\ldots,x_j)$ in $V_1\cap \partial B_R$  and $\bar y=(y_1,\ldots,y_k)$ in $V_2\cap \partial B_R$, and any intensity $\alpha\in (0,\half]$ for the loop soup under consideration,
	\begin{equation} \label{eq:sep_reversed}
		\Pb \big( Q^{j,k}(r)>1/(10(j+k)) \mid Q^{j,k}(r)>0 \big) \ge c.
	\end{equation}
	Furthermore, \eqref{eq:sep_reversed} also holds when $Q^{j,k}(r)$ is replaced by $\bar Q^{j,k}(r):=Q^{j,k}(r)\ind_\Dc$, for any decreasing event $\Dc$.
\end{proposition}

Let us mention that in Section~\ref{sec:arm-lwrs}, we use both Propositions~\ref{prop:sep-jk} and \ref{prop:inv-sep} in the particular case $j=k=2$ (two packets, containing two random walks each).

\subsubsection*{One packet of random walks}

Finally, we consider a specific event for a single packet of random walks, which takes into account the combined effect of non-intersection and non-disconnection. Let $1 \leq r \le R/2$ and $D\supseteq B_R$. Let $L_r$ be a set of loops in $B_r$. Let $V$ be a subset of $B_r$ with $V\cap \partial B_r\neq\emptyset$, and let $U$ be another subset of $B_r$, such that $U\cap \Fill(\Lambda(V,L_r)) = \emptyset$; that is, if we consider the union of $V$ and all the clusters in $L_r$ intersecting $V$, then this set does not intersect nor disconnect $U$. Let $\bar x=(x_1,\ldots,x_j)$ consist of $j$ vertices in $V \cap \partial B_r$ (not necessarily all distinct). In a similar way as before, we refer to a quadruple $(L_r,U,V,\bar x)$ satisfying the above conditions as an initial configuration.

For $r\le s\le R$, let $\Pi_s$ be the union of $j$ independent simple random walks started from $x_1,\dots,x_j$, respectively, and each stopped upon reaching $\partial B_s$. Consider an independent RWLS $\Lc_{r,D}$ with intensity $\alpha\in (0,\half]$. 
In this case, the quality at scale $s$ is, roughly speaking, the maximal (rescaled) radius of the balls that one can attach at the endpoints of the random walks such that, together with $V$, they do not intersect nor disconnect $U$ inside the whole loop soup. Formally, it is defined as
\begin{equation}\label{eq:quality-disc}
	Q^j(s) := \sup \Big\{ \delta \ge 0 \: : \: U \cap \Fill \Big( \Lambda \Big( \Pi_s \cup \Big( \cup_{z \in \Pi_s \cap \partial B_s} B_{\delta s}(z) \Big), L_r \uplus \Lc_{r,D}^b \Big) \Big) = \emptyset \Big\},
\end{equation}
where we recall the notation $\Fill(A)$ for the filling of $A$, introduced in Section~\ref{subsec:notation}.

The separation lemma for one packet and non-disconnection can then be stated as follows.

\begin{proposition}[Separation lemma for one packet] \label{prop:sep-disc}
	For all $j\ge 1$, there exists a constant $c(j)>0$ such that one has the following. For all $1 \leq r \le R/2$, $D\supseteq B_R$, each initial configuration $(L_r,U,V,\bar x)$ with $\bar x=(x_1,\ldots,x_j)$ in $V \cap \partial B_r$, and any intensity $\alpha\in (0,\half]$ of the loop soup under consideration,
	$$\Pb \big( Q^{j}(R)>1/(10j) \mid Q^{j}(R)>0 \big) \ge c.$$
	Moreover, as for earlier separation results, the above also holds with $Q^{j}(R)$ replaced by $\bar Q^{j}(R):=Q^{j}(R)\ind_\Dc$, for any event $\Dc$ (concerning the loop soup) which is decreasing.
\end{proposition}

Proposition~\ref{prop:sep-disc} can be established in exactly the same way as Proposition~\ref{prop:sep}, so we omit the proof. We also remark that a reversed version of Proposition~\ref{prop:sep-disc}, similar to Proposition~\ref{prop:inv-sep}, holds as well. Later, we use Proposition~\ref{prop:sep-disc} in the special case $j=2$.

\section{Arm events in the RWLS} \label{sec:arm-lwrs}

In this section, we introduce a key geometric event for the random walk loop soup, called \emph{four-arm event}. It requires the existence of two \emph{distinct} clusters of loops crossing a given annulus, and it is analogous to the classical polychromatic four-arm event (with alternating types occupied and vacant) for Bernoulli percolation in two dimensions.

We first define precisely the four-arm event in Section~\ref{sec:def_fourarm}, and we mention some of its elementary properties, following directly from its definition. We then prove, in Section~\ref{sec:locality}, a useful locality result for the four-arm event (Proposition~\ref{prop:locality}), allowing us to discard the large loops, which wander too far from the annulus. Next, we derive the corresponding result in the reverse direction, for a version of the four-arm event which is truncated inward, in Section~\ref{sec:locality2} (Proposition~\ref{prop:in-locality}).
Finally, in Section~\ref{subsec:other_arm}, we also provide similar results for the (interior and boundary) two-arm events, and the boundary four-arm events.

All of these results are achievable thanks to the various separation lemmas derived in Section~\ref{sec:sep}. Indeed, separation plays an instrumental role throughout the proofs, making it possible for us to perform local surgery repeatedly with random walk loops.

\subsection{Four-arm event} \label{sec:def_fourarm}

The (interior) four-arm event, which is a central protagonist in the remainder of the paper, is defined as follows. Recall that for any $D \subseteq \Zb^2$, we use $\Lc_D$ to denote the random walk loop soup in $D$ with intensity $\alpha\in (0,\half]$.

\begin{definition}[Four-arm event for RWLS] \label{def:arm event for rwls}
Let $1 \leq l<d$, $z \in \Zb^2$ and $D\subseteq \Zb^2$. The \emph{four-arm event} in the annulus $A_{l,d}(z)$, for a RWLS configuration $\Lc_D$, is defined as $\Ac_{\Lc_D}(z;l,d) := \{$there are (at least) two outermost clusters in $\Lc_D$ crossing $A_{l,d}(z)\}$. Moreover, we define the \emph{truncated} four-arm event, incorporating further restriction on the sizes of the crossing clusters, to be $\overrightarrow\Ac_{\Lc_D}(z;l,d) := \Ac_{\Lc_D}(z;l,d) \cap \{ \Lambda(B_l(z), \Lc_D) \subseteq B_{2d}(z) \}$.
\end{definition}

Later on, we also need to consider four-arm events for random loop configurations other than $\Lc_D$. For these events, we just replace the subscript $\Lc_D$ in Definition~\ref{def:arm event for rwls} by the corresponding loop configuration. For brevity, we also write $\Ac_D(z;l,d)$ and $\overrightarrow\Ac_D(z;l,d)$ for $\Ac_{\Lc_D}(z;l,d)$ and $\overrightarrow\Ac_{\Lc_D}(z;l,d)$, respectively.
We denote $\Ac_{\mathrm{loc}}(z;l,d) := \Ac_{B_{2d}(z)}(z;l,d)$, and we think about it as a ``localized'' arm event. As always, we omit the parameter $z$ in the notations if $z=0$. For example, we write $\Ac_{\loc}(l,d)$ for $\Ac_{\loc}(0;l,d)$. Furthermore, we suppress the dependence of various quantities on the intensity $\alpha$ of the random walk loop soup, as we did previously.

\begin{remark}[Monotonicity]\label{rmk:mono}
	By definition, the following inclusions are automatically satisfied, for all $1 \leq l<d$, $z \in \Zb^2$, and $D \subseteq \Zb^2$:
	\begin{equation}\label{eq:arm-inclusion}
		\overrightarrow\Ac_{D}(z;l,d) \subseteq \Ac_{D}(z;l,d) \quad \text{ and } \quad \overrightarrow\Ac_{D}(z;l,d) \subseteq \Ac_{\loc}(z;l,d).
	\end{equation}
	Moreover, the arm event is monotone in the sense that for any $1 \leq l_1\le l_2<d_2\le d_1$,
	\begin{equation}\label{eq:monotone}
		\Ac_D(z;l_1,d_1)\subseteq \Ac_D(z;l_2,d_2).
	\end{equation}
	But throughout the proofs, we have to be careful that the event $\Ac_{D}(z;l,d)$ is \emph{not} monotone in the associated domain $D$. Indeed, enlarging $D$ provides additional loops: these loops may expand the clusters and help them to cross the annulus $A_{l,d}(z)$, but they may also merge two distinct clusters into a single one. Thus, from our convention, $\Ac_{\loc}(z;l,d)$ is not monotone in $d$ (the implicit domain under consideration, i.e. $D = B_{2d}(z)$, depending on $d$). However, for future use, note that we still have $\Ac_{\loc}(z;l_1,d) \subseteq \Ac_{\loc}(z;l_2,d)$ for all $1 \leq l_1\le l_2<d$.
\end{remark}

\subsection{Locality property} \label{sec:locality}

In this section, we prove that the probability of a ``global'' arm event $\Ac_D(z;l,d)$ can be bounded by that of the associated ``local'' one $\Ac_{\loc}(z;l,d)$, up to some uniform multiplicative constant. In fact, we prove a stronger result, namely that $\Pb( \overrightarrow\Ac_D(z;l,d) \mid \Ac_D(z;l,d) )\ge C$ for some universal constant $C>0$.

\begin{proposition}\label{prop:locality}
	There is a universal constant $C>0$ such that for all $z\in\Zb^2$, $1 \leq l\le d/2$, $D \supseteq B_{2d}(z)$, and any intensity $\alpha\in (0,\half]$,
	\begin{equation}\label{eq:local-1}
	\Pb(\Ac_D(z;l,d))\le C\, \Pb( \overrightarrow\Ac_D(z;l,d) ).
	\end{equation}
	As a consequence, and by \eqref{eq:arm-inclusion} as well,
	\begin{equation}\label{eq:local-2}
		\Pb(\Ac_D(z;l,d)) \le C\, \Pb(\Ac_{\loc}(z;l,d)).
	\end{equation}
\end{proposition}
\begin{remark}
	We can observe that the reverse inequality of \eqref{eq:local-2} is not true. In order to see this, we can take $z=0$ and  $D=B_R$: as $R\rightarrow \infty$,
	\[
	\Pb\big(\Ac_{B_R}(l,d)\big)\le \Pb\big( \text{there is no loop in $\Lc_{R}$ that disconnects $B_l$ from $\infty$} \big)\rightarrow 0.
	\]
\end{remark}

\begin{remark}
	As one can easily check from the proof of Proposition~\ref{prop:locality}, the scale $2d$ is of course not essential in the definitions of $\overrightarrow\Ac_D(z;l,d)$ and $\Ac_{\loc}(z;l,d)$, and it can be replaced by $(1+\delta)d$, for any given $\delta>0$ (with now some constant $C(\delta) > 0$ in \eqref{eq:local-1} and \eqref{eq:local-2}).
\end{remark}

In the remainder of this section, we always suppose that $z=0$ for simplicity, and we assume that $1 \leq l\le d/2$ and $D \supseteq B_{2d}$ are given. We first introduce some notation, and then establish a sequence of intermediate lemmas, before finally proving Proposition~\ref{prop:locality} itself.

\subsubsection*{Role of large ``crossing'' loops}

We now give definitions which are specific to the proof of Proposition~\ref{prop:locality} (and which are thus not used after this section, or used but with a different meaning). First, if $d+1 \leq k \leq 2d$, a loop in $D$ is called a \emph{$(d,k)$-crossing loop} if it crosses $A_{d, k}$, but it does not disconnect $B_d$ from $\partial B_k$. Let $\cl_{d,k}$ be the set of $(d,k)$-crossing loops. Later we use this definition in the case $k = 1.8d$, primarily: a $(d,1.8d)$-crossing loop is simply called a \emph{crossing loop}, and we write $\cl := \cl_{d,1.8d}$. We say that an outermost cluster $\Cc$ across $A_{l,d}$ is \emph{good} if $\Cc\setminus\cl$ still crosses $A_{l,d}$, and \emph{bad} otherwise. In other words, a bad cluster is such that it crosses $A_{l,d}$ only thanks to a large crossing loop, connecting it to the complement of $B_{1.8d}$. Recall that in this paper, such a notation actually means $B_{[1.8d]}$ (and similarly later, when we consider the scales $1.2d$, $1.5d$, and so on). Note that if such a cluster $\Cc$ is bad, then at least one of the loops in $\cl$ is enough to produce a crossing of $A_{l,d}$.

We consider the following three (disjoint) cases, depending on the number of bad clusters that come into play:
\begin{itemize}
	\item $\Ec_0 := \{\Ac_D(l,d)$ occurs without bad clusters$\}$,
	
	\item $\Ec_1 := \{\Ac_D(l,d)$ occurs with only one bad cluster$\}$,
	
	\item $\Ec_2 := \{\Ac_D(l,d)$ occurs with at least two bad clusters$\}$.
\end{itemize}
In the following, we also write $\Ec_i(\Lc')$ for any loop configuration $\Lc'$ in place of $\Lc_D$.

Note that
\begin{equation}\label{eq:E123}
	\Pb( \Ac_D(l,d) ) \le \Pb( \Ec_0 ) + \Pb(  \Ec_1 )+\Pb(  \Ec_2 ),
\end{equation}
hence it suffices to show that for each $i = 0,1,2$,
\begin{equation}\label{eq:Eci}
	\Pb( \Ec_i ) \lesssim \Pb \big( \overrightarrow\Ac_D(l,d) \big).
\end{equation}
Below, we do it first for $i=0$, in Lemma~\ref{lem:E0}. Next, we take care of $i=2$, and then $i=1$, in Lemmas~\ref{lem:E1} and \ref{lem:E2} respectively (these last two cases are obtained through a combination with Lemma~\ref{lem:Palm}).

In the following proofs, we perform successive surgery arguments with loops. For this purpose, we need to discover the loop soup $\Lc_D$ in concentric annuli, and we use the notation $\Lc_{k,D} := \Lc_D\setminus \Lc_k$, for any $1 \leq k \le 2d$.

\subsubsection*{Case $i=0$: no crossing loop needed}

We begin with $\Ec_0$, which is the simplest case as it only requires an application of the FKG inequality.

\begin{lemma}\label{lem:E0}
	There exists a universal constant $C>0$ such that for all $1 \leq l\le d/2$, $D \supseteq B_{2d}$, and any intensity $\alpha\in (0,\half]$,
	\[
	\Pb( \Ec_0 ) \le C \,\Pb \big(  \overrightarrow\Ac_D(l,d) \big).
	\]
\end{lemma}
\begin{proof}
	Note that $\Ec_0$ is decreasing in the loop soup $\Lc_{1.8d,D}$ ($= \Lc_D \setminus \Lc_{1.8d}$). Indeed, in that loop soup, only crossing loops can reach $B_d$, and such loops can only spoil $\Ec_0$. Let $G$ be the event that if we consider the frontier loop soup $\Lc_{1.8d,D}^b$, all clusters intersecting $B_{1.8d}$ have a diameter smaller than $d/100$. Clearly, $G$ is also a decreasing event in $\Lc_{1.8d,D}$. Therefore, the FKG inequality (Lemma~\ref{lem:FKG-RWLS}) implies 
	$$\Pb( \Ec_0\cap G ) = \Eb[ \Pb( \Ec_0\cap G \mid \Lc_{1.8d} ) ] \ge \Eb[ \Pb( \Ec_0 \mid \Lc_{1.8d} ) \, \Pb( G \mid \Lc_{1.8d} )] = \Pb(\Ec_0) \, \Pb(G),$$
	where we also used that $\Pb( G \mid \Lc_{1.8d} ) = \Pb(G)$ (since $G$ is independent of $\Lc_{1.8d}$). It follows from Lemma~\ref{lem:cluster-size-2} that $\Pb(G)$ is greater than some universal constant, so
	$$\Pb( \Ec_0\cap G ) \gtrsim \Pb(\Ec_0).$$
	Therefore,
	$$\Pb \big( \overrightarrow\Ac_D(l,d) \big) \ge \Pb( \Ec_0\cap G ) \gtrsim  \Pb(\Ec_0),$$
	which completes the proof.
\end{proof}

\subsubsection*{Observation: exact number of crossing loops}

As for $\Ec_1$ (resp.\ $\Ec_2$), we first show that a positive fraction of the configurations in $\Ec_1$ (resp.\ $\Ec_2$) contain exactly one (resp.\ two) crossing loop (resp.\ loops) which can be used to fulfill the arm event. This is done through Palm's formula (Lemma~\ref{lem:Palm formula}). For any loop configuration $\Lc'$ and each $i=1,2$, let $\bar\Ec_i(\Lc') := \Ec_i(\Lc') \cap \{\Lc'$ has exactly $i$ crossing loops$\}$, and write in particular $\bar\Ec_i := \bar\Ec_i(\Lc_D)$.

\begin{lemma}\label{lem:Palm}
	There exists a universal constant $C>0$ such that for any intensity $\alpha\in(0,1/2]$, and each $i=1,2$,
	\begin{equation} \label{eq:csqPalm}
		\Pb(\bar\Ec_i )\ge e^{-\alpha \nu(\cl)} \,\Pb(\Ec_i) \ge C \,\Pb(\Ec_i).
	\end{equation}
\end{lemma}

\begin{proof}
	Note that $e^{-\alpha \nu(\cl)}$ is the probability that there is no crossing loop in $D$ (across $A_{d,1.8d}$, by definition), which is uniformly bounded away from $0$ by \eqref{eq:crossing loop}. Thus, it suffices to show the first inequality in \eqref{eq:csqPalm}.
	
	For any $k \ge 1$, and any random loop configuration $\Lc'$, let $\Ec_i^{(k)}(\Lc') := \Ec_i(\Lc') \cap \{\Lc'$ has exactly $k$ crossing loops$\}$, and, as usual, write $\Ec_i^{(k)} := \Ec_i^{(k)}(\Lc_D)$. Hence, $\bar\Ec_i = \Ec_i^{(i)}$. Palm's formula (Lemma~\ref{lem:Palm formula}) applied to $F(\gamma)=\ind\{ \gamma\in \cl \}$ and $\Phi(\Lc')=\ind\{ \Ec_i^{(k)}(\Lc') \}$ yields
	\begin{equation} \label{eq:Pk}
		k \, \Pb(\Ec_i^{(k)}(\Lc_D))= \sum_{\gamma\in \cl} \Pb( \Ec_i^{(k)}(\Lc_D\uplus\{\gamma\}) ) \alpha\nu(\gamma).
	\end{equation}
	Given $j\ge 1$, and $\gamma_1,\dots,\gamma_j \in \cl$, if we use Palm's formula again with the same $F$ but $\Phi(\Lc')$ replaced by $\ind\{ \Ec_i^{(k)}(\Lc'\uplus\{\gamma_1,\dots,\gamma_j\}) \}$, this gives 
	\begin{equation} \label{eq:recursive}
		(k-j)\,\Pb( \Ec_i^{(k)}(\Lc_D\uplus\{\gamma_1,\dots,\gamma_j\}) )
		=\sum_{\gamma_{j+1}\in \cl} \Pb( \Ec_i^{(k)}(\Lc_D\uplus\{\gamma_1,\dots,\gamma_{j+1}\}) ) \alpha\nu(\gamma_{j+1}).
	\end{equation}
	The above recursive relation gives that 
	\begin{equation} \label{eq:rec}
		\Pb(\Ec_i^{(k)}(\Lc_D))= \frac{1}{k!} \sum_{\gamma_1\in\cl}\cdots  \sum_{\gamma_k\in\cl}
		\Pb( \Ec_i^{(k)}(\Lc_D\uplus\{\gamma_1,\dots,\gamma_k\}) ) \prod_{j=1}^{k}\alpha\nu(\gamma_{j}).
	\end{equation}
	
	We now deal with the case $i=1$. Note that for any bad cluster containing $k\ge 2$ crossing loops, it is possible to keep only one of the $k$ crossing loops and throw away the other $k-1$ crossing loops, in a way that the new cluster is still a bad cluster across $A_{l, d}$. Therefore
	for all $\gamma_1, \ldots, \gamma_k \in \cl$,
	\begin{equation} \label{eq:rec_union}
		\Ec_1^{(k)}(\Lc_D\uplus\{\gamma_1,\dots,\gamma_k\}) \subseteq \bigcup_{m=1}^k \Ec^{(1)}_1( \Lc_D\uplus\{\gamma_m\} ).
	\end{equation}
	Hence, applying the union bound to \eqref{eq:rec_union}, and plugging it into \eqref{eq:rec}, we have
	$$\Pb(\Ec_1^{(k)}(\Lc_D))\le \frac{(\alpha \nu(\cl))^{k-1}}{(k-1)!} \sum_{\gamma\in\cl} \Pb( \Ec_1^{(1)}(\Lc_D\uplus\{\gamma\}) ) \alpha\nu(\gamma).$$
	Combining the case $k=1$ of \eqref{eq:Pk} with the above inequality, we obtain: for all $k\ge 1$,
	$$\Pb(\Ec_1^{(k)})\le \frac{(\alpha \nu(\cl))^{k-1}}{(k-1)!} \Pb( \bar\Ec_1 )$$
	(we also used $\bar\Ec_1 = \Ec_1^{(1)}$). We deduce that
	$$\Pb( \Ec_1 )=\sum_{k\ge 1} \Pb(\Ec_1^{(k)})\le \sum_{k\ge 1} \frac{(\alpha \nu(\cl))^{k-1}}{(k-1)!} \Pb( \bar\Ec_1 ) = e^{\alpha \nu(\cl)} \, \Pb( \bar\Ec_1 ),$$
	which completes the proof of \eqref{eq:csqPalm} for $i=1$.
	
	The case $i=2$ is similar. Let $k \geq 2$. By definition of $\Ec_2^{(k)}$, we have: for all $\gamma_1,\dots,\gamma_k$ in $\cl$,
	\begin{equation} \label{eq:Ec2k}
		\Ec_2^{(k)}(\Lc_D\uplus\{\gamma_1,\dots,\gamma_k\}) \subseteq \bigcup_{1 \le m_1 < m_2 \le k} \Ec^{(2)}_2( \Lc_D\uplus\{\gamma_{m_1},\gamma_{m_2}\} ).
	\end{equation}
	Moreover, \eqref{eq:rec} with $k = 2$ gives
	\begin{equation}\label{eq:Ec23}
		\Pb( \bar\Ec_2 )=\Pb( \Ec^{(2)}_2) = \frac12 \sum_{\gamma_1\in\cl} \sum_{\gamma_2\in\cl} \Pb( \Ec^{(2)}_2( \Lc_D\uplus\{\gamma_{1},\gamma_{2}\} )) \alpha\nu(\gamma_1) \alpha\nu(\gamma_2).
	\end{equation}
	By the recursive formula \eqref{eq:rec} again, and the union bound applied to \eqref{eq:Ec2k}, we obtain: for all $k \ge 2$,
	\begin{align*}
		\Pb(\Ec_2^{(k)})& \le \frac{(\alpha \nu(\cl))^{k-2}}{k!} \cdot \frac{k(k-1)}{2} 
		\sum_{\gamma_1\in\cl} \sum_{\gamma_2\in\cl} \Pb( \Ec_2^{(2)}(\Lc_D\uplus\{\gamma_1,\gamma_2\}) ) \alpha\nu(\gamma_1) \alpha\nu(\gamma_2)\\
		& = \frac{(\alpha \nu(\cl))^{k-2}}{(k-2)!} \, \Pb( \bar\Ec_2 ),
	\end{align*}
	where the second equality follows from \eqref{eq:Ec23}. This yields
	$$\Pb( \Ec_2 )=\sum_{k\ge 2} \Pb(\Ec_2^{(k)})\le e^{\alpha \nu(\cl)} \, \Pb( \bar\Ec_2 ),$$
	which is the desired result \eqref{eq:csqPalm} for $i=2$, and completes the proof of the lemma.
\end{proof}

The above lemma implies that it is enough to prove \eqref{eq:Eci} for $\bar\Ec_i$ in place of $\Ec_i$. Consequently, in the remainder of this section, we show that $\Pb(\bar\Ec_i)\lesssim \Pb( \overrightarrow\Ac_D(l,d))$ for $i=2$ and then $i=1$ (in Lemmas~\ref{lem:E1} and \ref{lem:E2}, respectively), which then provides Proposition~\ref{prop:locality}. The proof involves a delicate surgery on the crossing loops, to make them stay inside $B_{1.6d}$ without spoiling the arm event. To this end, we use suitable separation results from Section~\ref{sec:sep}, to show that such modifications only cost a constant probability.

\subsubsection*{Case $i=2$: exactly two crossing loops needed}

We first deal with the probability of $\bar\Ec_2$. 

\begin{lemma}\label{lem:E1}
There exists a universal constant $C>0$ such that for all $1 \leq l\le d/2$, $D \supseteq B_{2d}$, and any intensity $\alpha\in (0,\half]$,
$$\Pb( \bar\Ec_2 ) \le C \,\Pb(  \overrightarrow\Ac_D(l,d) ).$$
\end{lemma}

\begin{proof}
\textbf{Step 1: Setup}. 
Let $\Ec_2^*$ be the event that there are exactly two clusters in $\Lc_D$ across $A_{l,d}$ and they are bad. Then, it is not hard to show that
\begin{equation}\label{eq:E2*}
\Pb( \bar\Ec_2 ) \lesssim \Pb( \Ec_2^* ).
\end{equation}
To see it, we condition on $\bar\Ec_2$ and explore the two bad clusters, and let $\Lc_*$ be the collection of unexplored loops. Consider the event  that all clusters in $\Lc_*^b$ that intersect $B_d$ have diameter smaller than $d/10$, which occurs with probability greater than some universal constant $C$ by Lemma~\ref{lem:cluster-size-2}. This implies \eqref{eq:E2*}. Therefore, it suffices to work on the event $\Ec_2^*$ below.

We use $\Pb$ to denote the law of $\Lc_D$ throughout the proof. Recall from \eqref{eq:nu} that $\mu_0$ is the measure on unrooted loops which assigns the weight $4^{-|\gamma|}$ to each loop $\gamma$, while $\nu$ is obtained from $\mu_0$ by dividing by the loop multiplicity. Let $\nu_{\cl}$ and $\mu_{\cl}$ be $\nu$ and $\mu_0$ restricted to crossing loops, respectively.
Since \eqref{eq:Ec23} also holds with $\Ec_2^*$ in place of $\bar\Ec_2$, we have that
\begin{equation}\label{eq:E1A}
	\Pb( \Ec_2^* )= \frac 12 \alpha^2 \, \Pb\times\nu_{\cl}\times\nu_{\cl} (\Ec_2^*)\le \alpha^2 \, \Pb\times\mu_{\cl}\times\mu_{\cl} (\Ec_2^*),
\end{equation}
where the event $\Ec_2^*$ under the product measure should be understood as the collection of (ordered) triples $(L,\gamma_{1},\gamma_{2})$ such that $\Ec_2^*( L\uplus\{\gamma_{1},\gamma_{2}\} )$ holds.
  
\bigskip    

\begin{figure}[t]
	\centering
	\includegraphics[width=.6\textwidth]{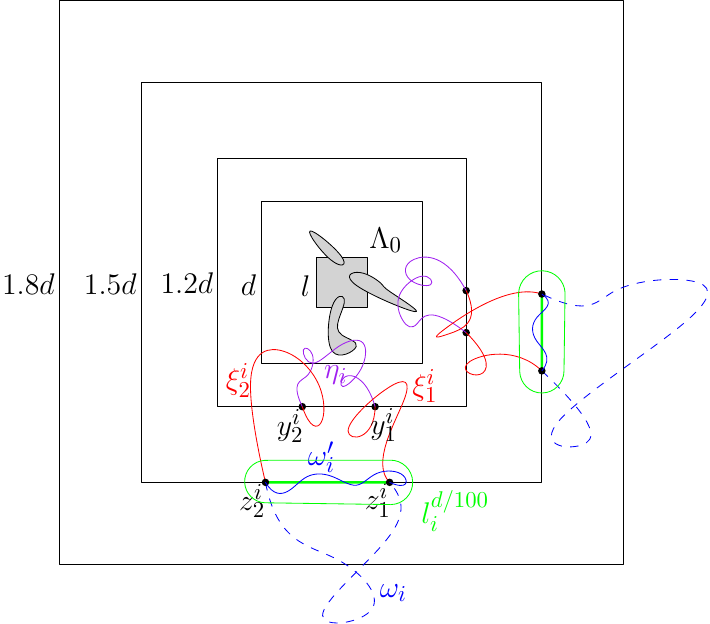}
	\caption{This figure illustrates the decomposition \eqref{eq:dec} for two crossing loops $\gamma_1$ and $\gamma_2$, which plays a central role in the proof of Lemma~\ref{lem:E1}. The filling of $\Lambda_0$ is colored in gray. For each $i =1,2$, the excursion $\eta_i$ in $B_{1.2d}$ is shown in purple, and its starting and ending points along $\partial B_{1.2d}$ are called $y^i_1$ and $y^i_2$, respectively. The red curves indicate the random walks $\xi^i_1$ and $\xi^i_2$ originating from $y^i_1$ and $y^i_2$ (resp.), and stopped upon reaching $\partial B_{1.5d}$. The corresponding hitting points are denoted by $z^i_1$ and $z^i_2$. The remaining part of the decomposition for the loop $\gamma_i$ is denoted by $\omega_i$, and we draw it in dashed blue line. Its refreshed version $\omega'_i$ is in solid blue: it stays in the green region $l_i^{d/100}$, which is the $d/100$-sausage of the arc $l_i$ (also in green, and thicker) of $\partial B_{1.5d}$ between $z^i_1$ and $z^i_2$.}
	\label{fig:E2}
\end{figure}

\textbf{Step 2: Markovian decomposition}. Next, we use the Markov property of $\Lc_D$ and $\mu_0$ to explore the configuration from inside to outside, on the occurrence of $\Ec_2^*$, as illustrated in Figure~\ref{fig:E2}. Roughly speaking, we first freeze the loop soup $\Lc_{1.2d}$, as well as two excursions inside $B_{1.2d}$ of the crossing loops (one for each), which are together used as the initial configuration. Then, we start two packets, of two random walks each, from the endpoints of the previous excursions, stopping them when they hit $\partial B_{1.5d}$. Finally, we require these two packets not to intersect, even with the addition of the loop soup $\Lc_{1.2d, D}$ ($= \Lc_D \setminus \Lc_{1.2d}$ by definition). This culminates with the use of a separation lemma for such a non-intersection event.
	
	More concretely, we first sample the loop soup $\Lc_{1.2d}$, and we then consider
\begin{equation}\label{eq:Lam0}
		\Lambda_0 := \Lambda(B_l,\Lc_{1.2d}).
	\end{equation}
    Let $\gamma_1$ and $\gamma_2$ be the two crossing loops in $\Ec_2^*$ that intersect $\Lambda_0$. For each $i=1,2$, we choose an arbitrary point along $\gamma_i$ lying on $\partial B_{1.5d}$, and from that distinguished point, we consider the first excursion of $\gamma_i$ inside $B_{1.2d}$ which intersects $\Lambda_0$, that we denote by $\eta_i$. Let $y^i_1$ and $y^i_2$ be the starting and ending points (resp.) of $\eta_i$ along $\partial B_{1.2d}$. Let $\xi^i_1$ be the subpath of $\gamma_i^R$ (recall that this denotes the time-reversal of $\gamma_i$) from $y^i_1$ to its first visit to $\partial B_{1.5d}$, and let $\xi^i_2$ be the subpath of $\gamma_i$ from $y^i_2$ to its first visit to $\partial B_{1.5d}$. Observe that one has necessarily $\xi^i_1\cap \Lambda_0 = \emptyset$, from our choice of $\eta_i$ as the \emph{first} excursion intersecting $\Lambda_0$.
    
    Let $\omega_i$ be the remaining part of $\gamma_i$, i.e., so that we have the decomposition
	\begin{equation}\label{eq:dec}
	\gamma_i=\Uc\left(\eta_i\oplus \xi^i_2\oplus \omega_i\oplus [\xi^i_1]^R\right),
	\end{equation}
where $\Uc$ is the unrooting map introduced in Section~\ref{subsec:notation}. Note that this decomposition is not necessarily unique, since there might be multiple choices for $\eta_i$ (indeed, the choice of a particular visit point of $\gamma_i$ along $\partial B_{1.5d}$ was arbitrary). 
	However, this does not matter because we are only after an upper bound at the moment, and this non-uniqueness issue can only increase the probability.
	
	We deduce that $\Pb\times\mu_{\cl}\times\mu_{\cl} (\Ec_2^*)$ can be upper bounded by the product of the following measures on specific events.
	\begin{enumerate}[(i)]
		\item \label{it:1} Sample the loop soup $\Lc_{1.2d}$, and sample the excursions $\eta_1$ and $\eta_2$ according to $\mu^{\exc}_{B_{1.2d}}$ (the excursion measure in $B_{1.2d}$, see Section~\ref{subsec:notation}), independently, such that the following event holds:
		\[
		\Kc:=\{
		\Lambda_0 \subseteq \mathring{B}_d, \, \eta_1\cap \Lambda_0 \neq\emptyset, \, \eta_2\cap \Lambda_0 \neq\emptyset, \text{ and } \Lambda(\eta_1,\Lc_{1.2d}) \cap \eta_2=\emptyset\}.
		\]
		
		\item For each $i=1,2$, let $\xi^i_1$ (resp.\ $\xi^i_2$) be a random walk started from $y^i_1$ (resp.\ $y^i_2$) and stopped upon reaching $\partial B_{1.5d}$. Additionally, we further require $\xi^i_1$ to avoid $\Lambda_0$.
		
		\item \label{it:4} Sample the loop soup $\Lc_{1.2d,D}$, and restrict to the event that
		$$\bar Q=\bar Q^{2,2}(1.5d):= Q^{2,2}(1.5d) \ind_\Dc>0,$$
		where $Q^{2,2}(1.5d)$ is the quality at the scale $1.5d$ (see \eqref{eq:Qjk}) induced by the two packets of two random walks $(\xi^1_1,\xi^1_2)$ and $(\xi^2_1,\xi^2_2)$ inside the frontier loop soup $\Lc_{1.2d,D}^b$, with initial configuration given by $(\Lc_{1.2d},\eta_1,\eta_2,\bar y^1,\bar y^2)$ (where $\bar y^i=(y^i_1,y^i_2)$, for $i=1,2$), for $r=1.2d$ and $R=1.5d$, and where the event $\Dc$ is defined as
		\[
		\Dc:=\{ \text{no loop in $\Lc_{1.2d,D}$ intersects $\Lambda_0$, and no cluster of $\Lc_{1.2d,D}$ disconnects $B_{1.2d}$ from $\infty$} \}
		\]
		(observe that it is clearly decreasing in $\Lc_{1.2d,D}$).
		
		\item \label{it:6} Choose $\omega_i$ according to $\mu^{z_2^i,z_1^i}$ (recall this measure from Section~\ref{subsec:notation}), where $z_1^i$ and $z_2^i$ denote the ending points of $\xi^i_1$ and $\xi^i_2$, respectively, and restrict to the case when $\omega_i$ intersects $\partial B_{1.8d}$, but it does not disconnect $B_{1.5d}$ from $\partial B_{1.8d}$. 
	\end{enumerate}
In fact, the modified quality $\bar Q$ in (\ref{it:4}) is a little bit different from that given in \eqref{eq:Qjk} since the random walks $\xi^1_1$ and $\xi^2_1$ used here should avoid $\Lambda_0$, due to the decomposition (there were no such restrictions in the original setup). However, since $\Lambda_0 \subseteq \mathring{B}_d$, from Remarks~\ref{rmk:general-sep} and~\ref{rmk:conditionl RW}, we can still apply Proposition~\ref{prop:sep-jk} in the case $j=k=2$ for $\bar Q$ to get that 
\begin{equation}\label{eq:E1-2}
	\ind_{\Kc}\,\wt\Pb( \bar Q >0 \mid \Lc_{1.2d}, \eta_1,\eta_2 ) \lesssim \ind_{\Kc}\,\wt\Pb( \bar Q >1/40 \mid \Lc_{1.2d}, \eta_1,\eta_2 ),
\end{equation}
where $\wt\Pb$ is the joint law of $\xi^1$, $\xi^2$ and $\Lc_{1.2d,D}$ (note that they are independent of each other).

Next, we deal with the last condition (\ref{it:6}) on $\omega_i$. In fact, this is just a necessary condition for $\Ec_2^*$ to occur, to prevent $\omega_i$ from spoiling the construction. The total mass of $\omega_i$ in (\ref{it:6}) under $\mu^{z_2^i,z_1^i}$ is given by $\mu^{z_2^i,z_1^i}( M^{z_2^i,z_1^i}_{1.5d,1.8d} )$ (recall this notation from the paragraph before Lemma~\ref{lem:non-disconnecting paths}), which is bounded from above uniformly in the scale $d$ and the ending points $z_2^i,z_1^i$, by Lemma~\ref{lem:non-disconnecting paths-1}. That is, for some universal constant $0<c<\infty$,
 \begin{equation}\label{eq:omega_i}
 	\max_{u_1,u_2\in \partial B_{1.5d}} \mu^{u_1,u_2}( M^{u_1,u_2}_{1.5d,1.8d} ) \le c.
 \end{equation}
Write $\Pb_{1.2d}$ for the law of $\Lc_{1.2d}$.
From the decomposition, we have 
\begin{align}
	\notag
\Pb\times\mu_{\cl} & \times\mu_{\cl} (\Ec_2^*)\\ \notag
&\le \Pb_{1.2d}\times\mu^{\exc}_{B_{1.2d}}\times\mu^{\exc}_{B_{1.2d}} \big[\ind_{\Kc}\,\wt\Pb( \bar Q >0 \mid \Lc_{1.2d}, \eta_1,\eta_2 )\big] \cdot \max_{u_1,u_2\in \partial B_{1.5d}} \mu^{u_1,u_2}( M^{u_1,u_2}_{1.5d,1.8d} )\\ \notag
&\le c\, \Pb_{1.2d}\times\mu^{\exc}_{B_{1.2d}}\times\mu^{\exc}_{B_{1.2d}} \big[\ind_{\Kc}\,\wt\Pb( \bar Q >0 \mid \Lc_{1.2d}, \eta_1,\eta_2 )\big]\\
&\lesssim \Pb_{1.2d}\times\mu^{\exc}_{B_{1.2d}}\times\mu^{\exc}_{B_{1.2d}} \big[\ind_{\Kc}\,\wt\Pb( \bar Q >1/40 \mid \Lc_{1.2d}, \eta_1,\eta_2 )\big], \label{eq:E1-1}
\end{align}
where we used \eqref{eq:omega_i} in the second inequality, and \eqref{eq:E1-2} in the last inequality.

\bigskip

\textbf{Step 3: Surgery}. In the remainder of the proof, on the event $\{\bar Q >1/40\}$, we reverse the above procedure and modify the crossing loops $\gamma_1,\gamma_2$ so that they stay in $B_{1.8d}$, in such a way that $\overrightarrow\Ac_D({l,d})$ occurs after the alteration. We refer again the reader to Figure~\ref{fig:E2} for an illustration.

Define the following event, on the sizes of the clusters in $\Lc_{1.2d,D}$:
\begin{equation}\label{eq:G}
G:=\{ \text{all clusters in $\Lc^b_{1.2d,D}$ that intersect $B_{1.8d}$ have a diameter smaller than $d/1000$} \}.
\end{equation}
By Lemma~\ref{lem:cluster-size-2}, there is a universal constant $C'>0$ such that
\begin{equation}\label{eq:GC'}
	 \Pb(G)\ge C'.
\end{equation}
Note that both $G$ and $\{ \bar Q >1/40 \}$ are decreasing events with respect to $\Lc_{1.2d,D}$. Therefore, it follows from the FKG inequality that 
\begin{align} 
	  \ind_{\Kc}\,\wt\Pb( \bar Q >1/40, G \mid \Lc_{1.2d}, \eta_1,\eta_2 ) & \geq \ind_{\Kc}\,\wt\Pb( \bar Q >1/40 \mid \Lc_{1.2d}, \eta_1,\eta_2 ) \, \wt\Pb( G \mid \Lc_{1.2d}, \eta_1,\eta_2 ) \nonumber\\
	  & \geq C'\,\ind_{\Kc}\,\wt\Pb( \bar Q >1/40 \mid \Lc_{1.2d}, \eta_1,\eta_2 ), \label{eq:E2'}
\end{align}
where we used that $\wt\Pb( G \mid \Lc_{1.2d}, \eta_1,\eta_2 ) = \wt\Pb( G )$ (since $G$ is measurable with respect to $\Lc_{1.2d,D}$).

For $i=1,2$, let $l_i$ be the arc along $\partial B_{1.5d}$ joining $z^i_1$ and $z^i_2$ (the ending points of $\xi^i_1$ and $\xi^i_2$, resp.), chosen so that $l_1\cap l_2=\emptyset$.
We repeat our previous sampling (\ref{it:1})-(\ref{it:4}), and in the last step (\ref{it:6}), we choose $\omega_i$ according to $\mu^{z_2^i,z_1^i}$, but we require that $\omega_i\subseteq l_i^{d/100}$ (the $d/100$-sausage of $l_i$, see \eqref{eq:sausage}) instead. For clarity, we denote the last path obtained now by $\omega_i'$, in order to distinguish it from the path $\omega_i$ constructed before. 
We then concatenate all the pieces, with $\omega_i$ replaced by  $\omega'_i$, to construct the new loop 
\begin{equation}\label{eq:decompose'}
\gamma'_i := \Uc\left(\eta_i\oplus \xi^i_2\oplus \omega'_i\oplus [\xi^i_1]^R\right),
\end{equation}
which is now an unrooted loop remaining inside $B_{1.8d}$. Moreover, the above decomposition becomes \emph{unique}, since we can always identify the first unique excursion $\eta_i$ of $\gamma'_i$, starting from any point of $\gamma'_i$ along $\partial B_{1.5d}$ (this was not the case earlier, since $\omega_i$ could well visit $B_{1.2 d}$). Note also that because of the part $\omega'_i$, the two loops $\gamma'_1$ and $\gamma'_2$ have multiplicity $1$.

We observe that on the event $\Kc\cap\{ \bar Q >1/40 \}\cap G$, we have
\begin{equation}\label{eq:l12}
\Lambda(\eta_1\cup \xi^1_1\cup \xi^1_2\cup l_1^{d/100}, \Lc_{D})\cap (\eta_2\cup \xi^2_1\cup \xi^2_2\cup l_2^{d/100})=\emptyset.
\end{equation}
Indeed, the requirement $\bar Q >1/40$ ensures in particular that $\{z^1_1,z^1_2\}$ and $\{z^2_1,z^2_2\}$ lie at a distance $> 1/20 \cdot 1.5 d$ from each other, using the definition of $\bar Q$ (see \eqref{eq:Qjk}). This implies that $\Lambda(\gamma'_1,\Lc_{D})$ and $\Lambda(\gamma'_2,\Lc_{D})$ are two disjoint outermost clusters in $\Lc_{D}$, containing $\gamma'_1$ and $\gamma'_2$ respectively, and obviously they cross $A_{l,d}$. Moreover, the event $G$ prevents the existence of a crossing loop in $\Lc_D$. Thus, we conclude that
\begin{equation}\label{eq:inclusion}
\Kc\cap\{ \bar Q >1/40 \}\cap G \subseteq \overrightarrow\Ac_{\Lc_{D}\uplus\{ \gamma'_1,\gamma'_2 \}}(l,d).
\end{equation}

\bigskip

\textbf{Step 4: Resampling}. In this last part of the proof, we use the previous procedure to resample the loop soup together with the modified crossing loops $\gamma'_1$ and $\gamma'_2$, which allows us to conclude the proof.
Define the following collection of ordered triples of loop configurations (in a similar way as for $\Ec_2^*$):
    \begin{equation}\label{eq:F1}
  \Upsilon:=\left\{ \begin{array}{c} (L,\gamma_1,\gamma_2): L \text{ is a loop configuration in $D$ without any} \\ 
  	\text{loop in $\cl_{d,1.5d}$, and $(\gamma_1,\gamma_2)\in\cl_{d,1.5d}^2$, both with multiplicity $1$,} \\
 	\text{such that $\Lambda(B_l,L)\subseteq \mathring{B}_d$ and $\overrightarrow\Ac_{L\uplus\{\gamma_1,\gamma_2\}}(l,d)$ occurs} 
 \end{array}	\right\}.
\end{equation}
We can deduce from the unique decomposition of $\gamma'_i$ in \eqref{eq:decompose'} and the inclusion \eqref{eq:inclusion} that 
\begin{equation}\label{eq:reverse}
	\Pb\times\mu_0\times\mu_0 (\Upsilon)
	\gtrsim \Pb_{1.2d}\times\mu^{\exc}_{B_{1.2d}}\times\mu^{\exc}_{B_{1.2d}} \big[\ind_{\Kc}\,\wt\Pb( \bar Q >1/40, G \mid \Lc_{1.2d}, \eta_1,\eta_2 )\big], 
\end{equation}
where we also used the fact that the total mass under $\mu^{z_2^i,z_1^i}$ of all the paths $\omega_i'\subseteq l_i^{d/100}$ is bounded away from $0$ uniformly (in $d$, $z_2^i$, and $z_1^i$), since the size of $l_i^{d/100}$ is of order $d$ (note that the total mass is just given by the Green's function between $z_1^i$ and $z_2^i$ in $l_i^{d/100}\cap\Zb^2$, so we can use Lemma~\ref{lem:green_sausage}).
Plugging \eqref{eq:E2'} and \eqref{eq:E1-1} into \eqref{eq:reverse}, we obtain that 
\begin{equation}\label{eq:F1-A}
\Pb\times\mu_0\times\mu_0 (\Upsilon)
\gtrsim \Pb\times\mu_\cl\times\mu_\cl (\Ec_2^*). 
\end{equation}
Let $E := \{$there are exactly two loops with multiplicity $1$ in $\Lc_D$ that belong to $\cl_{d,1.5d} \}$.
Using Palm's formula again, similarly to \eqref{eq:E1A},
\begin{align}\label{eq:barA-E}
 \Pb( \overrightarrow\Ac_D(l,d)\cap\{ \Lambda(B_l,\Lc_{1.5d})\subseteq \mathring{B}_d \}\cap E )
= \frac 12 \alpha^2 \, \Pb\times\mu_0\times\mu_0 (\Upsilon).
\end{align}
Therefore, using \eqref{eq:barA-E}, \eqref{eq:F1-A} and \eqref{eq:E1A} successively, we get that
\begin{align*}
	\Pb( \overrightarrow\Ac_D(l,d) ) \ge \frac 12 \alpha^2 \, \Pb\times\mu_0\times\mu_0 (\Upsilon) \gtrsim \alpha^2 \, \Pb\times\mu_\cl\times\mu_\cl (\Ec_2^*) \ge \Pb( \Ec_2^* ). 
\end{align*}
This combined with \eqref{eq:E2*} finishes the proof of Lemma~\ref{lem:E1}.
\end{proof}

\subsubsection*{Case $i=1$: exactly one crossing loop needed}

There remains to prove the result for $\bar\Ec_1$, which we do now.

\begin{lemma}\label{lem:E2}
	There exists a universal constant $C>0$ such that for all $1 \leq l\le d/2$, $D \supseteq B_{2d}$, and any intensity $\alpha\in (0,\half]$,
\[
\Pb( \bar\Ec_1 ) \le C \,\Pb(  \overrightarrow\Ac_D(l,d) ).
\]
\end{lemma}

\begin{proof}
	Recall from the definition that on the event $\bar\Ec_1$, there is exactly one crossing loop $\gamma$ in $\Lc_D$, which is used by the only one bad cluster of $\Lc_{D}$ to cross $A_{l,d}$.
	We discuss two cases, depending on the occurrence (or not) of the event 
	\[
	H:=\{\text{there is no cluster in $\Lc_{1.1d}$ across $A_{l,d}$}\}.
	\]
	
	\textbf{Case (a).}
	If $\bar\Ec_1\cap H$ occurs, then we are in the same situation as with $\bar\Ec_2$ in Lemma~\ref{lem:E1} (existence of two crossing loops in $\cl_{d,1.1d}$ to ensure the occurrence of the arm event), if we use $\cl_{d,1.1d}$ instead of $\cl=\cl_{d,1.8d}$. Therefore, by the same proof as for Lemma~\ref{lem:E1}, we have 
	\begin{equation}\label{eq:H}
	\Pb(\bar\Ec_1\cap H) \lesssim \Pb(\overrightarrow\Ac_D(l,d)).
	\end{equation}
     
   \textbf{Case (b).}
If $\bar\Ec_1\cap H^c$ occurs, then we consider the following event $G'$, which plays the same role as $G$ in \eqref{eq:G}:
\[
G':=\{ \text{all clusters in $(\Lc_{1.1d,D}\setminus\cl)^b$ intersecting $B_{1.8d}$ have a diameter smaller than $d/1000$} \},
\]
(recall that $\Lc_{1.1d,D} = \Lc_D \setminus \Lc_{1.1d}$). This event is used to ensure that on $\bar\Ec_1\cap H^c\cap G'$, all clusters in $\Lc_D\setminus\cl$ across $A_{l,d}$ are contained in $B_{1.2d}$.

Before proceeding any further, we first show that conditionally on $\bar\Ec_1\cap H^c$, $G'$ occurs with positive probability, similarly to \eqref{eq:E2'}. 
It is easy to verify that both $G'$ and $\bar\Ec_1\cap H^c$ are decreasing in $\Lc_{1.1d,D}\setminus\cl$. 
Therefore, by first conditioning on $\Lc_{1.1d}$ and the crossing loops in $\Lc_D$, and then applying the FKG inequality for decreasing events in $\Lc_{1.1d,D}\setminus\cl$, we obtain that 
\begin{equation}\label{eq:FKG-G'}
	\Pb(\bar\Ec_1\cap H^c\cap G')\ge \Pb(\bar\Ec_1\cap H^c)\, \Pb( G' ) \gtrsim \Pb(\bar\Ec_1\cap H^c),
\end{equation}
where in the last inequality, we used the same estimate for $G'$ as \eqref{eq:GC'}. 

We also want to show that there is only one cluster in $\Lc_D\setminus\cl$ across $A_{l,d}$. For this, we consider the event $F_1:=\Ac_{\Lc_D\setminus\cl}(l,d)$. Let $F_2$ be the event that there is no crossing loop in $\Lc_{D}$. Then, as explained earlier in the proof of Lemma~\ref{lem:Palm}, we have $\Pb( F_2 )= e^{-\alpha \nu(\cl)}$, which is greater than some universal constant.
	Note that $F_1$ and $F_2$ are independent, and $F_1\cap F_2\subseteq \Ec_0$. Therefore,
	\begin{equation}\label{eq:E0A}
		\Pb( F_1 )\lesssim \Pb( F_1 )\, \Pb( F_2 ) \lesssim \Pb(\Ec_0)\lesssim \Pb \big(  \overrightarrow\Ac_D(l,d) \big),
\end{equation} 
where we used Lemma~\ref{lem:E0} in the last inequality.
This combined with \eqref{eq:FKG-G'} tells us that it is sufficient to deal with the event $\bar\Ec_1\cap H^c\cap G'\cap F_1^c$ below. This combination of events is good enough for us to apply the separation lemma, Proposition~\ref{prop:sep-disc}, and conclude, as we now explain.

\begin{figure}[t]
	\centering
	\includegraphics[width=.5\textwidth]{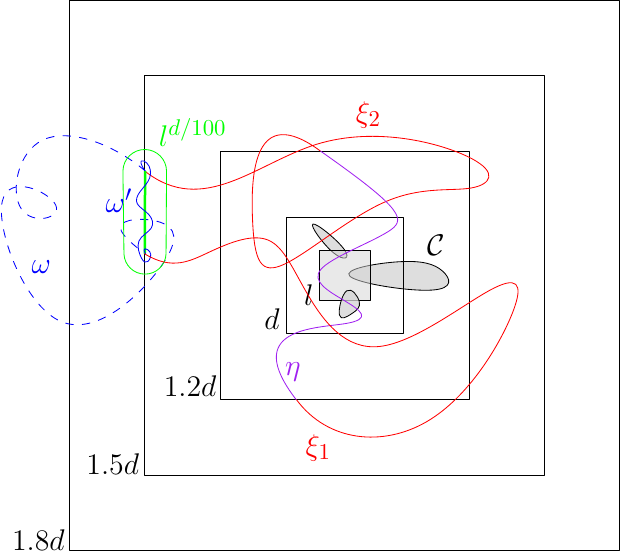}
	\caption{This figure illustrates the construction used for case (b) in the proof of Lemma~\ref{lem:E2}. 
		The gray parts are clusters in $\Lc_{1.2d}$ intersecting $B_l$, and there is only one cluster $\Cc$ that crosses $A_{l,d}$. The excursion $\eta$ in $B_{1.2d}$, depicted in purple, intersects the gray parts, but it avoids $\Cc$. The red curves indicate the random walks $\xi_1$ and $\xi_2$, originating, respectively, from the starting and ending points of $\eta$, and stopped upon reaching $\partial B_{1.5d}$. We draw $\omega$, the remaining part of the decomposition for the loop $\gamma$, in dashed and blue line. Its refreshed version $\omega'$ is in solid and blue: it remains in the green region $l^{d/100}$, which is the $d/100$-sausage of the arc $l$ (shown in green as well, and thicker) of $\partial B_{1.5d}$ between the ending points of $\xi_1$ and $\xi_2$.}
	\label{fig:tE2}
\end{figure}
	
	From Palm's formula (more specifically, \eqref{eq:Pk} with $k=1$), we have 
\begin{equation}\label{eq:P1}
	\Pb(\bar\Ec_1\cap H^c\cap G' \cap F_1^c)=  \alpha\,\Pb \times \nu_{\cl} ( \bar\Ec_1\cap H^c\cap G' \cap F_1^c ),
\end{equation}
where the event on right hand side should be understood as the collection of all pairs $(L,\gamma)$ such that $\bar\Ec_1\cap H^c\cap G'\cap F_1^c$ holds for $L\uplus\{\gamma\}$.
On the event $\bar\Ec_1\cap H^c\cap G'\cap F_1^c$, there is only one cluster in $\Lc_D\setminus\{\gamma\}$ crossing $A_{l,d}$, and it remains in $B_{1.2d}$, as explained earlier.
We are in a position to apply a decomposition for $\gamma$ similar to \eqref{eq:dec}, as follows (see Figure~\ref{fig:tE2} for an illustration).
	\begin{enumerate}[(i)]
		\item Sample the loop soup $\Lc_{1.2d}$, and sample the excursion $\eta$ in $B_{1.2d}$ according to $\mu^{\exc}_{B_{1.2d}}$, independently, and restrict to the case when the following event holds:
		\[
		\Kc:=\{ \text{only one cluster $\Cc$ of $\Lc_{1.2d}$ across $A_{l,d}$, $\eta\cap\Cc=\emptyset$, and $\eta\cap\Lambda(B_l,\Lc_{1.2d})\neq\emptyset$} \}.
		\]
		We denote the starting and ending points of $\eta$ by $x_1$ and $x_2$, respectively.
		
		\item Let $\xi_1$ and $\xi_2$ be two random walks, starting from $x_1$ and $x_2$ (resp.), and stopped upon reaching $\partial B_{1.5d}$. Moreover, we further require $\xi_1$ to avoid $\Lambda(B_l,\Lc_{1.2d})$.
		
		\item Sample the loop soup $\Lc_{1.2d,D}$, and restrict to the event that
		\begin{equation}\label{eq:bQ}
			\bar Q=\bar Q^2(1.5d):= Q^2(1.5d)\ind_\Dc>0,
		\end{equation}
		where $Q^2(1.5d)$ is defined as  in \eqref{eq:quality-disc}, with $r=1.2d$, $R=1.5d$, $L_r=\Lc_{1.2d}$, $U=\Cc$, $V=\eta$, and $\bar x=(x_1,x_2)$, except that the simple random walk $\xi_1$ (starting from $x_1$) should be replaced by a conditional one (avoiding $\Lambda(B_l,\Lc_{1.2d})$), and the event $\Dc$ is now given by
		\[
		\Dc:=\left\{ \begin{array}{c}
			\text{no loop in $\Lc_{1.2d,D}$ intersects $\Cc$, no cluster of $\Lc_{1.2d,D}$ disconnects $B_{1.2d}$ from $\infty$,}\\ 
			\text{and all clusters in $\Lc^b_{1.2d,D}$ intersecting $B_{1.8d}$ have a diameter smaller than $d/1000$}\end{array}\right\},
		\]
		which is clearly decreasing in $\Lc_{1.2d,D}$.
		
		\item Finally, sample $\omega$ according to $\mu^{z_2, z_1}$, where $z_i = \xi_i \cap \partial B_{1.5d}$ for $i=1,2$. Furthermore, we require that $\omega$ intersects $\partial B_{1.8d}$, but does not disconnect $B_{1.5d}$ from $\partial B_{1.8d}$.
	\end{enumerate}

    From the above decomposition, we have 
\begin{align}\label{eq:construction}
	\Pb \times \nu_{\cl} ( \bar\Ec_1\cap H^c\cap G'\cap F_1^c )  
	\lesssim \Pb_{1.2d}\times\mu^{\exc}_{B_{1.2d}} \big[\ind_{\Kc}\,\wt\Pb( \bar Q >0 \mid \Lc_{1.2d}, \eta )\big],
\end{align}
where $\wt\Pb$ is now the joint law of $\xi_1$, $\xi_2$ and $\Lc_{1.2d,D}$. We can then use Proposition~\ref{prop:sep-disc}, which implies
\begin{equation}\label{eq:sepa}
	\ind_{\Kc}\,\wt\Pb( \bar Q >0 \mid \Lc_{1.2d}, \eta )\lesssim \ind_{\Kc}\,\wt\Pb( \bar Q >1/20 \mid \Lc_{1.2d}, \eta ).
\end{equation}
On the event $\Kc\cap \{\bar Q>1/20\}$, one can find at least one arc $l$ of $\partial B_{1.5d}$ such that 
\begin{equation}\label{eq:CFill}
	\Cc\cap\Fill\big(\Lambda(\eta\cup\xi_1\cup\xi_2\cup l^{d/100},\Lc_{D})\big)=\emptyset.
\end{equation}
Hence, we can hook up the ending points of $\xi_1$ and $\xi_2$ (along $\partial B_{1.5d}$) with a refreshed path $\omega'$ (playing the role of $\omega$), remaining in $l^{d/100}$ and such that $\Ac_{\Lc_{D}\uplus\{\gamma'\}}(l,d)$ occurs with the new loop
\[
\gamma' = \Uc \left( \eta\oplus \xi_2\oplus \omega'\oplus [\xi_1]^R \right)
\]
(this recombination is shown in Figure~\ref{fig:tE2}). Note that the multiplicity of $\gamma'$ is $1$, so it has the same weight under $\nu$ and $\mu_0$.
Moreover, the total weight of all compatible paths $\omega'$ is bounded away from $0$, as in the proof of Lemma~\ref{lem:E1}.

We then consider the following collection of ordered pairs (playing the same role as $\Upsilon$ in \eqref{eq:F1}):
\begin{equation}\label{eq:F'}
	\Upsilon':=\left\{ \begin{array}{c} (L,\gamma): 
		L \text{ is a loop configuration in $D$ without any loop in $\cl_{d,1.5d}$,}\\ 
		\text{which contains only one cluster crossing $A_{l,d}$, this cluster stays in $B_{1.2d}$,}\\ 
		\text{and $\gamma\in\cl_{d,1.5d}$, with multiplicity $1$, is such that $\overrightarrow\Ac_{L\uplus\{\gamma\}}(l,d)$ occurs} 
	\end{array}	\right\}.
\end{equation}
Using the same arguments as in the proof of Lemma~\ref{lem:E1}, we have 
\begin{align*}
	\Pb(\overrightarrow\Ac_{D}(l,d)) 
	&\gtrsim\   \alpha\,\Pb\times \mu_0 ( \Upsilon' )   &\text{similarly to \eqref{eq:barA-E}}\\
	&\gtrsim\   \alpha\,\Pb_{1.2d}\times\mu^{\exc}_{B_{1.2d}} \big[\ind_{\Kc}\,\wt\Pb( \bar Q >1/20 \mid \Lc_{1.2d}, \eta )\big] &\text{using \eqref{eq:CFill}}\\
	&\gtrsim\  \alpha\,\Pb\times\nu_{\cl}(\bar\Ec_1\cap H^c\cap G'\cap F_1^c) &\text{using \eqref{eq:construction} and \eqref{eq:sepa}}\\
	&\gtrsim\  \Pb(\bar\Ec_1\cap H^c \cap G'\cap F_1^c) &\text{from \eqref{eq:P1}}.
\end{align*}
This last result, combined with \eqref{eq:FKG-G'} and \eqref{eq:E0A}, shows that 
\[
 \Pb(\bar\Ec_1\cap H^c) \lesssim \Pb(\bar\Ec_1\cap H^c\cap G'\cap F_1^c) + \Pb( F_1 ) \lesssim \Pb(\overrightarrow\Ac_{D}(l,d)). 
\]
Combining the above with \eqref{eq:H}, we obtain that 
\[
\Pb(\bar\Ec_1)= \Pb(\bar\Ec_1\cap H)+ \Pb(\bar\Ec_1\cap H^c)\lesssim \Pb(\overrightarrow\Ac_{D}(l,d)),
\]
which finishes the proof.
\end{proof}

\subsubsection*{Locality: wrap-up of the proof}

We are finally in a position to prove Proposition~\ref{prop:locality}.

\begin{proof}[Proof of Proposition~\ref{prop:locality}]
	By using \eqref{eq:E123}, and applying Lemma~\ref{lem:Palm} twice, to $\Ec_1$ and $\Ec_2$, we obtain
	\[
	\Pb(\Ac_D(l,d))\lesssim \Pb(  \Ec_0 )+ \Pb( \bar\Ec_1 )+\Pb( \bar\Ec_2 ).
	\]
	Then, combining Lemmas~\ref{lem:E0}, \ref{lem:E1} and~\ref{lem:E2} yields
	\[
	\Pb(\Ac_D(l,d))\lesssim \, \Pb(\overrightarrow\Ac_{D}(l,d)),
	\]
	completing the proof of Proposition~\ref{prop:locality}.
\end{proof}

\subsection{Reversed locality property} \label{sec:locality2}

In this section, we derive a locality result in the reversed, ``inward'' direction. For this purpose, we introduce the \emph{truncated inward arm event} as follows: for all $z\in \Zb^2$, $2 \leq l<d$, and any random loop configuration $\Lc'$,
\begin{equation}\label{eq:inward-ae}
	\overleftarrow\Ac_{\Lc'}(z;l,d) := \Ac_{\Lc'}(z;l,d) \cap\{ \Lambda(\partial B_d(z), \Lc') \subseteq B_{l/2}(z)^c \}.
\end{equation}
As before, we abbreviate $\overleftarrow\Ac_D(z;l,d):=\overleftarrow\Ac_{\Lc_D}(z;l,d)$ and $\overleftarrow\Ac_D(l,d):=\overleftarrow\Ac_D(0;l,d)$.
Note that by definition, $\overleftarrow\Ac_D(z;l,d)$ is independent of the loop soup $\Lc_{l/2}(z)$.
In a similar fashion as in the outward direction (Proposition~\ref{prop:locality}), the following inward locality result holds true.

\begin{proposition}\label{prop:in-locality}
	There is a universal constant $C>0$ such that for all $z\in\Zb^2$, $2 \leq l\le d/2$, $D \supseteq B_{d}(z)$, and any intensity $\alpha\in (0,\half]$,
	\[
	\Pb(\Ac_D(z;l,d))\le C\, \Pb(\overleftarrow\Ac_D(z;l,d)).
	\]
\end{proposition}

\begin{proof}[Proof of Proposition~\ref{prop:in-locality}]
	We give a sketch of the proof, since it is similar to that of Proposition~\ref{prop:locality}, except that we have to follow an inward exploration procedure instead. Hence, we only emphasize the main differences. Throughout the proof, we assume $z=0$, and we consider as crossing loops the elements of $\cl := \cl_{0.6l,l}$.
	By examining the number of crossing loops used to fulfill the arm event, we get three cases as before (see the paragraph preceding \eqref{eq:E123}), which are, respectively, $0$, $1$, or $2$ crossing loop(s), and we use the corresponding notation $\Ec_i$, for $i=0,1,2$. Then, it is easy to verify that Lemma~\ref{lem:Palm} also holds in this setting (with an obvious change of definition for $\bar\Ec_i$ as well).
	
	First, $\Pb(\Ec_0)$ can be upper bounded by  $C\,\Pb(\overleftarrow\Ac_D(l,d))$ as in Lemma~\ref{lem:E0}, through an application of the FKG inequality.
	
	As for $\bar\Ec_2$, we can use an inward Markovian decomposition, similar to Step 2 in the proof of Lemma~\ref{lem:E1}. More precisely, we first sample $\Lc_{D\setminus B_{0.9l}}$, to get
	$$\Lambda_0 := \Lambda(\partial B_d, \Lc_{D\setminus B_{0.9l}}),$$
	with the requirement that $\Lambda_0 \subseteq B_{l}^c$. Then, we decompose each crossing loop $\gamma_i\in\cl$ as
	$$\gamma_i = \Uc \left( \eta_i\oplus \xi^i_2\oplus \omega_i\oplus [\xi^i_1]^R \right),$$
	where $\eta_i$ is an excursion in $D\setminus \mathring{B}_{0.9l}$ with starting and ending points along $\partial B_{0.9l}$, $\xi^i_1$ (resp.\ $\xi^i_2$) is a random walk from the starting (resp.\ ending) point of $\eta_i$ to its first visit of $\partial B_{0.8l}$, which avoids $\Lambda_0$ (resp.\ $D^c$), and $\omega_i$ is the remaining part.
	
	Next, we use the reversed separation lemma, Proposition~\ref{prop:inv-sep}, with $j=k=2$ (i.e. two packets, of two random walks each), to conclude that $\xi^1_1\cup\xi^1_2$ and $\xi^2_1\cup\xi^2_2$ are well-separated at the scale $0.8l$ within the RWLS, with the initial configuration inside $D\setminus \mathring{B}_{0.9l}$. On such a well-separation event, with positive probability, we can replace $\omega_i$ by a path $\omega'_i$ that stays in a local region around the ending points of $\xi^i_1$ and $\xi^i_2$, which produces $\gamma'_i$ as before, such that the desired arm event $\overleftarrow\Ac_{\Lc_D\uplus\{\gamma'_1,\gamma'_2\}}(l,d)$ occurs. Hence, we can get the inward version of Lemma~\ref{lem:E1} in this way, from the same arguments as in the proof of Lemma~\ref{lem:E1}.
	
	Finally, the same bound for $\bar\Ec_1$, i.e. an inward version of Lemma~\ref{lem:E2}, can be obtained in a similar way as that result. More specifically, we use an inward exploration, and we then apply the reversed separation lemma, now for one packet of two random walks (see Proposition~\ref{prop:sep-disc}, and the paragraph below it). Since only simple adaptations are required, we omit the details. This concludes the proof of Proposition~\ref{prop:in-locality}.
\end{proof}

\subsection{Other two-arm and four-arm events}\label{subsec:other_arm}

In the companion paper \cite{GNQ2024b}, our proofs also involve two-arm and four-arm events in the RWLS near the boundary of a domain (corresponding to the situation when one or two big clusters of loops approach each other at a vertex close to that boundary). We thus have to define the boundary two-arm and four-arm events. For completeness, we also define the interior two-arm event, even though it is not used in \cite{GNQ2024b}.

\begin{definition} \label{def:2n arm event for rwls}
	Let $1 \leq l<d$ and $D\subseteq \Zb^2$. For a RWLS configuration $\Lc_D$, the \emph{interior two-arm event} in the annulus $A_{l,d}$, \emph{boundary two-arm event} in the annulus $A^+_{l,d}$ and the \emph{boundary four-arm event} in the annulus $A^+_{l,d}$, are respectively defined as
\begin{itemize}
\item $\Ac^{2}_{D}(l,d) := \{$there is at least one outermost cluster in $\Lc_D$ whose frontier crosses $A_{l,d}\}$
\item $\Ac^{2,+}_{D}(l,d) := \{$there is at least one outermost cluster in $\Lc_{D\cap \Hb}$ whose frontier crosses $A^+_{l,d}\}$
\item $\Ac^+_{D}(l,d) := \{$there are at least two outermost clusters in $\Lc_{D\cap \Hb}$ crossing $A_{l,d}^+\}$.
\end{itemize}
We also define the \emph{truncated arm event} by $\overrightarrow\Ac^{\boldsymbol{\cdot}}_{D}(l,d) := \Ac^{\boldsymbol{\cdot}}_{D}(l,d) \cap \{ \Lambda(B_l, \Lc_D) \subseteq B_{2d} \}$, 
 the \emph{truncated inward arm event} by
	$\overleftarrow\Ac^{\boldsymbol{\cdot}}_{D}(l,d) := \Ac^{\boldsymbol{\cdot}}_{D}(l,d) \cap\{ \Lambda(\partial B_d, \Lc_D) \subseteq B_{l/2}^c \}$, and the \emph{local arm event} by  $\Ac^{\boldsymbol{\cdot}}_{\mathrm{loc}}(l,d):=\Ac^{\boldsymbol{\cdot}}_{B_{2d}}(l,d)$, where the superscript $\boldsymbol{\cdot}$ can be either ``$2$'' (for interior two-arm),  ``$2,+$'' (for boundary two-arm) or ``$+$'' (for boundary four-arm) on both sides of each equality.
\end{definition}

We summarize the corresponding results for the (interior and boundary) two-arm events and boundary four-arm events.
The idea of the proof is similar to that employed in the interior four-arm event. In fact, for the two-arm events, the proof is much simpler because of the lack of interplay between the two crossing clusters.

\begin{proposition}[Locality]\label{prop:2n-locality}
	There exists a universal constant $C>0$ such that for all $1 \leq l\le d/2$, $D \supseteq B_{2d}$, and any intensity $\alpha\in (0,\half]$,
	\begin{align}\label{eq:2n-local-1}
		&\Pb(\Ac^{\boldsymbol{\cdot}}_D(l,d))\le C\, \Pb( \overrightarrow\Ac^{\boldsymbol{\cdot}}_D(l,d) )\le C\, \Pb(\Ac^{\boldsymbol{\cdot}}_{\loc}(l,d)),\\[2mm]
		\label{eq:2n-local-2}
    	&\Pb(\Ac^{\boldsymbol{\cdot}}_D(l,d))\le C\, \Pb( \overleftarrow\Ac^{\boldsymbol{\cdot}}_D(l,d) ),
    \end{align}
where the superscript $\boldsymbol{\cdot}$ can be either ``$\,2$'', ``$\,2,+$'' or ``$\,+$'', simultaneously in \eqref{eq:2n-local-1} and \eqref{eq:2n-local-2}.
\end{proposition}
In the following, we provide a sketch of proof for the above proposition, which only places emphasis on the potential modifications that need to be made.
\begin{proof}[Sketch of proof]
We start with the interior two-arm events, namely we let the  superscript $\boldsymbol{\cdot}$ be ``$2$''.
For \eqref{eq:2n-local-1}, we can use a similar exploration procedure as before and do surgery for the crossing loop. In this case, we only need to deal with the non-disconnection event (note that we require the frontier of the cluster to cross $A_{l,d}$, hence this cluster does not disconnect $B_l$), instead of the non-intersection event between two clusters. 
Suppose that we need a crossing loop $\gamma$ to fulfill the event $\Ac^{2}_D(l,d)$, otherwise the application of FKG-inequality (see Lemma~\ref{lem:E0}) will give the result. 
It is also safe to assume that there is only one crossing loop by using Palm's formula (see Lemma~\ref{lem:Palm}).
The crossing loop $\gamma$ must intersect the random set $\Lambda(B_l,\Lc_{1.2d})$. One can do a similar excursion-decomposition for $\gamma$ as in \eqref{eq:dec}.
Suppose we have explored the loop soup $\Lc_{1.2d}$ and the excursion part of $\gamma$. The restriction on what remains to be explored is given by the (non-disconnection) event
$\{ \bar Q:=Q^2(1.5d)\ind_{\Dc}>0 \}$, where $\bar Q$ is the same as \eqref{eq:bQ} with only $U=\Cc$ there replaced by $U=B_l$. 
Then, we can use the same surgery to deal with the crossing loop as we did for $\gamma$ in case (b) in the proof of Lemma~\ref{lem:E2}. It follows immediately that 
\[
\Pb( \Ac^{2}_D(l,d) )\lesssim \Pb( \overrightarrow\Ac^{2}_D(l,d) ),
\]
concluding the proof of \eqref{eq:2n-local-1}. 
It is not hard to show \eqref{eq:2n-local-2} by using an inward exploration (see the proof of Proposition~\ref{prop:in-locality} for similar idea).

As for the boundary two-arm or four-arm events, the exploration procedure is the same as in the interior case, while the surgery part needs some additional care. That is, we need to avoid the case when the crossing loop is close to $\partial\Hb$. To fix idea, we take the boundary four-arm event $\Ac^{+}_D(l,d)$ for illustration. Suppose further that we are in the case where exactly two crossing loops are needed, corresponding to $\bar\Ec_2$ in Section~\ref{sec:locality}. Then, we do the same decomposition for the two crossing loops $\gamma_1$ and $\gamma_2$ as in \eqref{eq:dec}, but in $\Hb$. Now, the two sausages $l_1^{d/100}$ and $l_2^{d/100}$ (see Figure~\ref{fig:E2}) need to be in $\Hb$ (or the endpoints $z_1^i$ and $z_2^i$ need to be macroscopically away from $\partial\Hb$ for $i=1,2$), and also need to satisfy the condition in \eqref{eq:l12}. This can be done by using a boundary version of the separation lemma, incorporating the distance between the endpoints of the simple random walks and $\partial\Hb$. More precisely, we use the setup in the paragraph containing Proposition~\ref{prop:sep-jk}, and we assume that all the quantities there are defined in $\Hb$ instead. We define the boundary version of \eqref{eq:Qjk} by 
\[
Q_+^{j,k}(s):=Q^{j,k}(s)\wedge \sup\{ \delta\ge0: \dist(z,\partial\Hb)\ge \delta s \text{ for all } z \in (\Pi^1_s\cup\Pi^2_s) \cap \partial B_s \}.
\]
It is not hard to check that \eqref{eq:sep_twopackets} also holds with $Q_+^{j,k}(R)$ in place of $Q^{j,k}(R)$. Using the boundary quality $Q_+^{j,k}(s)$, and the separation lemma associated with it, in the corresponding parts of the proof of Lemma~\ref{lem:E1}, we are able to conclude the proof of locality for the boundary four-arm events in the case of exactly two crossing loops. The other cases, and the boundary two-arm events, can be analyzed in a similar way, by using the boundary separation lemmas. 
\end{proof}

\section{Quasi-multiplicativity} \label{sec:quasi-mult}

In this final part, we show that the four-arm probabilities are quasi-multiplicative, in Section~\ref{sec:quasi-mult_proof} (Proposition~\ref{lem:quasi}). The proof is based on the locality and reversed locality properties established in Section~\ref{sec:arm-lwrs}, and again, it also relies crucially on the separation lemmas derived in Section~\ref{sec:sep}. More specifically, we obtain an upper bound, which is what is needed for future proofs. As a conclusion to this paper, we then use this result, in Section~\ref{sec:qm_four_arm}, to give an upper bound on the four-arm probability in the RWLS, involving the same exponent as for the BLS (Theorem~\ref{prop:armp}). As we mentioned heuristically in the introduction, this estimate plays a key role in the companion paper \cite{GNQ2024b}. Finally, in Section~\ref{subsec:other_arm2}, we state the analogous upper bounds on (interior and boundary) two-arm events, as well as boundary four-arm events (Theorem~\ref{prop:2n-armp}), which can be proved by following the same roadmap.

\subsection{Statement and proof} \label{sec:quasi-mult_proof}

We now prove that the four-arm event is quasi-multiplicative, which is a classical property shared by many statistical physics models, such as Bernoulli percolation at and near criticality (in dimension two).

\begin{proposition}[Quasi-multiplicativity]\label{lem:quasi}
	For any intensity $\alpha\in (0,\half]$, there exists a constant $c_1(\alpha)>0$ such that for all $z\in \Zb^2$, $1 \leq d_1 \le d_2/2\le d_3/16$, and $D \supseteq B_{2d_3}(z)$,
	\begin{equation}\label{eq:quasi-1}
	\Pb( \Ac_{D}(z;d_1,d_3) ) \le c_1\, \Pb( \Ac_{\loc}(z;d_1,d_2) )\, \Pb( \Ac_{D}(z;4d_2,d_3) ).
	\end{equation}
\end{proposition}

Before turning to the proof of Proposition~\ref{lem:quasi}, we want to mention that we expect the corresponding quasi-multiplicativity lower bound in \eqref{eq:quasi-1} to hold true, as is the case for e.g.\ Bernoulli percolation. Establishing it would be important to get the exact exponent for the discrete arm event (not just a lower bound in terms of the exponent for the BLS). However, we are not tackling this question in the present paper, since it seems to require a more detailed analysis for clusters of loops. In any case, it is arguably the upper bound on probabilities which is most relevant for applications, especially for our specific purpose in \cite{GNQ2024b} (connectivity properties of level sets in the RWLS and the GFF).

\begin{proof}
We establish this result through a combination of Propositions~\ref{prop:locality} and~\ref{prop:in-locality}, that is, a surgery on crossing loops in both directions. The constructions are quite similar to those in Section~\ref{sec:arm-lwrs}, although the precise ``combinatorics'' turns out to be somewhat more complicated, and we only sketch the key points. As usual, we assume that $z=0$.

Similarly to the paragraph above \eqref{eq:E123}, we can divide the arm events $\Ac_{D}(d_1,d_2)$ and $\Ac_{D}(4d_2,d_3)$ into disjoint unions of events according to the number of associated ``bad clusters''. More precisely, for $\Ac_{D}(d_1,d_2)$ (resp.\ $\Ac_{D}(4d_2,d_3)$), an associated bad cluster is defined as an outermost cluster across $A_{d_1,d_2}$ (resp.\ $A_{4d_2,d_3}$) which still crosses $A_{d_1,d_2}$ (resp.\ $A_{4d_2,d_3}$) if we remove all the loops in $\cl_{d_2,1.8d_2}$ (resp.\ $\cl_{2.2d_2,4d_2}$) contained in it. Depending on the number of associated bad clusters, we can define three disjoint subevents $\Ec_i(d_1,d_2)$ (resp.\ $\Ec_i(4d_2,d_3)$) of $\Ac_{D}(d_1,d_2)$ (resp.\ $\Ac_{D}(4d_2,d_3)$), for $i=0$ (no bad cluster), $1$ (exactly one bad cluster), $2$ (at least two bad clusters), so that
\begin{align*}
	\Ac_{D}(d_1,d_2) &= \Ec_0(d_1,d_2) \cup \Ec_1(d_1,d_2) \cup \Ec_2(d_1,d_2),\\
	\Ac_{D}(4d_2,d_3) &= \Ec_0(4d_2,d_3) \cup \Ec_1(4d_2,d_3) \cup \Ec_2(4d_2,d_3).
\end{align*}
This then leads us to consider the nine events
\begin{equation}
	\Ec_{i,j}:=\Ec_i(d_1,d_2)\cap \Ec_j(4d_2,d_3), \quad \text{ for all } 0\le i,j\le 2.
\end{equation}  
Since $\Ac_{D}(d_1,d_3)\subseteq \Ac_{D}(d_1,d_2)\cap \Ac_{D}(4d_2,d_3)$ by monotonicity, it follows from the union bound that 
\[
\Pb(\Ac_{D}(d_1,d_3))\le \sum_{0\le i,j\le 2} \Pb( \Ec_{i,j} ).
\]
Hence, we only need to show that for each $(i,j) \in \{0,1,2\}^2$,
\begin{equation}\label{eq:Ecij}
	\Pb(\Ec_{i,j})\lesssim_{\alpha} \Pb( \Ac_{\loc}(d_1,d_2) )\, \Pb( \Ac_{D}(4d_2,d_3) ).
\end{equation}

Each of these nine cases can be addressed by using the same strategy as we applied in the previous section. However, listing and analyzing all the subcases one by one would be quite tedious, so we focus on $\Ec_{2,2}$, which already contains all the potential difficulties.

We first observe that if $\Ec_{2,2}$ occurs for $\Lc_D$, then one can always find at most $4$ loops $\gamma^1_1, \gamma^2_1, \gamma^1_2, \gamma^2_2$ in $\cl_{d_2, 1.8 d_2} \cup \cl_{2.2d_2, 4d_2}$ so that  $\Ec_{2,2}$ also occurs for $\big( \Lc_D \setminus (\cl_{d_2, 1.8 d_2} \cup \cl_{2.2d_2, 4d_2}) \big) \uplus \{\gamma^1_1, \gamma^2_1, \gamma^1_2, \gamma^2_2\}$.
At this point, note that an extra difficulty arises (compared to the proofs of the locality results), because the sets $\{\gamma^1_1, \gamma^2_1\}$ and $\{\gamma^1_2, \gamma^2_2 \}$ can involve the same loop(s) in $\cl_{d_2,1.8d_2} \cap \cl_{2.2d_2,4d_2} = \cl_{d_2,4d_2}$. Note that $\Ec_{2,2}$ is included in the union of the three subevents below.

\begin{enumerate}[(a)]
	\item There exist $4$ loops $\gamma^1_1, \gamma^2_1, \gamma^1_2, \gamma^2_2 \in \Lc_D$ so that $(\gamma^1_1, \gamma^2_1, \gamma^1_2, \gamma^2_2) \in (\cl_{d_2,1.8d_2}) ^2 \times (\cl_{2.2d_2,4d_2}) ^2$,
	and that $\Ec_{2,2}$ also occurs for $\big( \Lc_D \setminus (\cl_{d_2, 1.8 d_2} \cup \cl_{2.2d_2, 4d_2}) \big) \uplus \{\gamma^1_1, \gamma^2_1, \gamma^1_2, \gamma^2_2\}$.
	
	\item There exist $3$ loops $\gamma^1_1, \gamma^1_2, \gamma^2_1 = \gamma^2_2 \in \Lc_D$ so that $(\gamma^1_1, \gamma^1_2, \gamma^2_1 = \gamma^2_2) \in (\cl_{d_2,1.8d_2})  \times (\cl_{2.2d_2,4d_2})  \times \cl_{d_2,4d_2}$,
	and that $\Ec_{2,2}$ also occurs for $\big( \Lc_D \setminus (\cl_{d_2, 1.8 d_2} \cup \cl_{2.2d_2, 4d_2}) \big) \uplus \{\gamma^1_1, \gamma^2_1, \gamma^1_2\}$.
	
	\item  There exist $2$ loops $\gamma^1_1 = \gamma^1_2, \gamma^2_1 = \gamma^2_2\in \Lc_D$ so that  $(\gamma^1_1 = \gamma^1_2, \gamma^2_1 = \gamma^2_2) \in (\cl_{d_2,4d_2})^2$, and that $\Ec_{2,2}$ also occurs for $\big( \Lc_D \setminus (\cl_{d_2, 1.8 d_2} \cup \cl_{2.2d_2, 4d_2}) \big) \uplus \{\gamma^1_1, \gamma^2_1\}$.
\end{enumerate}

Suppose for example that we are in the case (c), to fix ideas, and denote the corresponding subevent by $\Ec^{(c)}_{2,2}$. Let $\bar \Ec^{(c)}_{2,2}$ be the event that $\Ec_{2,2}^{(c)}$ occurs with exactly two loops in $\cl_{d_2,4d_2}$, which are the only two loops in $\Lc_D$ belonging to $\cl_{d_2,1.8d_2}\cup \cl_{2.2d_2,4d_2}$. Using Palm's formula, analogously to the proof of Lemma~\ref{lem:Palm}, we can obtain 
\begin{equation}\label{eq:E22c}
	\Pb( \Ec^{(c)}_{2,2} ) \lesssim \Pb( \bar\Ec^{(c)}_{2,2} ).
\end{equation}

Let $\Ec_{2,2}^*\subseteq \bar\Ec^{(c)}_{2,2}$ be the event that $\bar\Ec^{(c)}_{2,2}$ holds and except for the bad clusters, all the other clusters intersecting $B_{4d_2}$ have diameter smaller than $d_2/10$.
Then, similarly to \eqref{eq:E2*}, we have $\Pb( \bar\Ec^{(c)}_{2,2} ) \lesssim \Pb( \Ec_{2,2}^* )$. Hence, it suffices to work with $\Ec_{2,2}^*$ below.
Note that on $\Ec_{2,2}^*$, the following events both hold:
$$ \Lambda(B_{d_1}, \Lc_D \setminus \cl_{d_2, 1.8 d_2}) \subseteq \mathring{B}_{d_2} \quad \text{and} \quad \Lambda(D\setminus \mathring{B}_{d_3}, \Lc_D \setminus \cl_{2.2 d_2, 4d_2}) \subseteq B_{4d_2}^c.$$

From now on, we consider the set of (very big) crossing loops $\cl:=\cl_{d_2,4d_2}$, and we denote by $\nu_{\cl}$ and $\mu_{\cl}$ the associated measures. Similarly to \eqref{eq:E1A}, we have 
\begin{equation}\label{eq:E22}
	\Pb( \Ec_{2,2}^* )= \frac 12 \alpha^2 \, \Pb\times\nu_{\cl}\times\nu_{\cl} (\Ec_{2,2}^*)\le \alpha^2 \, \Pb\times\mu_{\cl}\times\mu_{\cl} (\Ec_{2,2}^*).
\end{equation}

\begin{figure}[t]
	\centering
	\subfigure{\includegraphics[width=.47\textwidth]{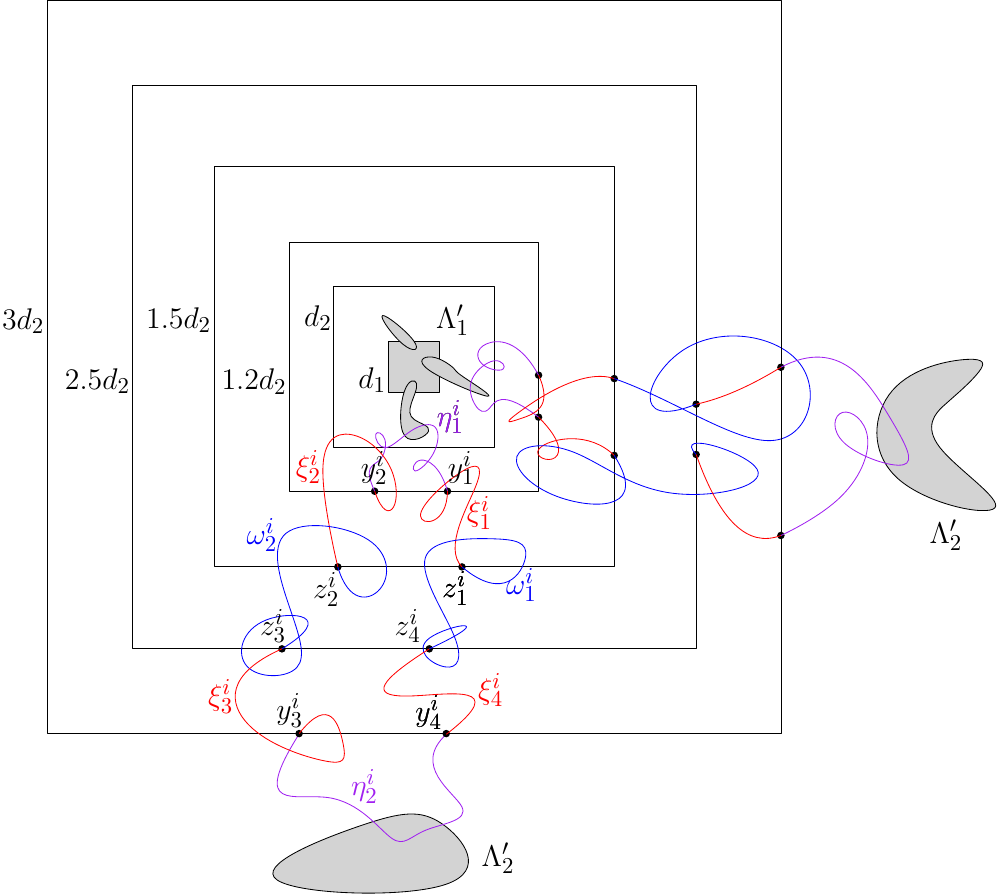}}
	\hspace{0.5cm}
	\subfigure{\includegraphics[width=.47\textwidth]{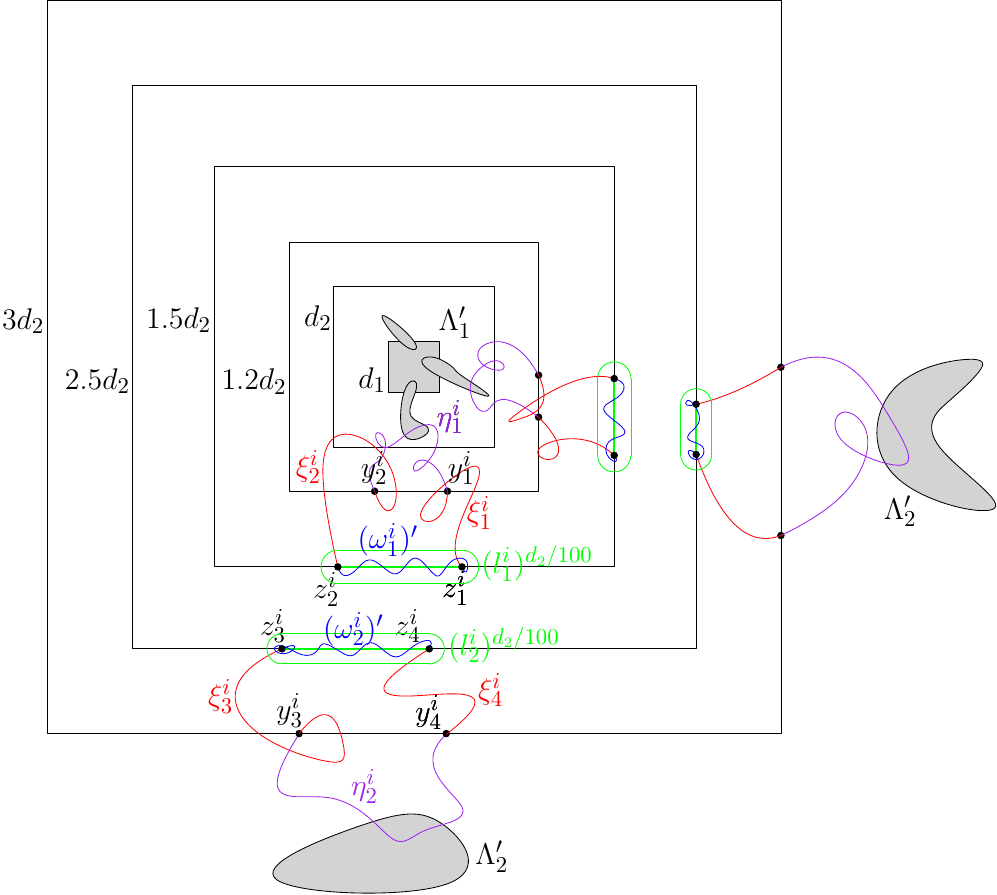}}
	\caption{This figure depicts the proof of Proposition~\ref{lem:quasi}, in the particular case of the subevent $\Ec^{*}_{2,2}$. \emph{Left:} We use the decomposition \eqref{eq:quasi-decom}, where the excursions $\eta^i_1$ and $\eta^i_2$ are drawn in purple, the walks $\xi^i_j$, for $1 \leq j \leq 4$, are in red, and the remaining parts $\omega^i_1$ and $\omega^i_2$ are in blue. The subset $\Lambda'_1$ consists of the inner gray parts, inside $B_{d_2}$, while the outer gray parts represent $\Lambda'_2$. \emph{Right:} We reconnect the endpoints of $\omega^i_1$ and $\omega^i_2$ within the local (green) regions $(l^i_j)^{d_2/100}$, for $j=1,2$, to produce two non-crossing loops $(\gamma^i_1)'$ and $(\gamma^i_2)'$. The corresponding arcs $l^i_j$ are shown in thicker line.} 
	\label{fig:quasi}
\end{figure}

Let $\gamma^1$ and $\gamma^2$ be the two crossing loops. In order to show that \eqref{eq:Ecij} holds for $\Ec_{2,2}^*$, the basic idea is to split each of them into two loops which are not in $\cl_{d_2,1.8d_2}\cup \cl_{2.2d_2,4d_2}$, to produce the occurrence of both $\Ac_{\loc}(d_1,d_2)$ and $\overleftarrow\Ac_{D}(4d_2,d_3)$. Note that these two events are independent by definition. Indeed, they depend on $\Lc_{2 d_2}$ and $\Lc_{2 d_2,D}$, respectively. For this purpose, we explore the configuration from both directions. More precisely, we consider the sampling procedure below (see Figure~\ref{fig:quasi} for an illustration).
\begin{enumerate}[(1)]
	\item Sample the loop soup $\Lc_{1.2d_2}$, and sample the excursions $\eta^1_1$ and $\eta^2_1$ according to $\mu^{\exc}_{B_{1.2d_2}}$, independently, in such a way that the following event holds, with $\Lambda'_1 := \Lambda(B_{d_1},\Lc_{1.2d_2})$:
	\[
	\Kc_1:=\{
	\Lambda'_1\subseteq \mathring{B}_{d_2}, \,\eta^1_1\cap \Lambda'_1\neq\emptyset, \, \eta^2_1\cap \Lambda'_1\neq\emptyset, \text{ and } \Lambda(\eta^1_1,\Lc_{1.2d_2}) \cap \eta^2_1 = \emptyset\}.
	\]
	For $i=1,2$, we denote the starting and ending points of $\eta^i_1$ by $y^i_1$ and $y^i_2$ (resp.).
	
	\item Let $\xi^i_1$ (resp.\ $\xi^i_2$) be the random walk starting from $y^i_1$ (resp.\ $y^i_2$), and stopped upon reaching $\partial B_{1.5d_2}$. Moreover, we further require $\xi^i_1$ to avoid $\Lambda'_1$.
	
	\item Sample the loop soup $\Lc_{D\setminus B_{3d_2}}$, and sample the excursions $\eta^1_2$ and $\eta^2_2$ according to $\mu^{\exc}_{D\setminus \mathring{B}_{3d_2}}$, with endpoints on $\partial B_{3d_2}$, independently, such that the following event holds, with $\Lambda'_2 := \Lambda(D\setminus \mathring{B}_{d_3},\Lc_{D\setminus B_{3d_2}})$:
	\[
	\Kc_2:=\{ \Lambda'_2\subseteq B^c_{4d_2}, \, \eta^1_2\cap \Lambda'_2\neq\emptyset, \, \eta^2_2\cap \Lambda'_2\neq\emptyset, \text{ and } \Lambda(\eta^1_2,\Lc_{D\setminus B_{3d_2}})\cap \eta^2_2=\emptyset \}.
	\]
	Let $y^i_3$ and $y^i_4$ be the starting and ending points (resp.) of $\eta^i_2$, for $i=1,2$.
	
	\item Let $\xi^i_3$ (resp.\ $\xi^i_4$) be the random walk starting from $y^i_3$ (resp.\ $y^i_4$), and stopped upon reaching $\partial B_{2.5d_2}$. Furthermore, we require $\xi^i_3$ to avoid $\Lambda'_2$.
	
	\item Sample the loop soup
	$$\Lc_{\mathrm{int}}:=\Lc_{D}^{B_{3d_2}\setminus B_{1.2d_2}}$$
	(the loops in $\Lc_D$ intersecting $B_{3d_2}\setminus B_{1.2d_2}$), and restrict to the event that
	$$\overleftrightarrow Q:= \overrightarrow Q^{2,2}(1.5d_2) \wedge \overleftarrow Q^{2,2}(2.5d_2) \ind_\Dc>0.$$
	Here,
	\begin{enumerate}[(i)]
	\item $\overrightarrow Q^{2,2}(1.5d_2)$ is the (forward) quality at the scale $1.5d_2$, as in \eqref{eq:Qjk}, induced by the two packets of two random walks $(\xi^1_1,\xi^1_2)$ and $(\xi^2_1,\xi^2_2)$ inside the frontier loop soup $\Lc_{\mathrm{int}}^b$, with initial configuration $(\Lc_{1.2d_2},\eta^1_1,\eta^2_1,\bar y^1,\bar y^2)$, for $\bar y^i=(y^i_1,y^i_2)$ ($i=1,2$), $r=1.2d_2$, and $R=1.5d_2$ (see \eqref{eq:Qjk}),
	
	\item $\overleftarrow Q^{2,2}(2.5d_2)$ is the (reversed) quality at the scale $2.5d_2$, as in \eqref{eq:Qrjk}, induced by the two packets of two random walks $(\xi^1_3,\xi^1_4)$ and $(\xi^2_3,\xi^2_4)$ inside the frontier loop soup $\Lc_{\mathrm{int}}^b$, with initial configuration $(\Lc_{D\setminus B_{3d_2}},\eta^1_2,\eta^2_2,\underline y^1,\underline y^2)$, for $\underline y^i=(y^i_3,y^i_4)$ ($i=1,2$), $r=2.5d_2$, and $R=3d_2$ (see \eqref{eq:Qrjk}),
	
	\item and the event $\Dc$ is given by
$$\Dc:=\{ \text{no loop in $\Lc_{\mathrm{int}}$ intersects $\Lambda'_1\cup\Lambda'_2$, and no cluster in $\Lc_{\mathrm{int}}$ disconnects $B_{1.2d_2}$ from $\partial B_{3d_2}$} \},$$
which is decreasing in $\Lc_{\mathrm{int}}$.
	\end{enumerate}
	
	\item Sample $\omega^i_1$ and $\omega^i_2$ according to $\mu^{z_4^i,z_1^i}$ and $\mu^{z_2^i,z_3^i}$, respectively, where $z_j^i$ is the ending point of $\xi^i_j$, for $j=1, \ldots, 4$, and restrict to the case that each of $\omega^i_1$ and $\omega^i_2$ does not disconnect $B_{1.5d_2}$ from $\partial B_{2.5d_2}$. 
\end{enumerate}
We have thus decomposed the two crossing loops $\gamma^1$ and $\gamma^2$ in $\Ec_{2,2}^*$ as follows (see Figure~\ref{fig:quasi}):
\begin{equation}\label{eq:quasi-decom}
	\gamma^i=\Uc\left(\eta^i_1\oplus \xi^i_2\oplus \omega^i_2\oplus [\xi^i_3]^R \oplus \eta^i_2\oplus \xi^i_4\oplus \omega^i_1 \oplus [\xi^i_1]^R\right), \quad \text{ for each } i=1,2.
\end{equation} 

By following a similar strategy as for the separation and reversed separation results, one can check that separation also holds for the ``two-sided'' quality $\overleftrightarrow Q$, namely,
\begin{align}\notag
	&\ind_{\Kc_1\cap \Kc_2}\,\wt\Pb( \overleftrightarrow Q >0 \mid \Lc_{1.2d_2}, \Lc_{D\setminus B_{3d_2}}, \eta^1_1,\eta^2_1,\eta^1_2,\eta^2_2 ) \\ 
	&\hspace{3cm}\lesssim \ind_{\Kc_1\cap \Kc_2}\,\wt\Pb( \overleftrightarrow Q >1/40 \mid \Lc_{1.2d_2}, \Lc_{D\setminus B_{3d_2}}, \eta^1_1,\eta^2_1,\eta^1_2,\eta^2_2 ), \label{eq:tsQ}
\end{align}
where $\wt\Pb$ is the joint law (product measure) of $\Lc_{\mathrm{int}}$ and $(\xi^i_j)_{i=1,2, 1\le j\le 4}$.

Lemma~\ref{lem:non-disconnecting paths-1} can be adapted here to show that the total mass of $\omega^i_1$ and $\omega^i_2$ is uniformly bounded. Hence, similarly to \eqref{eq:E1-1},
\begin{align}\notag
	\Pb\times\mu_{\cl}\times\mu_{\cl} (\Ec_{2,2}^*) & \lesssim \Pb_{1.2d_2}\times\Pb_{D\setminus B_{3d_2}}\times(\mu^{\exc}_{B_{1.2d_2}})^2\times(\mu^{\exc}_{D\setminus \mathring{B}_{3d_2}})^2 \\ \label{eq:E-decom}
	& \hspace{2.5cm} \Big[\ind_{\Kc_1\cap \Kc_2}\,\wt\Pb( \overleftrightarrow Q >1/40 \mid \Lc_{1.2d_2}, \Lc_{D\setminus B_{3d_2}}, \eta^1_1,\eta^2_1,\eta^1_2,\eta^2_2 )\Big],
\end{align}
where $\Pb_{D\setminus B_{3d_2}}$ denotes the law of $\Lc_{D\setminus B_{3d_2}}$.

Let $l^1_1$ and $l^2_1$ be the arcs along $\partial B_{1.5d_2}$ joining $z^1_1$ and $z^1_2$, such that $l^1_1\cap l^2_1=\emptyset$, and similarly, consider $l^1_2$ and $l^2_2$ along $\partial B_{2.5d_2}$ joining $z^1_3$ and $z^1_4$, with $l^1_2\cap l^2_2=\emptyset$. 
For $1\le i,j\le 2$, sample $(\omega^i_j)'$ according to $\mu^{z^i_2,z^i_1}$ if $j=1$, or $\mu^{z^i_4,z^i_3}$ if $j=2$, such that $(\omega^i_j)'\subseteq (l^i_j)^{d_2/100}$.
Replacing $\omega^i_j$ by $(\omega^i_j)'$ in
\eqref{eq:quasi-decom}, we obtain two disjoint non-crossing loops $(\gamma^i_1)'$ and $(\gamma^i_2)'$, with
\begin{equation}\label{eq:gamma12}
	(\gamma^i_1)' := \Uc\left( \eta^i_1\oplus \xi^i_2\oplus (\omega^i_1)'\oplus [\xi^i_1]^R \right) \quad \text{and} \quad 
	(\gamma^i_2)' := \Uc\left(\eta^i_2\oplus \xi^i_4\oplus (\omega^i_2)' \oplus [\xi^i_3]^R\right).
\end{equation}

Let $G$ be the event that all clusters in $\Lc_{\mathrm{int}}$ have a diameter smaller than $d/1000$. Then, we observe that
\begin{equation}\label{eq:Q-inclusion}
	\Kc_1\cap \Kc_2\cap \{\overleftrightarrow Q >1/40\}\cap G \subseteq \overrightarrow\Ac_{\Lc_{D}\uplus\{ \gamma^1_1,\gamma^2_1 \}}(d_1,d_2) \cap \overleftarrow\Ac_{\Lc_{D}\uplus\{ \gamma^1_2,\gamma^2_2 \}}(4d_2,d_3).
\end{equation}
Let $\overline\Upsilon$ be the collection of all quintuples $(L,\gamma^1_1,\gamma^2_1,\gamma^1_2,\gamma^2_2)$ such that together, they satisfy the event on the right-hand side of \eqref{eq:Q-inclusion}, with $L$ in place of $\Lc_D$, and moreover, $\Lambda(B_{d_1},L)\subseteq \mathring{B}_{d_2}$, $\Lambda(D\setminus \mathring{B}_{d_3},L)\subseteq B_{4d_2}^c$, and $\gamma^1_1$ and $\gamma^2_1$ (resp.\ $\gamma^1_2$ and $\gamma^2_2$) are the only two loops in this collection that cross $A_{d_2,1.5d_2}$ (resp.\  $A_{2.5d_2,4d_2}$), and they have multiplicity $1$.   
Hence, by \eqref{eq:Q-inclusion},
\begin{align}\notag
	\Pb\times \mu_0^4 (\overline\Upsilon) & \gtrsim \Pb_{1.2d_2}\times\Pb_{D\setminus B_{3d_2}}\times(\mu^{\exc}_{B_{1.2d_2}})^2\times(\mu^{\exc}_{D\setminus \mathring{B}_{3d_2}})^2 \\ \label{eq:mu04}
	& \hspace{2.5cm} \Big[\ind_{\Kc_1\cap \Kc_2}\,\wt\Pb( \overleftrightarrow Q >1/40, \, G \mid \Lc_{1.2d_2}, \Lc_{D\setminus B_{3d_2}}, \eta^1_1,\eta^2_1,\eta^1_2,\eta^2_2 )\Big].
\end{align}
This, combined with \eqref{eq:E-decom}, and the FKG inequality for the decreasing events $\{\overleftrightarrow Q >1/40\}$ and $G$, shows that 
\begin{equation}\label{eq:E-U}
	\Pb\times\mu_{\cl}\times\mu_{\cl} (\Ec_{2,2}^*) \lesssim \Pb\times \mu_0^4 (\overline\Upsilon).
\end{equation}
Using Palm's formula again, we have
\begin{equation}\label{eq:Aarrows}
	\Pb( \overrightarrow\Ac_{D}(d_1,d_2) \cap \overleftarrow\Ac_{D}(4d_2,d_3) ) \ge \frac{1}{4!} \alpha^4\,  \Pb\times \mu_0^4 (\overline\Upsilon).
\end{equation}
Combining \eqref{eq:Aarrows}, \eqref{eq:E-U} and \eqref{eq:E22}, we obtain that 
\begin{equation}\label{eq:Aarrows2}
		\Pb( \Ec_{2,2}^* ) \lesssim \alpha^{-2}\, \Pb( \overrightarrow\Ac_{D}(d_1,d_2)\cap \overleftarrow\Ac_{D}(4d_2,d_3) ). 
\end{equation}
Noting that $\overrightarrow\Ac_{D}(d_1,d_2)\subseteq \Ac_{\loc}(d_1,d_2)$ and $\overleftarrow\Ac_{D}(4d_2,d_3)\subseteq \Ac_{D}(4d_2,d_3)$, and that $\Ac_{\loc}(d_1,d_2)$ and $\overleftarrow\Ac_{D}(4d_2,d_3)$ are independent by definition, we deduce from \eqref{eq:Aarrows2} (and \eqref{eq:E22c}) that 
\[
\Pb( \Ec^{(c)}_{2,2} ) \lesssim
\Pb( \Ec_{2,2}^* ) \lesssim \alpha^{-2}\, \Pb( \Ac_{\loc}(d_1,d_2) )\, \Pb(\Ac_{D}(4d_2,d_3) ).
\]

Therefore, we have proved the desired result \eqref{eq:Ecij}, for the particular subevent $\Ec^{(c)}_{2,2}$ of $\Ec_{2,2}$ in the left-hand side. As we mentioned in the beginning, the other cases can be handled in a similar way. We hope that after this detailed proof for $\Ec^{(c)}_{2,2}$, the reader is now convinced that in each of the subcases, a well-suited surgery argument, adapted to the particular configuration of loops, can be employed. This completes the proof of the proposition.
\end{proof}

\subsection{Consequence: four-arm exponent in the RWLS} \label{sec:qm_four_arm}

Finally, as an application of the quasi-multiplicativity property, we relate the discrete four-arm event to the corresponding event in the continuum. This enables us to upper bound the four-arm probabilities in terms of the arm exponents computed in \cite{GNQ2024c} (see Section~\ref{sec:exp_cle}).

\begin{theorem} \label{prop:armp}
For any $\alpha\in(0,1/2]$ and $\eps>0$, there exists a constant $c_2(\alpha,\eps)>0$ such that for all $1 \leq d_1<d_2$ and $D \supseteq B_{2d_2}$,
	\begin{equation}\label{eq:armp}
	\Pb( \Ac_{D}(d_1,d_2) ) \le c_2\, (d_2/d_1)^{-\xi(\alpha)+\eps}.
	\end{equation}
\end{theorem}

\begin{proof}
First, we can choose a large enough constant $c_2$ such that \eqref{eq:armp} holds for all $d_2/2 \le d_1<d_2$. Thus, we only need to consider the case $d_1\le d_2/2$ below. Furthermore, by the locality result in Proposition~\ref{prop:locality}, it suffices to prove the upper bound for the local arm event $\Ac_{\loc}(d_1,d_2)$, when $d_1\le d_2/2$.

By Theorem~\ref{thm:convergence} and Definition~\ref{def:arm_bls}, for all $R >1$,
\begin{equation} \label{eq:lim_armp1}
\lim_{n\rightarrow\infty}\Pb(\Ac_{\loc}(n,nR))=\Pb( \wt\Ac_{\alpha}(1/(2R), 1/2) ).
\end{equation}
	Consider an arbitrary $\eps>0$. It follows from Proposition~\ref{prop:arm-exponent} that for some $R_0(\alpha, \eps)$,
	\begin{equation} \label{eq:lim_armp2}
\Pb( \wt\Ac_{\alpha}(1/(2R), 1/2) ) \leq R^{-\xi(\alpha)+\eps/2} \quad \text{for all $R \geq R_0$}.
\end{equation}
	In the remainder of the proof, we write $\xi := \xi(\alpha)$, forgetting about the (key) dependence on $\alpha$ for the ease of notation. It follows from \eqref{eq:lim_armp1} and \eqref{eq:lim_armp2} that we can choose an integer $R(\alpha, \eps)$ large enough, and an associated $m(\alpha, \eps)$, such that 
	\begin{equation}\label{eq:c1}
	 \text{for all } j\ge m, \quad c_1 \, \Pb(\Ac_{\loc}((4R)^j,4^jR^{j+1}))\le (4R)^{-\xi+\eps},
	\end{equation}
	where $c_1$ is the constant given in Proposition~\ref{lem:quasi}.
	Let $m_1$ and $m_2$ be the integers such that, respectively,
	\begin{equation} \label{eq:def_m1_m2}	
	(4R)^{m_1-1}< d_1\le (4R)^{m_1} \quad \text{and} \quad (4R)^{m_2}< d_2 \le (4R)^{m_2+1}.
\end{equation}

	If we first assume that $m_1\ge m$, we have
\begin{align*}
	\Pb( \Ac_{\loc}(d_1,d_2) ) 
	&\le \Pb( \Ac_{\loc}((4R)^{m_1},d_2) ) \\
	&\le c_1^{m_2-m_1}\, \bigg( \prod_{j=m_1}^{m_2-1} \Pb( \Ac_{\loc}((4R)^j,4^jR^{j+1}) ) \bigg) \, \Pb( \Ac_{\loc}((4R)^{m_2},d_2) ) \\
	&\le \prod_{j=m_1}^{m_2-1} (4R)^{-\xi+\eps} ,
\end{align*}
where we used the fact $\Ac_{\loc}(d_1,d_2)\subseteq \Ac_{\loc}((4R)^{m_1},d_2)$ (see Remark~\ref{rmk:mono}) in the first inequality, Proposition~\ref{lem:quasi} in the second inequality (applied repeatedly, $(m_2-m_1)$ times), and \eqref{eq:c1} in the third inequality. Hence, from \eqref{eq:def_m1_m2},
$$\Pb( \Ac_{\loc}(d_1,d_2) ) \le c_2\, (d_2/d_1)^{-\xi+\eps},$$
for some constant $c_2(\alpha,\eps)$.

If $m_1< m$ and $m_2> m$, then a similar reasoning as above gives, using $d_1 \leq (4R)^{m_1} < (4R)^m$ (from \eqref{eq:def_m1_m2}),
\[
\Pb( \Ac_{\loc}(d_1,d_2) ) \le \Pb(\Ac_{\loc}((4R)^{m},d_2)) \le \prod_{j=m}^{m_2-1} (4R)^{-\xi+\eps}.
\]
We deduce that for some $c'_2(\alpha,\eps)>0$,
$$\Pb( \Ac_{\loc}(d_1,d_2) ) \le c'_2\, (d_2/d_1)^{-\xi+\eps}.$$

Finally, if $m_1< m$ and $m_2\le m$, it is easy to see that \eqref{eq:armp} holds uniformly in this regime for some large constant $c_2$, e.g., $c_2=(4R)^{(m+1)\xi}$. This completes the proof.
\end{proof}

\subsection{Other two-arm and four-arm exponents}\label{subsec:other_arm2}

In this section, we summarize the corresponding results for the two-arm events and the boundary four-arm events, defined in Section~\ref{subsec:other_arm}. Since the proof strategy is very similar to that of Proposition~\ref{lem:quasi}, we omit the proofs for the sake of brevity.
\begin{proposition}[Quasi-multiplicativity]\label{prop:2n-quasi}
	For any $\alpha\in (0,\half]$, there exists a constant $c_3(\alpha)>0$ such that for all $1 \leq d_1\le d_2/2\le d_3/16$ and $D \supseteq B_{2d_3}$,
	\begin{align}\label{eq:2n-quasi-1}
		\Pb( \Ac^{\boldsymbol{\cdot}}_{D}(d_1,d_3) ) \le c_3\, \Pb( \Ac^{\boldsymbol{\cdot}}_{\loc}(d_1,d_2) )\, \Pb( \Ac^{\boldsymbol{\cdot}}_{D}(4d_2,d_3) ),
	\end{align}
where the superscript $\boldsymbol{\cdot}$ can be either ``$2$'', ``$2,+$'' or ``$+$'', simultaneously.
\end{proposition}

We recall Proposition~\ref{prop:arm-exponent} and Proposition~\ref{prop:b-arm-exponent}, which provide the interior two-arm exponent $\xi^2$, the boundary two-arm exponent $\xi^{2,+}$ and the boundary four-arm exponent $\xi^+$ in the BLS.
Combined with the above quasi-multiplicativity, following the same proof strategy as that of Theorem~\ref{prop:armp}, we obtain the following upper bound on the arm probabilities.

\begin{theorem} \label{prop:2n-armp}
	For any $\alpha\in (0,\half]$ and $\eps>0$, there exists a constant $c_4(\alpha,\eps)>0$ such that for all $1 \leq d_1<d_2$ and $D \supseteq B_{2d_2}$,
	\begin{align}\label{eq:2n-armp}
		\Pb( \Ac^{\boldsymbol{\cdot}}_{D}(d_1,d_2) ) \le c_4\, (d_2/d_1)^{-\xi^{\boldsymbol{\cdot}}(\alpha)+\eps},
	\end{align}
where the superscript $\boldsymbol{\cdot}$ can be either ``$\,2$'', ``$\,2,+$'' or ``$\,+$'', simultaneously on both sides.
\end{theorem}
Note that we can first obtain upper bounds for the corresponding local arm events as in Theorem~\ref{prop:armp}, and then use the locality result in Proposition~\ref{prop:2n-locality} to obtain upper bounds for the arm events in a general set $D$.

\subsection*{Acknowledgments}

YG and PN are partially supported by a GRF grant from the Research Grants Council of the Hong Kong SAR (project CityU11307320). WQ is partially supported by a GRF grant from the Research Grants Council of the Hong Kong SAR (project CityU11308624).

\bibliographystyle{abbrv}
\bibliography{percolation_GFF1}

\end{document}